\documentclass[11pt]{article}
\usepackage{amssymb, amsmath, amsthm, mathrsfs, ascmac, color}
\usepackage[pdftex]{graphicx} 
\usepackage[subrefformat=parens]{subcaption}
\usepackage[top=3cm, bottom=3cm, left=3.25cm,right=3.25cm]{geometry}
\usepackage{comment}
\usepackage{appendix}
\theoremstyle{plain}
\newtheorem{thm}{Theorem}[section]

\newtheorem{prop}[thm]{Proposition}
\newtheorem{lem}{Lemma}[section]

\newtheorem{rmk}{Remark}
\makeatletter

\@addtoreset{equation}{section}
\makeatother
\newcommand{\ind}{\bs 1}
\newcommand{\dd}{\mathrm d}
\newcommand{\pd}{\partial}
\newcommand{\ee}{\mathrm e}
\newcommand{\EE}{\mathsf E}
\newcommand{\OO}{\mathsf{O}}
\newcommand{\oo}{\mathsf{o}}
\newcommand{\Op}{\mathsf{O}_p}
\newcommand{\op}{\mathsf{o}_p}
\newcommand{\PP}{\mathsf P}
\newcommand{\VV}{\mathsf V}
\newcommand{\cov}{\mathsf{Cov}}

\newcommand{\FF}{\mathscr F}
\newcommand{\GG}{\mathscr G}

\newcommand{\pto}{\stackrel{p}{\to}}
\newcommand{\dto}{\stackrel{d}{\to}}
\newcommand{\wto}{\stackrel{w}{\to}}
\newcommand{\TT}{\mathsf T}
\newcommand{\bs}{\boldsymbol}

\newcommand{\tand}{\widetilde\land}

\allowdisplaybreaks

\title{\textbf{Volatility change point detection for linear parabolic SPDEs}}
\date{}

\author{\textbf{Yozo Tonaki}\thanks{Graduate School of Engineering Science, 
The University of Osaka, Toyonaka, Japan}
\thanks{Center for Mathematical Modeling and Data Science (MMDS), 
The University of Osaka, Toyonaka, Japan} 
\and \textbf{Yusuke Kaino}\thanks{Graduate School of Maritime Sciences, 
Kobe University, Kobe, Japan}
\and \textbf{Masayuki Uchida}$^{* \dag}$\thanks{CREST, 
Japan Science and Technology Agency, Kawaguchi, Japan}}
\begin{document}
\maketitle
\begin{abstract}
We consider change point detection for the volatility  
in second order linear parabolic stochastic partial differential equations
based on high frequency spatio-temporal data.
We give a test statistic to detect changes in the volatility 
based on change point analysis for diffusion processes
and derive the asymptotic null distribution of the test statistic.
We also show that 
the test is consistent.
Moreover, we provide some examples and then perform numerical simulations of the proposed test statistic.

\begin{center}
\textbf{Keywords and phrases}
\end{center}
Change point detection, 
high frequency spatio-temporal data,
linear parabolic stochastic partial differential equations,
stochastic differential equations,
volatility function.
\end{abstract}

\section{Introduction}
The change point problem was originally formulated 
by Page \cite{Page1954, Page1955} for quality control.
Today, it has applications in various fields such as 
economics \cite{Manner_etal2024, Song2020},
environmental science \cite{Cahill_etal2015, Gallagher_etal2013},
and medicine \cite{Otto_Schmid2016, Yang_etal2006} and is becoming increasingly important. 
Recently, stochastic partial differential equations (SPDEs) have become popular in many fields as a modeling tool for spatio-temporal phenomena 
(see Jones and Zhang \cite{Jones_Zhang1997}, 
Piterbarg and Ostrovskii \cite{Piterbarg_Ostrovskii1997}, 
Mohapl \cite{Mohapl2000},
North et al.\cite{North_etal2011},
Tuckwell \cite{Tuckwell2013}
and Altmeter et al.\cite{Altmeyer_etal2022}), 
and change point analysis is also essential for spatio-temporal phenomena.

We study change point analysis for the linear parabolic SPDE 
\begin{equation}\label{spde}
\dd X_t(y) = -A_\theta X_t(y) \dd t +\sigma(t) \dd W_t^Q(y), 
\quad(t, y) \in [0,1] \times D
\end{equation}
with an initial value $X_0$ 
and the Dirichlet boundary condition $X_t(y) = 0$, $(t,y) \in [0,1] \times \pd D$,
where $D =(0,1)^d$, $d \in \mathbb N$, 
\begin{equation*}
-A_\theta = \theta_2 \triangle + \theta_1 \cdot \nabla +\theta_0,
\end{equation*}
$\theta_0 \in \mathbb R$, $\theta_1 \in \mathbb R^d$,  
$\theta_2 \in (0,\infty)$ are unknown parameters,
$[0,1] \ni t \mapsto \sigma(t) \in (0,\infty)$ 
is possibly a time-dependent volatility function, 
$\{ W_t^Q \}_{t \ge 0}$ is a $Q$-Wiener process in a Sobolev space on $D$ 
(see \eqref{QW} below), 
and the initial value $X_0$ is an $L^2(D)$-valued random variable and 
independent of $\{ W_t^Q \}_{t \ge 0}$.

Change point analysis consists of change point detection, change point estimation 
and change point model selection, and has been studied in various models
by many researchers. See, for instance, 
Bai \cite{Bai1997}, 
Cs\"{o}rg\"{o} and Horv\'{a}th \cite{Csorgo_Horvath1997}, 
Bai and Perron \cite{Bai_Perron1998}, Koul et al.\cite{Koul_etal2003}, 
Lee et al.\cite{Lee_etal2003},
Ninomiya \cite{Ninomiya2005, Ninomiya2015},
Zou el al.\cite{Zou_etal2020}, 
and Horv\'{a}th and Rice \cite{Horvath_Rice2024}.
De Gregorio and Iacus \cite{DeGregorio_Iacus2008}
and Iacus and Yoshida \cite{Iacus_Yoshida2012} 
studied the volatility change point estimation 
in stochastic differential equations 
with a fixed time interval based on high frequency data.
Song and Lee \cite{Song_Lee2009},
Lee \cite{Lee2011}, 
Negri and Nishiyama \cite{Negri_Nishiyama2017},
Tonaki et al.\cite{TKU2022, TKU2023a}
and Tonaki and Uchida \cite{TU2023}
considered the change point problem 
for the diffusion or drift parameter in ergodic diffusion processes 
based on high frequency data.
In particular, for the SPDE model, 
Reiss et al.\cite{Reiss_etal2023arXiv} 
and Tiepner and Trottner \cite{Tiepner_Trottner2024arXiv}
considered the stochastic heat equation
with a space-dependent diffusivity, 
and estimated the change hypersurface of the diffusivity 
based on local measurement observations.

We consider change detection for the volatility function $\sigma(t)$ 
in SPDE \eqref{spde} based on high frequency spatio-temporal data.
The coordinate process defined by the inner product of 
the random field $X_t$ and the eigenfunction $e_l$ given in \eqref{eigenfunc} below
is an Ornstein-Uhlenbeck process, and an approximate coordinate process 
can be constructed through statistics for SPDEs 
based on high frequency spatio-temporal data.
For parametric estimation for parabolic SPDEs based on discrete observations, see 
Markussen \cite{Markussen2003}, 
Chong \cite{Chong2020}, 
Bibinger and Trabs \cite{Bibinger_Trabs2020}, 
Cialenco and Huang \cite{Cialenco_Huang2020}, 
Kaino and Uchida \cite{Kaino_Uchida2021a, Kaino_Uchida2021b},
Hildebrandt and Trabs \cite{Hildebrandt_Trabs2021}, 
Bibinger and Bossert \cite{Bibinger_Bossert2023}, 
Gamain and Tudor \cite{Gamain_Tudor2023}, 
Tonaki et al.\cite{TKU2023b}--\cite{TKU2026} 
and Bossert \cite{Bossert2024}.
In this paper, we thus propose a change point test 
for the volatility function in liner parabolic SPDEs 
based on the change point analysis for diffusion processes.
The aim of this paper is to show that 
the asymptotic null distribution of the proposed test statistic 
for the volatility change is the distribution of the supremum of a Brownian bridge, 
and that the test statistic has consistency of test.

This paper is organized as follows.
In Section \ref{sec2}, we state main results. 
We propose a test statistic to detect changes of the volatility function, 
and derive the asymptotic null distribution of the test statistic.
We then show that the test statistic has consistency of test.
Section \ref{sec3} gives examples of the procedure to perform 
volatility change point detection for SPDE \eqref{spde} with $d = 1, 2$.
In Section \ref{sec4}, we provide simulation results of change point detection 
for the volatility function in SPDE \eqref{spde} with $d = 1$. 
Section \ref{sec5} is devoted to the proofs of the results in Section \ref{sec2}.
In Appendices \ref{secA} and \ref{secB}, 
we treat parametric estimation for linear parabolic SPDEs 
with a time-dependent volatility for $d = 1, 2$.

\section{Main results}\label{sec2}
The eigenpairs $\{ \lambda_l, e_l \}_{l \in \mathbb N^d}$ of the operator $A_\theta$ 
are given by
\begin{align}
e_l(y) = e_l(y;\kappa) &= 2^{d/2} \exp\biggl( {-\frac{1}{2} \kappa^\TT y} \biggr) 
\prod_{k=1}^d \sin(\pi l_k y^{(k)}),
\label{eigenfunc}
\\
\lambda_l &= \theta_2 \pi^2 |l|_{2}^2 +\frac{|\theta_1|_{2}^2}{4 \theta_2} -\theta_0
\nonumber
\end{align}
for $ l = (l_1,\ldots,l_d) \in \mathbb N^d$ 
and $y = (y^{(1)},\ldots, y^{(d)})^\TT \in \overline D$,
where $\kappa = \theta_1/\theta_2 \in \mathbb R^d$, $|\cdot|_{2}$ denotes the Euclidean norm, and 
$\TT$ denotes the transpose. 
The eigenfunctions $\{ e_l \}_{l \in \mathbb N^d}$ are orthonormal with respect to 
the weighted $L^2$-inner product
\begin{equation*}
\langle u, v \rangle
= \int_D u(y)v(y)\exp(\kappa^\TT y) \dd y, 
\quad 
\| u \| = \sqrt{\langle u, u \rangle}
\end{equation*}
for $u, v \in L^2(D)$.
We assume $\lambda_{1_d} > 0$ 
so that $A_\theta$ is a positive definite and self-adjoint operator, 
where $1_d = (1,\ldots,1)^\TT \in \mathbb R^d$.

The $Q$-Wiener process in SPDE \eqref{spde} is given by
\begin{equation}\label{QW}
W_t^Q = \sum_{l \in \mathbb N^d} \gamma_l^{-\alpha/2} w_l(t) e_l,
\quad 
t \ge 0,
\end{equation}
where $\{ w_l \}_{l \in \mathbb N^d}$ are independent 
$\mathbb R$-valued standard Brownian motions, 
$\{ \gamma_l \}_{l \in \mathbb N^d}$ is a $(0,\infty)$-valued sequence
such that $c_1 |l|_2^2 \le \gamma_l \le c_2 |l|_2^2$ 
for some universal constants $c_1,c_2 \in (0, \infty)$ and all $l \in \mathbb N^d$,
and $\alpha \in [0,\infty) \cap (d/2-1, \infty)$. 
$\gamma_l$ and $\alpha$ 
may be  unknown parameters.
Note that the $Q$-Wiener process $\{W_t^Q\}_{t \ge 0}$ given in \eqref{QW}
is well-defined in a Hilbert space larger than $L^2(D)$ 
and we have $\sup_{t \in [0,1]} \EE[\| X_t \|^2] < \infty$.
For details, see \cite{DaPrato_Zabczyk2014}, \cite{TKU2023b} and \cite{Bossert2024}.

There exists a unique mild solution of SPDE \eqref{spde}, which is given by
\begin{equation*}
X_t=\ee^{-t A_\theta}X_0 +\int_0^t \sigma(s) \ee^{-(t-s)A_\theta}\dd W_s^{Q}
\quad \mathrm{a.s.}
\end{equation*}
for any $t \ge 0$, where $\ee^{-t A_\theta} u 
= \sum_{l \in \mathbb N^d} \ee^{-\lambda_l t}
\langle u, e_l \rangle e_l$ for $u \in L^2(D)$.
The random field $X_t(y)$ is then decomposed as follows.
\begin{equation*}
X_t(y)=\sum_{l \in \mathbb N^d} x_l(t)e_l(y),
\quad t \ge 0, \ y \in \overline D
\end{equation*}
with
\begin{equation}\label{cor-pro}
x_l(t) = \langle X_t, e_l \rangle =
\ee^{-\lambda_l t} \langle X_0, e_l \rangle 
+\gamma_l^{-\alpha/2} \int_0^t \sigma(s) \ee^{-\lambda_l(t-s)} \dd w_l(s).
\end{equation}
The coordinate process $x_l(t)$ satisfies the Ornstein-Uhlenbeck dynamics
\begin{equation}\label{OU-dyna}
\left\{
\begin{split}
\dd x_l(t) &= -\lambda_l x_l(t)\dd t +\sigma(t) \gamma_l^{-\alpha/2} \dd w_l(t),
\\
x_l(0) &= \langle X_0, e_l \rangle.
\end{split}
\right.
\end{equation}

Let $N \in \mathbb N$ and $M = (M_1,\ldots, M_d) \in \mathbb N^d$. 
We suppose that a mild solution of SPDE \eqref{spde} 
is discretely observed on the grid 
$(t_i, y_j) = (t_i, y_{j_1}^{(1)}, \ldots, y_{j_d}^{(d)}) \in [0,1] \times [0,1]^d$
with
\begin{equation*}
t_i = \frac{i}{N}, 
\quad
y_{j_k}^{(k)} = \frac{j_k}{M_k}
\end{equation*}
for $i \in \{ 0, 1, \ldots, N \}$, $j_k \in \{ 0, 1, \ldots, M_k \}$ and $k \in \{ 1, \ldots, d \}$.

For $n \in \{ 1, 2, \ldots, N \}$, we define
\begin{equation*}
t_i^n = i \Delta_n = i \cdot \frac{1}{N} \biggl\lfloor \frac{N}{n} \biggr\rfloor,
\quad 
i \in \{ 0, \ldots, n \}
\end{equation*}
and for a process $\{ Z_t \}_{t \in [0,1]}$,
we set $\Delta_i^n Z = Z_{t_i^n} -Z_{t_{i-1}^n}$.

For $h \in (0,1)$, $a, b, c \in (0, \infty)$ and $d \in \mathbb R$, we define
\begin{equation*}
h^{a \tand b} =  
\begin{cases}
h^{a}, & a < b,
\\
-h^{b} \log (h), & a = b,
\\
h^{b}, & a > b,
\end{cases}
\end{equation*}
and $h^{d+ c(a \tand b)} = h^d \cdot (h^c)^{a \tand b}$.
Furthermore, for $L \in (1,\infty)$, we write
\begin{equation*}
L^{a \tand b} = \frac{1}{(1/L)^{a \tand b}} =
\begin{cases}
L^{a}, & a < b,
\\
L^{b} /\log (L), & a = b,
\\
L^{b}, & a > b.
\end{cases}
\end{equation*}

\subsection{Volatility change point detection}
In this subsection, we consider the following hypothesis testing problem
in order to investigate whether the volatility function $\sigma(t)$ in SPDE \eqref{spde} 
changes over time.
\begin{equation}
\left \{
\begin{gathered}
H_0: \text{there exists } \sigma^* \in (0,\infty) \text{ such that }
\sigma(t) = \sigma^* \text{ over } t \in [0,1]
\\
\text{vs.}
\\
H_1: \text{there exist } r \in \mathbb N \text{ and }
0 = \tau_0 < \tau_1 < \tau_2 < \cdots < \tau_r < \tau_{r+1} = 1
\text{ such that}
\\
\sigma(t) = \sum_{j = 1}^{r+1} \sigma_j^{\dag} \ind_{[\tau_{j-1}, \tau_{j})}(t),
\end{gathered}
\label{HTP}
\right.
\end{equation}
where $\ind_A$ stands for the indicator function of $A$,  
$\sigma_1^{\dag}, \ldots, \sigma_{r+1}^\dag \in (0,\infty)$
and $\sigma_j^{\dag} \neq \sigma_{j+1}^{\dag}$ for $j \in \{ 1,\ldots, r \}$. 
For convenience, we write $[\tau_r, \tau_{r+1}) = [\tau_r,1]$.
Note that $r$, $\tau_1,\ldots, \tau_r$, 
$\sigma_1^\dag, \ldots, \sigma_{r+1}^\dag$ and $\sigma^*$ are unknown.
We suppose that there is no change in the parameters 
except for the volatility $\sigma(t)$ over $t \in [0,1]$.

Let  $M_{(1)} = \min \{ M_1,\ldots, M_d \}$.
We assume that the true value $\kappa^*$ of $\kappa$ belongs to the interior of
a compact convex subset of $\mathbb R^d$ and the following conditions hold.
\begin{description}
\item[\textbf{[A1]}]
The initial value $X_0$ of SPDE \eqref{spde} satisfies the following (i)--(iii). 
\begin{enumerate}
\item[(i)]
For $\alpha \in [0,\infty) \cap (d/2-1, \infty)$, 
either of the following conditions holds.
\begin{enumerate}
\item[a)]
$\EE[\langle X_0, e_l \rangle] = 0$ for all $l \in \mathbb N^d$ and
$\sup_{l \in \mathbb N^d} \lambda_l^{1+\alpha} \EE[\langle X_0, e_l \rangle^2] < \infty$.

\item[b)]
$\EE[\| A_\theta^{(1+\alpha)/2} X_0 \|^2 ] < \infty$.

\end{enumerate}

\item[(ii)]
$\{ \langle X_0, e_l \rangle \}_{l \in \mathbb N^d}$ are independent.

\item[(iii)]
There exists $\ell \in \mathbb N^d$ such that 
$\EE[\langle X_0, e_{\ell} \rangle^4] < \infty$.
\end{enumerate}

\item[\textbf{[A2]}]
There exist a positive sequence $\{ R_{M,N} \}$ and an estimator $\widehat \kappa$
such that $R = R_{M,N} \to \infty$ and 
\begin{equation*}
R (\widehat \kappa -\kappa^*) = \Op(1)
\end{equation*}
as $N \to \infty$ and 
$M_{(1)} \to \infty$.

\item[\textbf{[A3]}]
For a positive sequence $R = R_{M,N}$ obtained from [A2], 
\begin{equation*}
\frac{n^2 \Delta_n^{(1+\alpha-d/2) \tand 1}}{R^2} \to 0,
\quad
\frac{n^{3/2}}{M_{(1)}^{(2(1+\alpha)-d) \tand 2}} \to 0
\end{equation*}
as $n \to \infty$, $M_{(1)} \to \infty$ and $R \to \infty$.
\end{description}

Let $\beta_{\ell}^* \in (0,\infty)$ be the true value of 
$\beta_{\ell}(t) = \sigma(t) \gamma_{\ell}^{-\alpha/2}$ under $H_0$.
We additionally assume the following condition.

\begin{description}
\item[\textbf{[B]}]
For $\ell \in \mathbb N^d$ obtained from [A1]-(iii), there exists an estimator 
$\widehat \beta_{\ell}^2$ such that 
$\widehat \beta_{\ell}^2 \pto (\beta_{\ell}^*)^2$ under $H_0$.
\end{description}

\begin{rmk}[]\label{rmk1}
\begin{enumerate}
\item[(i)]
[A1] is a standard regular condition in the statistical inference 
for linear parabolic SPDEs based on high frequency spatio-temporal data. 
See \cite{Bibinger_Trabs2020}, \cite{TKU2023b} and \cite{Bossert2024}.

\item[(ii)]
Although it is not stated explicitly in [A2], 
we need an estimator of the damping parameter $\alpha$ 
in order to estimate $\kappa$ if $\alpha$ is unknown. 
Furthermore, it is possible to construct the estimator $\widehat \kappa$ 
which satisfies [A2] under either $H_0$ or $H_1$. 
See \cite{Bibinger_Bossert2023}, \cite{Hildebrandt_Trabs2021}, 
\cite{Bossert2024}, \cite{TKU2025arXiv1}, Appendices \ref{secA} and \ref{secB}
for the construction of the estimator $\widehat \kappa$ under $H_0$ or $H_1$.

\item[(iii)]
[A3] is a sufficient condition for the balance of $n$, $N$ and $M$ 
to control the approximate coordinate process \eqref{approx_cp} below.
See Proposition \ref{prop0} below for details.

\item[(iv)]
We suppose that [A1]--[A3] hold, 
and define $\widehat \beta_{\ell}^2 
= \sum_{i=1}^n (\Delta_i^n \widehat x_{\ell})^2$
for $\widehat x_{\ell}$ given by \eqref{approx_cp} below. 
Since it follows that under $H_0$,
\begin{equation*}
\sum_{i=1}^n (\Delta_i^n x_{\ell})^2 \pto (\beta_{\ell}^*)^2,
\end{equation*}
we see from Proposition \ref{prop0} below that 
$\widehat \beta_{\ell}^2 \pto (\beta_{\ell}^*)^2$ 
under $H_0$.
Hence, we can construct the estimator $\widehat \beta_{\ell}^2$ 
which satisfies [B].

\item[(v)]
Since we assume no change in the parameters except for the volatility $\sigma(t)$, 
a change in $\beta_l(t)$ implies a corresponding change in the volatility $\sigma(t)$.
\end{enumerate}
\end{rmk}

The CUSUM test statistic
\begin{equation*}
\frac{1}{\widehat \beta_{\ell}^2} \sqrt{\frac{n}{2}} 
\max_{1 \le k \le n}
\biggl| \sum_{i=1}^k (\Delta_i^n x_{\ell})^2 
-\frac{k}{n} \sum_{i=1}^n (\Delta_i^n x_{\ell})^2 \biggr|
\end{equation*}
is useful as a test statistic to detect changes 
for the volatility $\beta_{\ell} (t) = \sigma(t) \gamma_{\ell}^{-\alpha/2}$
in the discretely observed Ornstein-Uhlenbeck dynamics $x_\ell(t)$ 
given in \eqref{OU-dyna}.
However, we cannot directly use this test statistic 
since the coordinate process $x_l(t)$ is represented as 
\begin{equation*}
x_l(t) = \langle X_t, e_l \rangle 
= \int_D X_t(y) e_l(y; \kappa) \exp(\kappa^\TT y) \dd y,
\end{equation*}
$\kappa$ is unknown, 
and continuous spatio-temporal data are not observable.
We therefore construct an approximate process of the coordinate process $x_l(t)$ 
using the estimator $\widehat \kappa$ in the assumption [A2]
and discrete spatio-temporal data.

Let $\mathbb M_d = \{ 1,\ldots, M_1\} \times \cdots \times \{ 1,\ldots, M_d\}$
and $D_j = [y_{j_1-1}^{(1)}, y_{j_1}^{(1)}] \times \cdots 
\times [y_{j_d-1}^{(d)}, y_{j_d}^{(d)}]$ for $j = (j_1, \ldots, j_d) \in \mathbb M_d$. 
For $t \in [0,1]$ and $l \in \mathbb N^d$, we define
\begin{align}
\widehat x_l(t) 
&= \sum_{j \in \mathbb M_d} \int_{D_j} X_t(y_{j}) 
e_l(y;\widehat \kappa) \exp(\widehat \kappa^\TT y)\dd y
\nonumber
\\
&= \sum_{j_1 = 1}^{M_1} \cdots \sum_{j_d = 1}^{M_d} 
X_t(y_{j_1}^{(1)},\ldots,y_{j_d}^{(d)})
\delta_{j_1}^{(1)} g_{l_1}(\widehat \kappa_1) 
\cdots \delta_{j_d}^{(d)} g_{l_d}(\widehat \kappa_d),
\label{approx_cp}
\end{align}
where $\widehat \kappa = (\widehat \kappa_1,\ldots, \widehat \kappa_d)^\TT$,
\begin{align}
g_p(x:a) &= 
\frac{\sqrt{2}\ee^{ax/2}}{(a/2)^2+(\pi p)^2}
\biggl( \frac{a}{2} \sin(\pi p x) - \pi p \cos(\pi p x) \biggr),
\label{gl}
\\
\delta_{j_k}^{(k)} g_{p}(a) &= g_{p}(y_{j_k}^{(k)}:a) - g_{p}(y_{j_k-1}^{(k)}:a)
\nonumber
\end{align}
for $a, x \in \mathbb R$, $p \in \mathbb N$, $j_k \in \{ 1,\ldots, M_k \}$ and $k \in \{ 1,\ldots, d \}$.

We define the test statistic for the change of the volatility function $\sigma(t)$
as follows.
\begin{equation*}
T_n = \frac{1}{\widehat \beta_{\ell}^2} \sqrt{\frac{n}{2}} 
\max_{1 \le k \le n}
\biggl| \sum_{i=1}^k (\Delta_i^n \widehat x_{\ell})^2 
-\frac{k}{n} \sum_{i=1}^n (\Delta_i^n \widehat x_{\ell})^2 \biggr|.
\end{equation*}
Let $\{ B^\circ(t) \}_{t \in [0,1]}$ be a Brownian bridge on $[0,1]$.
We then obtain the following result.
\begin{thm}\label{th1}
Assume that [A1]--[A3] and [B] hold. Then, it holds that under $H_0$,
\begin{equation*}
T_n \dto \sup_{t \in [0,1]}|B^\circ(t)|.
\end{equation*}
\end{thm}

\begin{rmk}
The distribution of $\sup_{t \in [0,1]}|B^\circ(t)|$ is known as 
the Kolmogorov distribution, and in the hypothesis testing problem \eqref{HTP}, 
the critical value can be calculated from the following equation
(see e.g., (9.40) in \cite{Billingsley1999}).
\begin{equation}\label{Kol}
\begin{split}
\PP \biggl( \sup_{t \in [0,1]}|B^\circ(t)| \le x \biggr) 
&= 1 -2\sum_{n =1}^\infty (-1)^{n-1} \exp(-2 n^2 x^2)
\\
&= \frac{\sqrt{2\pi}}{x} \sum_{n=1}^\infty 
\exp \biggl( -\frac{(2n-1)^2 \pi^2}{8 x^2} \biggr),
\quad
x \in (0,\infty).
\end{split}
\end{equation}
\end{rmk}

\subsection{The power of test statistic}\label{sec2-2}
In this subsection, we consider the hypothesis testing problem \eqref{HTP}
and address the power of the test statistic $T_n$.

We assume the following condition.
\begin{description}
\item[\textbf{[C]}]
For the estimator $\widehat \beta_{\ell}^2$ given in [B], 
there exists $\beta_{\ell}^\dag \in (0, \infty)$ such that
$\widehat \beta_{\ell}^2 \pto (\beta_{\ell}^\dag)^2$ under $H_1$.
\end{description}

\begin{rmk}
\begin{enumerate}
\item[(i)]
Since an estimator of $\kappa$ can be constructed regardless of 
the time dependence of the volatility function $\sigma(t)$, 
we can construct the estimator $\widehat \kappa$ which satisfies [A2] even under $H_1$. 
For details, see Appendices \ref{secA} and \ref{secB}.

\item[(ii)]
We suppose that [A1]--[A3] hold, 
and define $\widehat \beta_{\ell}^2 
= \sum_{i=1}^n (\Delta_i^n \widehat x_{\ell})^2$ and 
$\beta_{j,\ell} = \sigma_j^{\dag} \gamma_{\ell}^{-\alpha/2}$.
We see from Remark \ref{rmk1}-(iv) that 
the estimator $\widehat \beta_{\ell}^2$ satisfies [B].
On the other hand, we see from Proposition \ref{prop0} below and 
\begin{equation*}
\sum_{i=1}^n (\Delta_i^n x_{\ell})^2 
\pto (\beta_{\ell}^\dag)^2 := \gamma_{\ell}^{-\alpha} \int_0^1 \sigma(t)^2 \dd t
= \sum_{j=1}^{r+1} \beta_{j,\ell}^2 (\tau_j -\tau_{j-1}) \in (0,\infty)
\end{equation*}
that $\widehat \beta_{\ell}^2 \pto (\beta_{\ell}^\dag)^2$ under $H_1$
(see, for example, Subsection 5.6.1 in \cite{Jacod_Protter2012}).
Therefore, the estimator $\widehat \beta_{\ell}^2$ also satisfies [C].
\end{enumerate}
\end{rmk}

We obtain the following theorem.
\begin{thm}\label{th2}
Assume that [A1]--[A3] and [C] hold.
Then, it follows that under $H_1$, for any $\epsilon >0$,
\begin{equation*}
\lim_{n \to \infty} \PP( T_n \ge \epsilon ) = 1.
\end{equation*}
\end{thm}

\section{Examples}\label{sec3}
In this section, we consider the hypothesis testing problem \eqref{HTP}
for the volatility function in SPDE \eqref{spde} with $d = 1, 2$
and give procedures of the change detection for the volatility.

\subsection{Linear parabolic SPDEs with $d = 1$}\label{sec3-1}
We consider the linear parabolic SPDE
\begin{equation}\label{spde_ex1}
\dd X_t(y) = 
\biggl( \theta_2 \frac{\pd^2}{\pd y^2} +\theta_1 \frac{\pd}{\pd y} 
+ \theta_0 \biggr) X_t(y) \dd t 
+ \sigma(t) \dd B_t(y), 
\quad(t, y) \in [0,1] \times (0,1)
\end{equation}
with an initial value $X_0$ and the Dirichlet boundary condition 
$X_t(0) = X_t(1) = 0$, $t \in [0,1]$, where $\theta_2 \in (0, \infty)$, 
$\theta_1, \theta_0 \in \mathbb R$ and 
$[0,1] \ni t \mapsto \sigma(t) \in (0,\infty)$ is a time-dependent volatility function.
Note that $\sigma(t) = \sigma^*$ under $H_0$ and 
$\sigma(t) = \sum_{j=1}^{r+1} \sigma_j^\dag \ind_{[\tau_{j-1}, \tau_j)}(t)$
under $H_1$ in the hypothesis testing problem \eqref{HTP}.
$\{ B_t \}_{t \ge 0}$ is the cylindrical Brownian motion and is given by
\begin{equation*}
B_t = \sum_{l \in \mathbb N} w_l(t) e_l
\end{equation*}
with $e_l(y) = \sqrt{2} \exp(-\kappa y/2)\sin(\pi l y)$, 
$\kappa = \theta_1/\theta_2$ and 
independent real valued standard Brownian motions $\{ w_l \}_{l \in \mathbb N}$.
Note that $\beta_{\ell}(t) = \sigma(t)$ with $\ell = 1$ in this example.

Suppose that we have discrete observations 
$\mathbf X_{M,N} = \{ X_{t_i^N}(y_j) \}_{0 \le i \le N, 0 \le j \le M}$ with
$t_i^N = i \Delta = i/N$ and $y_j = j/M$.
For $b \in (0,1/2)$, $m \in \{ 1,\ldots, M \}$ and 
$n \in \{1, \ldots, N \}$, we will write the thinned data 
obtained from $\mathbf X_{M,N}$ as 
$\mathbf X_{m,n}^{(b)} = \{ X_{t_i^n}(\widetilde y_j) \}_{0 \le i \le n, 0 \le j \le m}$
with 
\begin{equation*}
t_i^n = i \cdot \frac{1}{N} \biggl\lfloor \frac{N}{n} \biggr\rfloor,
\quad
\widetilde y_j = b + j \cdot \frac{1-2b}{m},
\end{equation*}
where $\widetilde y_j \in \{ y_0, \ldots, y_M \}$.

We construct the test statistic $T_n$ based on the following two methodologies.
\begin{description}
\item[A.]
Methodology based on Bibnger and Trabs \cite{Bibinger_Trabs2020}. 

\item[B.]
Methodology based on Hildebrandt and Trabs \cite{Hildebrandt_Trabs2021}.

\end{description}

\subsubsection{Methodology A}\label{sec3-1-1}
Suppose that we have thinned data $\mathbf X_{m,N}^{(b)}$ 
with $b \in (0,1/2)$ and $m = \OO(N^\rho)$ for some $\rho \in (0,1/2)$.
We set
\begin{equation*}
Z_{j,N} = \frac{1}{N\sqrt{\Delta}}\sum_{i=1}^N (\Delta_i^N X(\widetilde y_j))^2.
\end{equation*}
Let $\Xi$ be a compact convex subset of $\mathbb R \times (0,\infty)$.
We define the estimator 
$(\widehat \kappa, \widehat{V}_0)$ of the parameter $(\kappa, V_0)$ by
\begin{equation*}
(\widehat \kappa, \widehat V_0) = \underset{(\kappa, V_0) \in \Xi}{\mathrm{argmin}} 
\frac{1}{m} \sum_{j=1}^m 
\biggl( Z_{j,N} - \frac{V_0}{\sqrt{\pi}} \exp(-\kappa \widetilde y_j) \biggr)^2
\end{equation*}
obtained from (20) in \cite{Bibinger_Trabs2020}. 
It then holds from Theorem \ref{app_th1} below that under either $H_0$ or $H_1$, 
\begin{equation*}
\sqrt{m N} (\widehat \kappa -\kappa^*) = \Op(1),
\end{equation*}
where $\sqrt{m N} = \OO(N^{(1+\rho)/2})$ for some $\rho \in (0, 1/2)$.
Using the estimator $\widehat \kappa$, 
we construct the approximate process of the coordinate process $x_1$ as follows.
\begin{equation}\label{CP-ex1}
\widehat x_1(t) 
= \sum_{j = 1}^M X_t(y_j) 
\bigl( g_1(y_j:\widehat \kappa) -g_1(y_{j-1}:\widehat \kappa) \bigr),
\quad t > 0,
\end{equation}
where $g_1$ is given in \eqref{gl}. We then define
\begin{equation*}
\langle \widehat x_1 \rangle_{k,n} = \sum_{i=1}^k (\Delta_i^n \widehat x_1)^2
\end{equation*}
and $\widehat \beta_1^2 = \widehat \sigma^2 = \langle \widehat x_1 \rangle_{n,n}$, 
and thus we have 
\begin{equation*}
\widehat \beta_1^2 =
\widehat \sigma^2 \pto V^* = \int_0^1 \sigma(t)^2 \dd t  =
\begin{cases}
(\sigma^*)^2 = (\beta_1^*)^2 & \text{under } H_0,
\\
\displaystyle \sum_{j=1}^{r+1} (\sigma_j^\dag)^2 (\tau_j -\tau_{j-1}) 
= (\beta_1^\dag)^2
& \text{under } H_1
\end{cases}
\end{equation*}
if [A3] holds true, i.e., $\dfrac{n^{3/2}}{N^{1+\rho}} \to 0$ 
and $\dfrac{n^{3/2}}{M} \to 0$.

The test statistic to detect the change in the volatility is given by
\begin{align*}
T_n &= \frac{1}{\widehat \sigma^2} \sqrt{\frac{n}{2}} 
\max_{1 \le k \le n}
\biggl| \sum_{i=1}^k (\Delta_i^n \widehat x_1)^2 
-\frac{k}{n} \sum_{i=1}^n (\Delta_i^n \widehat x_1)^2 \biggr|
\\
&= \frac{1}{\langle \widehat x_1 \rangle_{n,n}} \sqrt{\frac{n}{2}} 
\max_{1 \le k \le n}
\biggl| \langle \widehat x_1 \rangle_{k,n} 
-\frac{k}{n} \langle \widehat x_1 \rangle_{n,n} \biggr|
\\
&= \sqrt{\frac{n}{2}} 
\max_{1 \le k \le n}
\biggl| \frac{\langle \widehat x_1 \rangle_{k,n}}{\langle \widehat x_1 \rangle_{n,n}} 
-\frac{k}{n}  \biggr|,
\end{align*}
which converges to the supremum of a Brownian bridge on $[0,1]$ 
when [A3] holds.

\subsubsection{Methodology B}\label{sec3-1-2}
Suppose that we have thinned data $\mathbf X_{m,N}^{(b)}$ 
with $b \in (0,1/2)$, $m = \OO(\sqrt{N})$ and $N = \OO(m^2)$.
We set
\begin{equation*}
D_{i,j} X = \Delta_i^N X(\widetilde y_j) -\Delta_i^N X(\widetilde y_{j-1}),
\quad 
\widetilde D_{i,j} X = D_{i,j} X +D_{i+1,j} X
\end{equation*}
and $\overline y_j = (\widetilde y_{j-1} +\widetilde y_j)/2$.
Let $\Xi$ be a compact convex subset of $\mathbb R \times (0,\infty)^2$
and $\widehat V=\widehat \beta_1^2$.
We define 
the estimator $(\widehat \kappa, \widehat{\theta}_2, \widehat{V})$ by
\begin{align*}
(\widehat \kappa, \widehat \theta_2, \widehat V) 
&= \underset{(\kappa, \theta_2, V) \in \Xi}{\mathrm{argmin}} \Biggl\{
\frac{1}{m} \sum_{j=1}^m  
\biggl( \frac{1}{N\sqrt{\Delta}} \sum_{i=1}^N (D_{i,j} X)^2 
-V \exp(-\kappa \overline y_j) \psi_r(\theta_2) \biggr)^2
\\
&\qquad+\frac{1}{m} \sum_{j=1}^m  
\biggl( \frac{1}{N\sqrt{2\Delta}} \sum_{i=1}^{N-1} (\widetilde D_{i,j} X)^2 
-V \exp(-\kappa \overline y_j) 
\psi_{r/\sqrt{2}}(\theta_2) \biggr)^2
\Biggr\},
\end{align*}
where 
\begin{equation*}
\psi_r(\theta_2) = 
\frac{2}{\sqrt{\pi \theta_2}} \biggl(
1 -\exp\Bigl(-\frac{r^2}{4 \theta_2}\Bigr) 
+\frac{r}{\sqrt{\theta_2}} \int_{r/\sqrt{4\theta_2}}^\infty \ee^{-x^2} \dd x
\biggr)
\end{equation*}
obtained from (17) in \cite{Hildebrandt_Trabs2021}. 
It then holds from Theorem \ref{app_th2} below that under either $H_0$ or $H_1$, 
\begin{equation*}
\sqrt{m N} (\widehat \kappa -\kappa^*) = \Op(1),
\quad
\sqrt{m N} ( \widehat V -V^* ) = \Op(1),
\end{equation*}
where $\sqrt{m N} = \OO(N^{3/4})$. 
By using the estimator $\widehat \kappa$, 
we construct the approximate process of the coordinate process $x_1$ 
given in \eqref{CP-ex1}. 
We then define the test statistic of the change for the volatility by
\begin{equation*}
T_n = \frac{1}{\widehat V} \sqrt{\frac{n}{2}} 
\max_{1 \le k \le n}
\biggl| \sum_{i=1}^k (\Delta_i^n \widehat x_1)^2 
-\frac{k}{n} \sum_{i=1}^n (\Delta_i^n \widehat x_1)^2 \biggr|,
\end{equation*}
which converges to the supremum of a Brownian bridge on $[0,1]$ 
when [A3] holds.
Note that, in this case, the condition [A3] is given by
$\dfrac{n^{3/2}}{N^{3/2}} \to 0$ and $\dfrac{n^{3/2}}{M} \to 0$.

\subsection{Linear parabolic SPDEs with $d = 2$}
We consider the linear parabolic SPDE
\begin{equation}\label{spde_ex3-2}
\dd X_t(y,z)
= \biggl\{ 
\theta_2 
\biggl( \frac{\pd^2}{\pd y^2} +\frac{\pd^2}{\pd z^2}\biggr)
+\theta_{1,1} \frac{\pd}{\pd y} +\theta_{1,2} \frac{\pd}{\pd z}
+ \theta_0 
\biggr\} X_t(y,z) \dd t
+\sigma(t) \dd W_t^{Q}(y,z)
\end{equation}
for $(t,y,z) \in [0,1] \times (0,1)^2$
with an initial value $X_0$ and the Dirichlet boundary condition 
$X_t(y,z) = 0$, $(t,y,z) \in [0,1] \times \pd (0,1)^2$, 
where $\theta_2 \in (0, \infty)$, 
$\theta_1 = (\theta_{1,1}, \theta_{1,2})^\TT \in \mathbb R^2$, 
$\theta_0 \in \mathbb R$ and $[0,1] \ni t \mapsto \sigma(t) \in (0,\infty)$ 
is a time-dependent volatility function.
Note that $\sigma(t) = \sigma^*$ under $H_0$ and 
$\sigma(t) = \sum_{j=1}^{r+1} \sigma_j^\dag \ind_{[\tau_{j-1}, \tau_j)}(t)$
under $H_1$ in the hypothesis testing problem \eqref{HTP}.
The $Q$-Wiener process $W_t^Q$ is defined by
\begin{equation*}
W_t^Q = \sum_{l \in \mathbb N^2} \gamma_l^{-\alpha/2} w_l(t) e_l
\end{equation*}
with an unknown damping parameter $\alpha \in (0,2)$, 
independent real valued standard Brownian motions $\{ w_l \}_{l \in \mathbb N^2}$,
and $e_l(y,z) = e_{l_1}^{(1)}(y) e_{l_2}^{(2)}(z)$ for $l = (l_1,l_2) \in \mathbb N^2$,
where $\kappa = (\kappa_1, \kappa_2)^\TT$, $\kappa_j = \theta_{1,j}/\theta_2$
and $e_{l}^{(j)}(x) = \sqrt{2} \exp(-\kappa_j x/2) \sin(\pi l x)$.
We here set 
$\gamma_l = \lambda_l = \theta_2 \pi^2 |l|_2^2 
+\frac{|\theta_1|_2^2}{4 \theta_2} -\theta_0$ or $\gamma_l = \pi^2 |l|_2^2 +\mu_0$ 
with an unknown parameter $\mu_0 \in (-2\pi^2,\infty)$.
Note that $\beta_{\ell}(t) = \sigma(t) \gamma_{\ell}^{-\alpha/2}$ with $\ell = (1,1)$ 
in this example.
See \cite{TKU2023b}--\cite{TKU2026} 
for parametric estimation of the coefficient parameters in SPDE \eqref{spde_ex3-2}.

We provide the methodology based on Tonaki et al.\cite{TKU2025arXiv1}.
Suppose that we have discrete observations 
$\mathbf X_{M_1,M_2,N} = \{ X_{t_i^N}(y_{j_1}^{(1)},y_{j_2}^{(2)}) \}_{
0 \le i \le N, 0 \le j_1 \le M_1, 0 \le j_2 \le M_2}$ with
\begin{equation*}
t_i^N = i \Delta = \frac{i}{N},
\quad
y_j^{(k)} = \frac{j}{M_k}, 
\quad k \in \{ 1,2 \}.
\end{equation*}
For $b \in (0,1/2)$, $m_k \in \{1, \ldots, M_k \}$, $k \in \{1,2\}$ 
and $n \in \{1, \ldots, N \}$, we will write the thinned data 
obtained from $\mathbf X_{M_1,M_2,N}$ as 
$\mathbf X_{m_1,m_2,n}^{(b)} = 
\{ X_{t_i^n}(\widetilde y_{j_1}^{(1)}, \widetilde y_{j_2}^{(2)}) \}_{
0 \le i \le n, 0 \le j_1 \le m_1, 0 \le j_2 \le m_2}$
with $m_1 = m_2$,
\begin{equation*}
t_i^n = i \cdot \frac{1}{N} \biggl\lfloor \frac{N}{n} \biggr\rfloor,
\quad
\widetilde y_j^{(k)} = b + j \cdot \frac{1-2b}{m_1},
\end{equation*}
where $\widetilde y_j^{(k)} \in \{ y_0, \ldots, y_{M_k} \}$.

Suppose that we have two thinned data $\mathbf X_{m_1,m_2,N}^{(b)}$
and $\mathbf X_{m_1',m_2',N'}^{(b)}$ obtained from $\mathbf X_{M_1,M_2,N}$,
where $m_k' = m_k/2$, $m' = m_1' m_2'$, $N' = N/4$, 
$m := m_1 m_2 = \OO(N)$ and $N = \OO(m)$.
For the thinned data $\mathbf X_{m_1,m_2,N}^{(b)}$, we set the triple increments
\begin{align*}
T_{i,j,k} X 
&= \Delta_i^N X(\widetilde y_{j}^{(1)},\widetilde y_{k}^{(2)})
-\Delta_i^N X(\widetilde y_{j-1}^{(1)},\widetilde y_{k}^{(2)})
-\Delta_i^N X(\widetilde y_{j}^{(1)},\widetilde y_{k-1}^{(2)})
+\Delta_i^N X(\widetilde y_{j-1}^{(1)},\widetilde y_{k-1}^{(2)})
\\
&=
\sum_{l_1,l_2 \in \mathbb N} (x_{l_1,l_2}(t_i) -x_{l_1,l_2}(t_{i-1}))
(e_{l_1}^{(1)}(\widetilde y_{j}^{(1)}) -e_{l_1}^{(1)}(\widetilde y_{j-1}^{(1)}))
(e_{l_2}^{(2)}(\widetilde y_{k}^{(2)}) -e_{l_2}^{(2)}(\widetilde y_{k-1}^{(2)})).
\end{align*}
Let $T_{i,j,k}' X$ be the triple increments 
obtained from $\mathbf X_{m_1', m_2', N'}^{(b)}$.
We then define the estimator of $\alpha$ by
\begin{equation*}
\widehat \alpha =
\log \left( \frac{\displaystyle\frac{1}{m' N'} \sum_{k=1}^{m'_2} \sum_{j=1}^{m'_1} 
\sum_{i=1}^{N'} (T_{i,j,k}' X)^2}
{\displaystyle\frac{1}{m N} \sum_{k=1}^{m_2} \sum_{j=1}^{m_1} \sum_{i=1}^N 
(T_{i,j,k}X)^2}
\right)/\log(4)
\end{equation*}
obtained from (3.1) in \cite{TKU2025arXiv1}.
We write
$\overline y_j^{(p)} = (\widetilde y_{j-1}^{(p)} +\widetilde y_{j}^{(p)})/2$
for $p \in \{1,2\}$. 
Let $J_0$ be the Bessel function of the first kind of order $0$:
\begin{equation*}
J_0(x) = 1+\sum_{k=1}^\infty \frac{(-1)^k}{(k!)^2} \Bigl(\frac{x}{2} \Bigr)^{2k}.
\end{equation*}
For $r, \alpha >0$, we define
\begin{equation*}
\psi_{r,\alpha}(\theta_2)
=\frac{2}{\theta_2 \pi}
\int_0^\infty 
\frac{1-\ee^{-x^2}}{x^{1+2\alpha}}
\biggl(
J_0\Bigl(\frac{\sqrt{2}r x}{\sqrt{\theta_2}}\Bigr)
-2J_0\Bigl(\frac{r x}{\sqrt{\theta_2}}\Bigr)+1
\biggr) \dd x.
\end{equation*}
Let $\widetilde T_{i,j,k} X = T_{i,j,k}X +T_{i+1,j,k}X$. We define 
\begin{align*}
&(\widehat \kappa, \widehat \theta_2, \widehat V) 
\\
&= 
\underset{(\kappa, \theta_2, V) \in \Xi}{\mathrm{argmin}}
\Biggl\{
\frac{1}{m} \sum_{k = 1}^{m_2} \sum_{j = 1}^{m_1} 
\biggl(
\frac{1}{N \Delta^{\widehat \alpha}}\sum_{i=1}^{N} (T_{i,j,k}X)^2
-V \exp(-\kappa_1 \overline y_j^{(1)} -\kappa_2 \overline y_k^{(2)}) 
c_\gamma^{\widehat \alpha} \psi_{r,\widehat \alpha}(\theta_2)
\biggr)^2
\\
&\qquad+ 
\frac{1}{m} \sum_{k=1}^{m_2} \sum_{j=1}^{m_1} 
\biggl(
\frac{1}{N (2\Delta)^{\widehat \alpha}}\sum_{i=1}^{N-1} (\widetilde T_{i,j,k}X)^2
-V \exp(-\kappa_1 \overline y_j^{(1)} -\kappa_2 \overline y_k^{(2)}) 
c_\gamma^{\widehat \alpha} \psi_{r/\sqrt{2},\widehat \alpha}(\theta_2)
\biggr)^2 \Biggr\},
\end{align*}
where $\Xi$ is defined in Appendix \ref{secB} below and
$c_\gamma = \lim_{|l|_2 \to \infty} \frac{\gamma_l}{\pi^2 |l|_2^2}$.
Note that $c_\gamma = \theta_2$ for $\gamma_l = \lambda_l$, and 
$c_\gamma = 1$ for $\gamma_l = \pi^2 |l|_2^2 +\mu_0$.
It then follows from Theorem \ref{app_th4} below that under either $H_0$ or $H_1$,
\begin{equation*}
\frac{\sqrt{m N}}{\log(N)} (\widehat \kappa -\kappa^*) 
= \Op(1),
\end{equation*}
where $\sqrt{m N} = \OO(N)$.
Using the estimator $\widehat \kappa = (\widehat \kappa_1, \widehat \kappa_2)$, 
we construct the approximate process of the coordinate process $x_{1,1}$ as follows.
\begin{align*}
\widehat x_{1,1}(t) 
&= \sum_{j = 1}^{M_1} \sum_{k = 1}^{M_2} X_t(y_j^{(1)}, y_k^{(2)}) 
\bigl( g_1(y_j^{(1)}:\widehat \kappa_1) -g_1(y_{j-1}^{(1)}:\widehat \kappa_1) \bigr)
\\
& \qquad \qquad \times 
\bigl( g_1(y_k^{(2)}:\widehat \kappa_2) -g_1(y_{k-1}^{(2)}:\widehat \kappa_2) \bigr).
\end{align*}
We set
\begin{equation*}
\langle \widehat x_{1,1} \rangle_{k,n} = \sum_{i=1}^k (\Delta_i^n \widehat x_{1,1})^2,
\quad
\widehat \beta_{1,1}^2 = \langle \widehat x_{1,1} \rangle_{n,n}.
\end{equation*}
In this case, the condition [A3] is given by 
$\dfrac{n^{3/2}}{(M_1 \land M_2)^{2\alpha \tand 2}} \to 0$
since the convergence
\begin{equation*}
\dfrac{n^2 \Delta_n^{\alpha \tand 1}}{N^2} (\log(N))^2 
= \dfrac{n^2 \Delta_n^{\alpha \tand 1}}{N^2 \Delta_N^{\alpha \tand 1}} 
\cdot \Delta_N^{\alpha \tand 1} (\log(N))^2 
\to 0
\end{equation*}
always holds true. Hence, we have 
\begin{equation*}
\widehat \beta_{1,1}^2 \pto \gamma_{1,1}^{-\alpha} \int_0^1 \sigma(t)^2 \dd t  =
\begin{cases}
(\beta_{1,1}^*)^2 & \text{under } H_0,
\\
\displaystyle \gamma_{1,1}^{-\alpha} 
\sum_{j=1}^{r+1} (\sigma_j^\dag)^2 (\tau_j -\tau_{j-1}) = (\beta_{1,1}^\dag)^2
& \text{under } H_1
\end{cases}
\end{equation*}
when $\dfrac{n^{3/2}}{(M_1 \land M_2)^{2\alpha \tand 2}} \to 0$ holds.
We then define the test statistic of the change for the volatility by
\begin{align*}
T_n &= \frac{1}{\widehat \beta_{1,1}^2} \sqrt{\frac{n}{2}} 
\max_{1 \le k \le n}
\biggl| \sum_{i=1}^k (\Delta_i^n \widehat x_{1,1})^2 
-\frac{k}{n} \sum_{i=1}^n (\Delta_i^n \widehat x_{1,1})^2 \biggr|
\\
&=
\sqrt{\frac{n}{2}} 
\max_{1 \le k \le n}
\biggl| 
\frac{\langle \widehat x_{1,1} \rangle_{k,n}}{\langle \widehat x_{1,1} \rangle_{n,n}}
-\frac{k}{n} 
\biggl|.
\end{align*}

\section{Simulations}\label{sec4}
In this section, we consider the hypothesis testing problem \eqref{HTP}
for linear parabolic SPDE \eqref{spde_ex1}
in order to verify the asymptotic behavior of our test statistic $T_n$.

A numerical solution of SPDE \eqref{spde_ex1} is generated by
\begin{equation*}
\widetilde X_{t_i}(y_j)
= \sum_{l = 1}^{L} x_l (t_i) e_l(y_j), 
\quad i \in \{ 1,\ldots, N \}, j \in \{ 1,\ldots, M \}
\end{equation*}
with
\begin{equation*}
\left\{
\begin{split}
x_l(0) &= \langle X_0, e_l \rangle,
\nonumber
\\
x_l (t_i) &= 
\ee^{-\lambda_l \Delta} x_l (t_{i-1})
+\sigma(t_{i-1}) \sqrt{\frac{1-\ee^{-2\lambda_l \Delta}}{2\lambda_l}} Z_{i,l},
\quad i \in \{1,\ldots, N \},
\end{split}
\right.
\end{equation*}
where $\{ Z_{i,l} \}$ are independent standard normal random variables.
We set $N = 10^4$, $M = 10^4$, $L = 10^5$, $X_0 = 0$
and the true values of the parameters 
$(\theta_0^*, \theta_1^*, \theta_2^*) = (0, 0.2, 0.2)$ in SPDE \eqref{spde_ex1}.

\begin{figure}[h]
\captionsetup{margin=5pt}
\captionsetup[sub]{margin=5pt}
 \begin{minipage}[t]{0.32\linewidth}
  \centering
  \includegraphics[keepaspectratio, width=5.2cm]{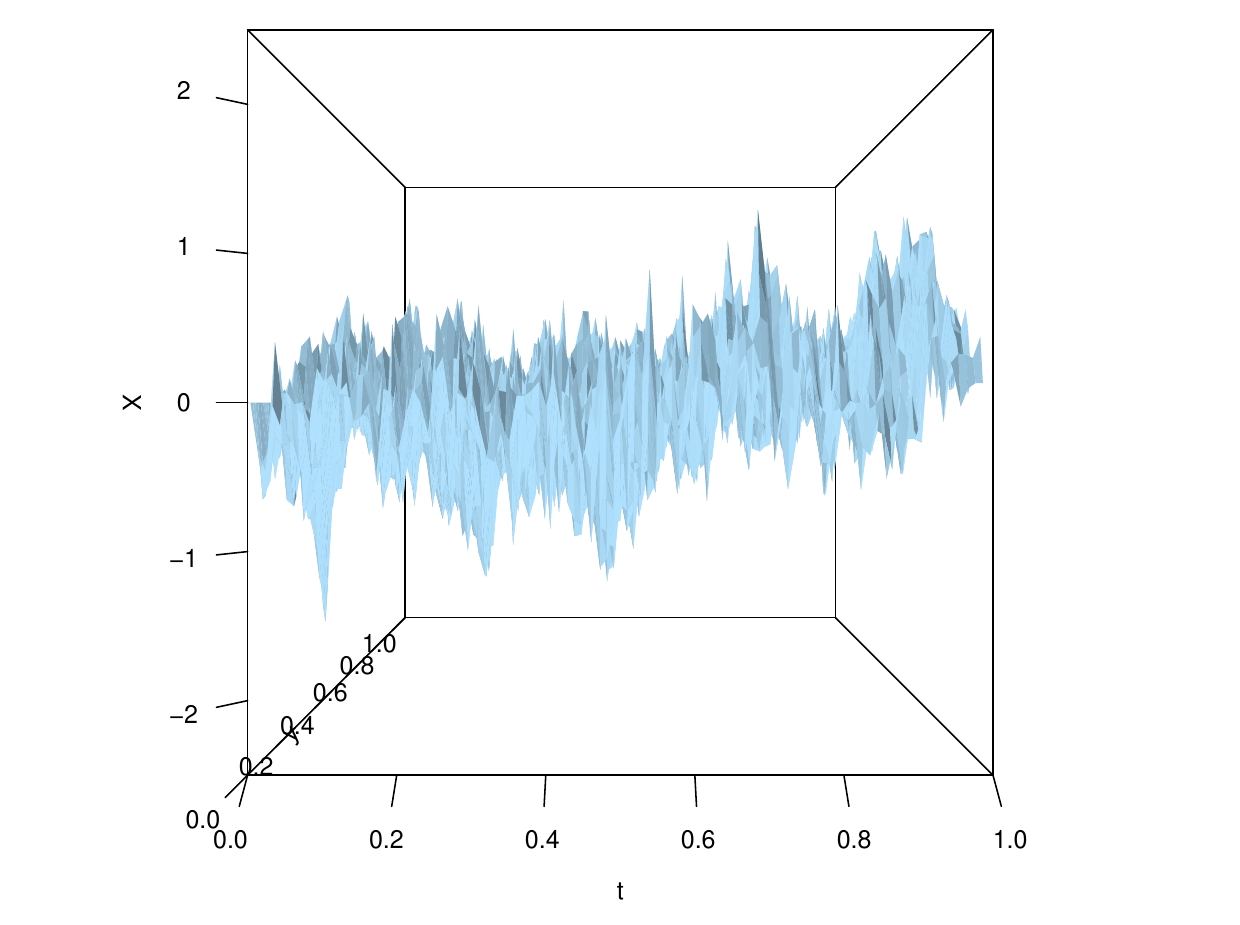}
  \subcaption{$\tau=1$ (No change points).}
 \end{minipage}
 \begin{minipage}[t]{0.32\linewidth}
  \centering
  \includegraphics[keepaspectratio, width=5.2cm]{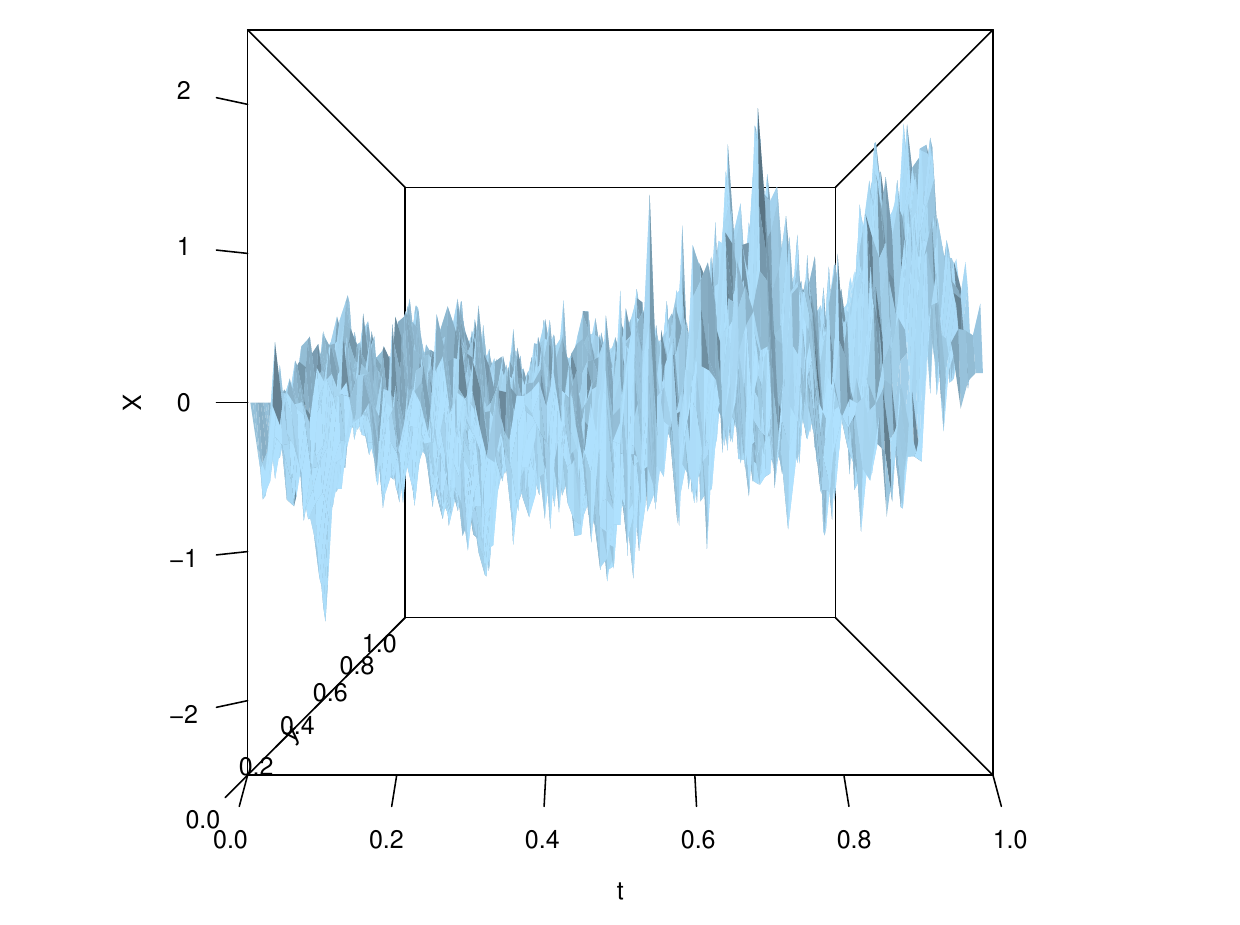}
  \subcaption{$\tau=0.5$, $\sigma_2^\dag = 1.5$.}
 \end{minipage}
 \begin{minipage}[t]{0.32\linewidth}
  \centering
  \includegraphics[keepaspectratio, width=5.2cm]{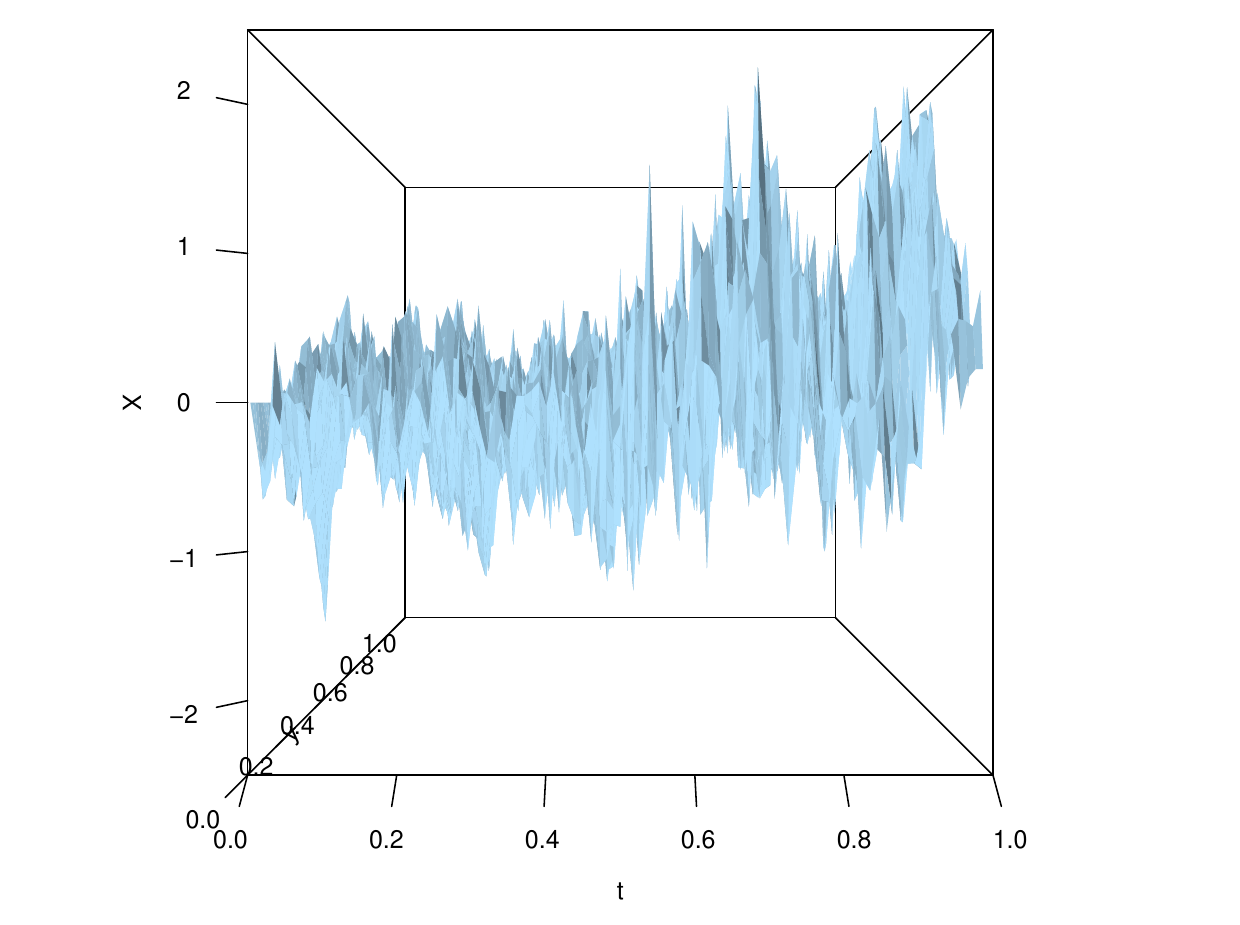}
  \subcaption{$\tau=0.5$, $\sigma_2^\dag = 1.7$.}
 \end{minipage}\\
  \begin{minipage}[t]{0.32\linewidth}
  \centering
  \includegraphics[keepaspectratio, width=5.2cm]{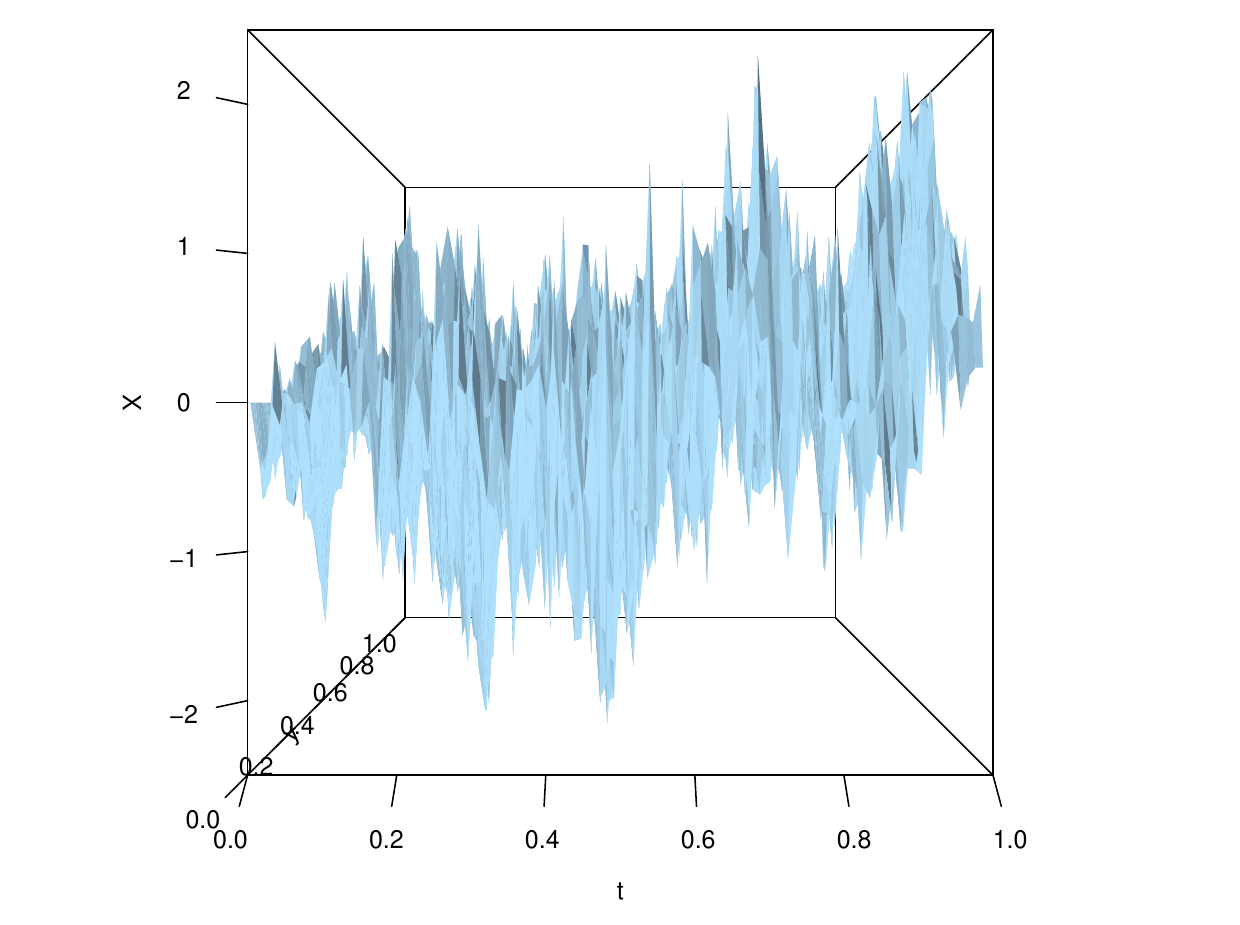}
  \subcaption{$\tau=0.1$, $\sigma_2^\dag = 1.8$.}
 \end{minipage}
 \begin{minipage}[t]{0.32\linewidth}
  \centering
  \includegraphics[keepaspectratio, width=5.2cm]{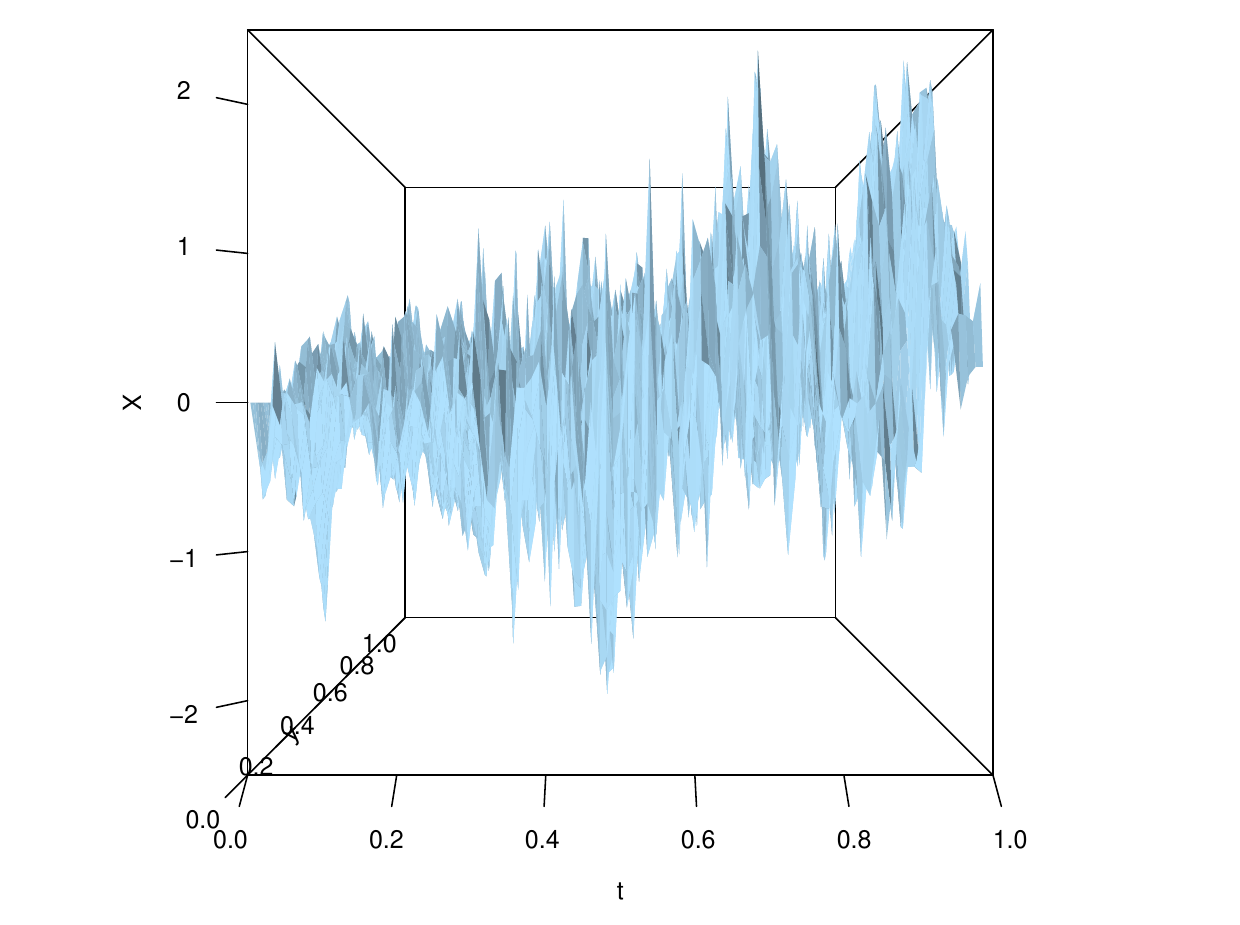}
  \subcaption{$\tau=0.3$, $\sigma_2^\dag = 1.8$.}
 \end{minipage}
 \begin{minipage}[t]{0.32\linewidth}
  \centering
  \includegraphics[keepaspectratio, width=5.2cm]{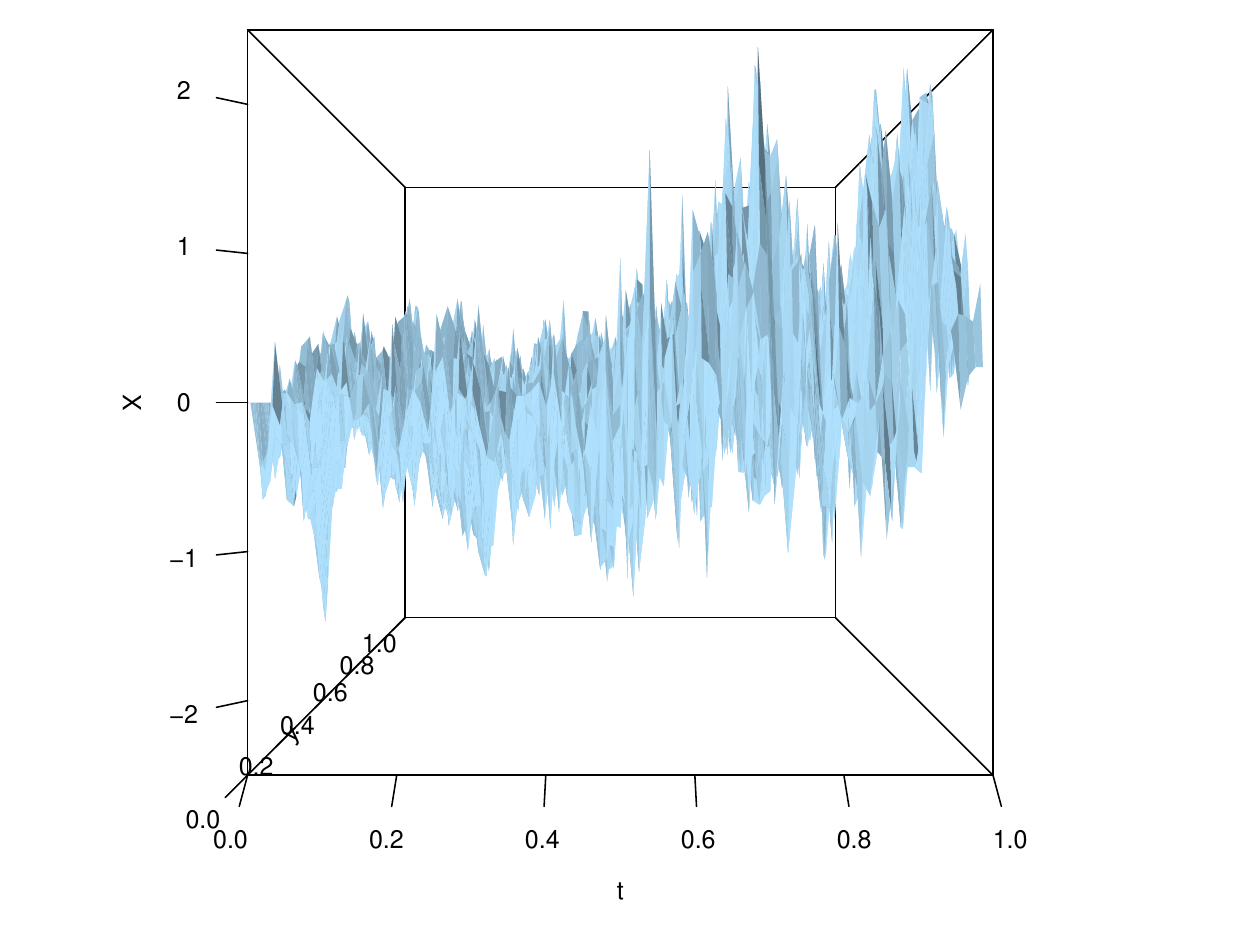}
  \subcaption{$\tau=0.5$, $\sigma_2^\dag = 1.8$.}
 \end{minipage}
 \caption{Sample paths of $X_t(y)$ given in SPDE \eqref{spde_ex1} 
 with $\sigma(t) = \ind_{[0,\tau)}(t) +\sigma_2^\dag \ind_{[\tau,1]}(t)$.
 }
 \label{fig_path}
\end{figure}
In order to verify the asymptotic behavior of the test statistic $T_n$, 
we perform numerical simulations with the following situations.
\begin{description}
\item[Situation 1:]
$\sigma(t) = \sigma^*$ over $t \in [0,1]$ with $\sigma^* = 1$.

\item[Situation 2:]
$\sigma(t) = \sigma_1^\dag \ind_{[0,\tau)}(t) +\sigma_2^\dag \ind_{[\tau,1]}(t)$
with $\tau = 0.5$, $\sigma_1^\dag = 1$, $\sigma_2^\dag \in \{ 1.4, 1.5, \ldots, 1.8 \}$.

\item[Situation 3:]
$\sigma(t) = \sigma_1^\dag \ind_{[0,\tau)}(t) +\sigma_2^\dag \ind_{[\tau,1]}(t)$
with $\tau \in \{0.1,0.2,\ldots,0.5\}$, $\sigma_1^\dag = 1$, $\sigma_2^\dag = 1.8$.
\end{description}

Figure \ref{fig_path} illustrates sample paths of $X_t(y)$
given in SPDE \eqref{spde_ex1}.
We consider Situations 1--3
and examine whether the volatility $\sigma(t)$ changes based on the discrete spatio-temporal data obtained from these sample paths.

We construct the test statistic $T_n$ 
using the method presented in Subsubsection \ref{sec3-1-2}, 
and examine the performance of $T_n$.
We compute the estimators $(\widehat \kappa, \widehat V)$ in $T_n$ 
based on the thinned data $\mathbf X_{m,N}^{(b)}$ with $m = 100$, 
$N = 10^4$ and $b = 0.0297$.
In all situations, the number of Monte Carlo simulations is $1000$.
All situations are conducted at significant level $0.05$ and 
the corresponding critical value is $1.3581$,
which is calculated from \eqref{Kol}.
Note that we have
$\frac{n^{3/2}}{m N} \approx 0.001, 0.003, 0.005, 0.008$
and $\frac{n^{3/2}}{M} \approx 0.100, 0.283, 0.520, 0.800$ 
for $n = 100, 200, 300, 400$, respectively.

Figure \ref{fig1} gives the histogram of $T_n$ with $n = 400$ in Situation 1
versus the probability density function (red line) 
of the theoretical asymptotic distribution
and the empirical distribution function versus 
the cumulative distribution function (red line).
The empirical results are 
broadly consistent with
the theoretical results.
\begin{figure}[h]
\begin{center}
\includegraphics[keepaspectratio, width=6.5cm]{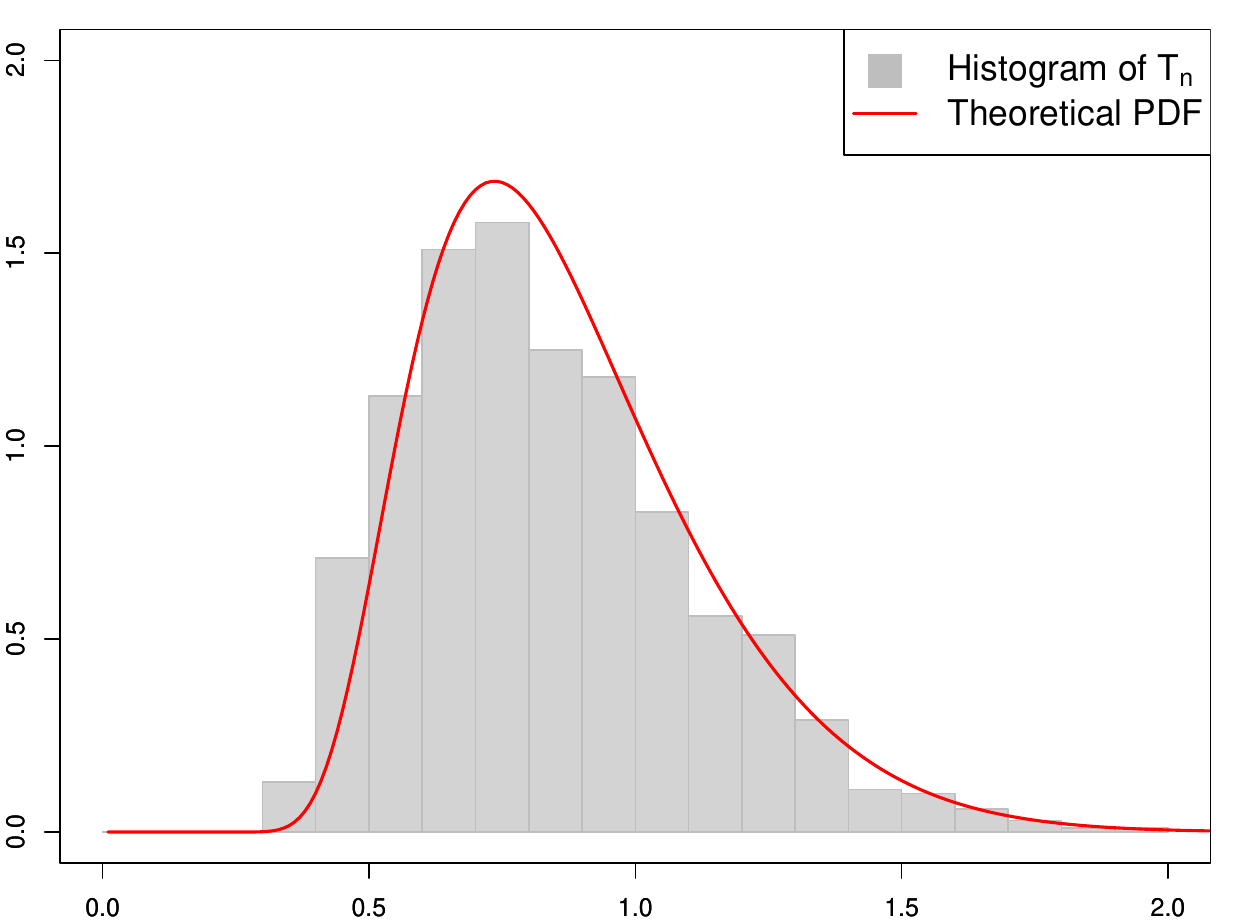}
\includegraphics[keepaspectratio, width=6.5cm]{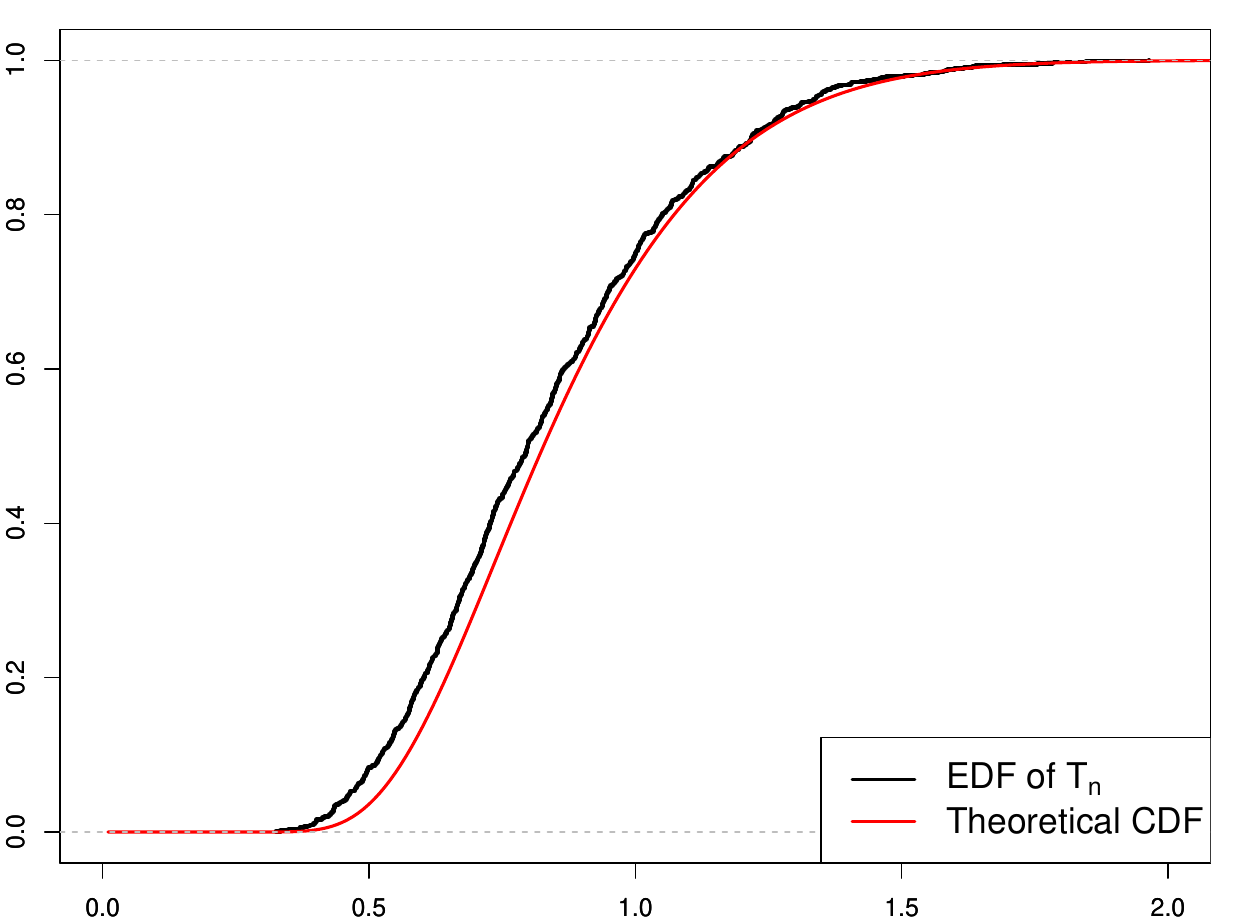}
\end{center}
\caption{Histogram and empirical distribution function of $T_n$ 
with $n = 400$ in Situation 1.}
\label{fig1}
\end{figure}

Tables \ref{tab2} and \ref{tab3} show the empirical powers of 
$T_n$ with $n \in \{100,200,300,400\}$ in Situations 2 and 3, respectively. 
These values represent the percentages obtained from $1000$ Monte Carlo iterations
which exceed the critical value $1.3581$.
In all situations, the empirical powers of $T_n$ approach $1$ as $n$ increases.
\renewcommand{\arraystretch}{1.2}
\begin{table}[h]
\caption{Empirical powers of $T_n$ in Situation 2.}
\label{tab2}
\begin{center}
\begin{tabular}{c|ccccc} \hline
 & & $\sigma_1^\dag = 1$ & $\underset{\tau = 0.5}{\longrightarrow}$ & $\sigma_2^\dag$ &
\\
		& $\sigma_2^\dag = 1.4$ & $\sigma_2^\dag = 1.5$ & $\sigma_2^\dag = 1.6$ & 
$\sigma_2^\dag = 1.7$ & $\sigma_2^\dag = 1.8$
\\ \hline
$n = 100$ & 0.466 & 0.593 & 0.941 & 0.963 & 0.983
 \\
$n = 200$ & 0.769 & 0.918 & 0.999 & 1.000 & 1.000
 \\
$n = 300$ & 0.929 & 0.978 & 1.000 & 1.000 & 1.000
 \\
$n = 400$ & 0.977 & 0.999 & 1.000 & 1.000 & 1.000
 \\   \hline
\end{tabular}
\end{center}
\end{table}
\begin{table}[h]
\caption{Empirical powers of $T_n$ in Situation 3.}
\label{tab3}
\begin{center}
\begin{tabular}{c|ccccc} \hline
 & & $\sigma_1^\dag = 1$ & $\underset{\tau}{\longrightarrow}$ & $\sigma_2^\dag = 1.8$ &
\\
		& $\tau=0.1$ & $\tau=0.2$ & $\tau=0.3$ & $\tau=0.4$ & $\tau=0.5$ 
\\ \hline
$n = 100$ & 0.460 & 0.813 & 0.957 & 0.973 & 0.983
 \\
$n = 200$ & 0.635 & 0.989 & 1.000 & 0.999 & 1.000
 \\
$n = 300$ & 0.819 & 0.999 & 1.000 & 1.000 & 1.000
 \\
$n = 400$ & 0.937 & 1.000 & 1.000 & 1.000 & 1.000
 \\   \hline
\end{tabular}
\end{center}
\end{table}

\section{Proofs}\label{sec5}
We set the following notation.

\begin{enumerate}
\item
For families $\{ a_\lambda \}, \{ b_\lambda \} \subset \mathbb R$, 
we write $a_\lambda \lesssim b_\lambda$ 
if $|a_\lambda| \le C |b_\lambda|$ for some universal constant $C > 0$ 
and any $\lambda$, and we write $a_\lambda \sim b_\lambda$ 
if $a_\lambda \lesssim b_\lambda$ and $b_\lambda \lesssim a_\lambda$.

\item
For $x = (x_1, \ldots, x_d) \in \mathbb R^d$ and 
$f:\mathbb R^d \to \mathbb R$, 
we write $\pd_{x_i} f(x) = \frac{\pd}{\pd x_i}f(x)$
and $\pd_x f(x) = (\pd_{x_1}f(x), \ldots, \pd_{x_d}(x))$.

\item
Let $\GG_i = \GG_i^n = \sigma[\{ w_{\ell}(s) \}_{s \le t_i^n}]$.

\item
$X_n(\cdot) \wto X(\cdot)$ in $\mathbb D[0,1]$ denotes
$X_n(\cdot)$ weakly converges to $X(\cdot)$ in the Skorohod space on $[0,1]$.

\end{enumerate}

\subsection{Proof of Theorem \ref{th1}}
For $\ell \in \mathbb N^d$ obtained from [A1], 
we consider the Ornstein-Uhlenbeck dynamics 
\begin{equation*}
\left\{
\begin{split}
\dd x_{\ell}(t) &= -\lambda_{\ell} x_{\ell}(t) \dd t +\beta_{\ell}(t) \dd w_{\ell}(t),
\\
x_{\ell}(0) &= \langle X_0, e_{\ell} \rangle.
\end{split}
\right.
\end{equation*}
We define
\begin{equation*}
T_n^* = \frac{1}{(\beta_{\ell}^*)^2} \sqrt{\frac{n}{2}} 
\max_{1 \le k \le n}
\biggl| \sum_{i=1}^k (\Delta_i^n x_{\ell})^2 
-\frac{k}{n} \sum_{i=1}^n (\Delta_i^n x_{\ell})^2 \biggr|.
\end{equation*}
We then prove the following two propositions in order to show Theorem \ref{th1}.
\begin{prop}\label{prop1}
Under $H_0$, it holds that
\begin{equation*}
T_n^* \dto \sup_{t \in [0,1]}|B^\circ(t)|.
\end{equation*}
\end{prop}

\begin{prop}[]\label{prop2}
Assume that [A1]--[A3] and [B] hold. Under $H_0$, it holds that
\begin{equation*}
T_n -T_n^* \pto 0.
\end{equation*}
\end{prop}

\begin{proof}[\bf{Proof of Proposition \ref{prop1}}]
For the coordinate process \eqref{cor-pro}, we set 
\begin{equation*}
\zeta_{i,l} = \zeta_{i,l,n} = x_l(t_i^n) -\ee^{-\lambda_l \Delta_n} x_l(t_{i-1}^n)
= \int_{t_{i-1}^n}^{t_i^n} \beta_l(s) \ee^{-\lambda_l (t_i^n -s)} \dd w_l(s),
\quad
\zeta_i = \zeta_{i,\ell},
\end{equation*}
\begin{equation*}
\eta_{i,l} = \eta_{i,l,n} = (\ee^{-\lambda_l \Delta_n} -1) x_l(t_{i-1}^n),
\quad
\eta_i = \eta_{i,\ell}.
\end{equation*}
Since we have $\Delta_i^n x_l = \zeta_{i,l} +\eta_{i,l}$, 
we obtain the following decomposition.
\begin{align*}
\sum_{i=1}^k (\Delta_i^n x_{\ell})^2 
-\frac{k}{n} \sum_{i=1}^n (\Delta_i^n x_{\ell})^2
&= \sum_{i=1}^k \zeta_i^2 
-\frac{k}{n} \sum_{i=1}^n \zeta_i^2
\\
&\quad + \sum_{i=1}^k G_{1,i}^{\Delta} 
-\frac{k}{n} \sum_{i=1}^n G_{1,i}^{\Delta}
\\
&\quad + \sum_{i=1}^k G_{2,i}^{\Delta}
-\frac{k}{n} \sum_{i=1}^n G_{2,i}^{\Delta},
\end{align*}
where $G_{1,i}^{\Delta} = 2 \zeta_i \eta_i$, $G_{2,i}^{\Delta} = \eta_i^2$.
Since it follows that
\begin{align*}
\Biggl| T_n^* - 
\frac{1}{(\beta_{\ell}^*)^2} \sqrt{\frac{n}{2}} 
\max_{1 \le k \le n}
\biggl| \sum_{i=1}^k \zeta_i^2 
-\frac{k}{n} \sum_{i=1}^n \zeta_i^2 \biggr| 
\Biggr|
\lesssim
\sum_{j=1}^2 \sqrt{n} 
\max_{1 \le k \le n}
\biggl| \sum_{i=1}^k G_{j,i}^{\Delta} 
-\frac{k}{n} \sum_{i=1}^n G_{j,i}^{\Delta} \biggr|,
\end{align*}
we show the following two assertions
in order to show Proposition \ref{prop1}.

\begin{enumerate}
\item[(i)]
Under $H_0$, it holds that
\begin{equation}\label{lem1}
\frac{1}{(\beta_{\ell}^*)^2} \sqrt{\frac{n}{2}} 
\max_{1 \le k \le n}
\biggl| \sum_{i=1}^k \zeta_i^2 
-\frac{k}{n} \sum_{i=1}^n \zeta_i^2 \biggr|
\dto 
\sup_{t \in [0,1]}|B^\circ(t)|.
\end{equation}

\item[(ii)]
For $j = 1, 2$, it holds that under $H_0$,
\begin{equation}\label{lem2}
\sqrt{n} 
\max_{1 \le k \le n}
\biggl| \sum_{i=1}^k G_{j,i}^{\Delta} 
-\frac{k}{n} \sum_{i=1}^n G_{j,i}^{\Delta} \biggr|
\pto 0.
\end{equation}

\end{enumerate}

\textit{Proof of \eqref{lem1}.}
Since $\EE[\zeta_i^2 | \GG_{i-1}]$ is independent of $i$ under $H_0$, that is,
\begin{equation*}
\EE[\zeta_i^2 | \GG_{i-1}] = \EE[\zeta_i^2]
= (\beta_{\ell}^*)^2 
\frac{1- \ee^{-2\lambda_{\ell} \Delta_n}}{2 \lambda_{\ell}},
\end{equation*}
we have
\begin{align*}
&\frac{1}{(\beta_{\ell}^*)^2} \sqrt{\frac{n}{2}} \max_{1 \le k \le n}
\biggl| \sum_{i=1}^k \zeta_i^2 
-\frac{k}{n} \sum_{i=1}^n \zeta_i^2 \biggr|
\\
&= \frac{1}{(\beta_{\ell}^*)^2} \sqrt{\frac{n}{2}} \max_{1 \le k \le n}
\biggl| \sum_{i=1}^k (\zeta_i^2 -\EE[\zeta_i^2 | \GG_{i-1}])
-\frac{k}{n} \sum_{i=1}^n (\zeta_i^2 -\EE[\zeta_i^2 | \GG_{i-1}]) \biggr|
\\
&= \sup_{0 \le t \le 1}
\biggl| 
\frac{1}{(\beta_{\ell}^*)^2} \sqrt{\frac{n}{2}} 
\sum_{i=1}^{\lfloor n t \rfloor} (\zeta_i^2 -\EE[\zeta_i^2 | \GG_{i-1}])
-\frac{\lfloor n t \rfloor}{n} 
\frac{1}{(\beta_{\ell}^*)^2} \sqrt{\frac{n}{2}} 
\sum_{i=1}^n (\zeta_i^2 -\EE[\zeta_i^2 | \GG_{i-1}]) \biggr|
\end{align*}
under $H_0$. Let $\{ B(t) \}_{t \ge 0}$ be a one-dimensional Brownian motion.
We then show 
\begin{equation*}
\frac{1}{(\beta_{\ell}^*)^2} \sqrt{\frac{n}{2}} 
\sum_{i=1}^{\lfloor n t \rfloor} (\zeta_i^2 -\EE[\zeta_i^2 | \GG_{i-1}])
\wto B(t) \quad \text{in} \ \mathbb D[0,1]
\end{equation*}
under $H_0$, which can be obtained by showing that under $H_0$, 
\begin{equation*}
\frac{n}{2(\beta_{\ell}^*)^4}
\sum_{i=1}^{\lfloor n t \rfloor} 
\EE \bigl[ (\zeta_i^2 -\EE[\zeta_i^2 | \GG_{i-1}])^2 \big| \GG_{i-1} \bigr] \pto t,
\end{equation*}
\begin{equation*}
n^2 \sum_{i=1}^{\lfloor n t \rfloor} 
\EE \bigl[ (\zeta_i^2 -\EE[\zeta_i^2 | \GG_{i-1}])^4 \big| \GG_{i-1} \bigr] \pto 0
\end{equation*}
for $t \in [0,1]$ (see Corollary 3.8 in \cite{McLeish1974}).

Since
\begin{equation*}
\EE \bigl[ (\zeta_i^2 -\EE[\zeta_i^2 | \GG_{i-1}])^2 \big| \GG_{i-1} \bigr]
= \EE[\zeta_i^4| \GG_{i-1}] -\EE[\zeta_i^2|\GG_{i-1}]^2
\end{equation*}
and
\begin{equation*}
\EE[\zeta_i^4| \GG_{i-1}] 
= \EE[\zeta_i^4]
= 3 (\beta_{\ell}^*)^4 
\biggl( \frac{1- \ee^{-2\lambda_{\ell} \Delta_n}}
{2 \lambda_{\ell}} \biggr)^2,
\end{equation*}
we obtain
\begin{align*}
n \sum_{i=1}^{\lfloor n t \rfloor} 
\EE \bigl[ (\zeta_i^2 -\EE[\zeta_i^2 | \GG_{i-1}])^2 \big| \GG_{i-1} \bigr]
&= n \sum_{i=1}^{\lfloor n t \rfloor} 
2 (\beta_{\ell}^*)^4 
\biggl( 
\frac{1- \ee^{-2\lambda_{\ell} \Delta_n}}{2 \lambda_{\ell}} \biggr)^2
\\
&= 2(\beta_{\ell}^*)^4 \cdot \Delta_n \lfloor n t \rfloor \cdot n \Delta_n \cdot
\biggl( 
\frac{1- \ee^{-2\lambda_{\ell} \Delta_n}}{2 \lambda_{\ell} \Delta_n} 
\biggr)^2
\\
&\to 2 (\beta_{\ell}^*)^4 t.
\end{align*}
We also find from 
\begin{equation*}
\EE \bigl[ (\zeta_i^2 -\EE[\zeta_i^2 | \GG_{i-1}])^4 \big| \GG_{i-1} \bigr]
\lesssim \EE [\zeta_i^8 | \GG_{i-1} ]
= \EE[\zeta_i^8] = \OO(\Delta_n^4)
\end{equation*} 
that
\begin{equation*}
n^2 \sum_{i=1}^{\lfloor n t \rfloor} 
\EE \bigl[ (\zeta_i^2 -\EE[\zeta_i^2 | \GG_{i-1}])^4 \big| \GG_{i-1} \bigr] 
=\OO(n^3 \Delta_n^4) = \oo(1).
\end{equation*}
This concludes the proof of \eqref{lem1}.

\textit{Proof of \eqref{lem2}.}
First, we consider $j = 1$. 
Note that $\sum_{i=1}^{k_n} x_{\ell}(t_{i-1}^n)^2 = \Op(k_n)$
for a positive integer sequence $\{ k_n \}$.
Since
\begin{align*}
\frac{1}{n} \EE[\Delta_n^{-3/2} \zeta_i \eta_i | \GG_{i-1}]
&= \frac{1}{n \Delta_n^{3/2}} \eta_i\EE[\zeta_i]
= 0,
\\
\frac{1}{n^2} \sum_{i=1}^n 
\EE[ \Delta_n^{-3} \zeta_i^2 \eta_i^2 | \GG_{i-1}]
&= \frac{1}{n^2 \Delta_n^3} \sum_{i=1}^n \eta_i^2 \EE[\zeta_i^2]
\\
&= \frac{1}{n} \cdot 
\biggl( \frac{\ee^{-\lambda_{\ell} \Delta_n} -1}{\Delta_n} \biggr)^2
\frac{1-\ee^{-2\lambda_{\ell} \Delta_n}}{2\lambda_{\ell} \Delta_n} 
\cdot \frac{1}{n} \sum_{i=1}^n x_{\ell}(t_{i-1}^n)^2 
\\
& \pto 0,
\end{align*}
we see from Lemma \ref{lem3} below that
\begin{equation*}
\frac{1}{n} \max_{1 \le k \le n}
\biggl| \sum_{i=1}^k \Delta_n^{-3/2} \zeta_i \eta_i \biggr|
= \frac{1}{n \Delta_n^{3/2}} \max_{1 \le k \le n}
\biggl| \sum_{i=1}^k \zeta_i \eta_i \biggr|
=\op(1),
\end{equation*}
which yields 
\begin{equation*}
\sqrt{n} \max_{1 \le k \le n}
\biggl| \sum_{i=1}^k G_{1,i}^{\Delta}
-\frac{k}{n} \sum_{i=1}^n G_{1,i}^{\Delta} \biggr|
= \sqrt{n} \max_{1 \le k \le n}
\biggl| \sum_{i=1}^k \zeta_i \eta_i -\frac{k}{n} \sum_{i=1}^n \zeta_i \eta_i \biggr|
=\op(1).
\end{equation*}

We next consider $j = 2$. 
Since $1 -\ee^{-x} \le x$ for $x \in [0,\infty)$, we obtain
\begin{align*}
&\sqrt{n} \max_{1 \le k \le n}
\biggl| \sum_{i=1}^k G_{2,i}^{\Delta}
-\frac{k}{n} \sum_{i=1}^n G_{2,i}^{\Delta} \biggr|
\\
&= \sqrt{n} (1 -\ee^{-\lambda_{\ell} \Delta_n})^2 \max_{1 \le k \le n}
\biggl| \sum_{i=1}^k x_{\ell}(t_{i-1}^n)^2
-\frac{k}{n} \sum_{i=1}^n x_{\ell}(t_{i-1}^n)^2 \biggr|
\\
&\le \sqrt{n} (\lambda_{\ell}\Delta_n)^2 
\times 2 \sum_{i=1}^n x_{\ell}(t_{i-1}^n)^2 
\\
&= \Op(n^{3/2} \Delta_n^2) = \op(1).
\end{align*}
This completes the proof of \eqref{lem2}.
\end{proof}

\begin{proof}[\bf{Proof of Proposition \ref{prop2}}]
Let
\begin{align*}
U_n^* = 
\sqrt{\frac{n}{2}} 
\max_{1 \le k \le n}
\biggl| \sum_{i=1}^k (\Delta_i^n x_{\ell})^2 
-\frac{k}{n} \sum_{i=1}^n (\Delta_i^n x_{\ell})^2 \biggr|,
\\
U_n = 
\sqrt{\frac{n}{2}} 
\max_{1 \le k \le n}
\biggl| \sum_{i=1}^k (\Delta_i^n \widehat x_{\ell})^2 
-\frac{k}{n} \sum_{i=1}^n (\Delta_i^n \widehat x_{\ell})^2 \biggr|.
\end{align*}
Since it follows from [B] that
\begin{equation*}
\frac{1}{\widehat \beta_{\ell}^2} -\frac{1}{(\beta_{\ell}^*)^2} = \op(1),
\end{equation*}
and from Proposition \ref{prop1} that $U_n^* = \Op(1)$, we have
\begin{equation*}
|T_n -T_n^*| 
= \biggl| \frac{U_n}{\widehat \beta_{\ell}^2} 
-\frac{U_n^*}{(\beta_{\ell}^*)^2} \biggr|
\le \biggl| \frac{U_n -U_n^*}{\widehat \beta_{\ell}^2} \biggr|
+\biggl| \frac{1}{\widehat \beta_{\ell}^2} 
-\frac{1}{(\beta_{\ell}^*)^2} \biggr| U_n^* 
= \biggl| \frac{U_n -U_n^*}{\widehat \beta_{\ell}^2} \biggr| +\op(1).
\end{equation*}
Hence, we show $U_n -U_n^* = \op(1)$ in order to prove Proposition \ref{prop2}. 
We obtain 
\begin{align*}
|U_n -U_n^*| 
&\le \sqrt{n} \max_{1 \le k \le n}
\biggl| \sum_{i=1}^k 
((\Delta_i^n \widehat x_{\ell})^2 -(\Delta_i^n x_{\ell})^2) 
-\frac{k}{n} \sum_{i=1}^n
((\Delta_i^n \widehat x_{\ell})^2 -(\Delta_i^n x_{\ell})^2)
\biggr|
\\
&\lesssim \sqrt{n} \max_{1 \le k \le n}
\biggl| \sum_{i=1}^k ((\Delta_i^n \widehat x_{\ell})^2 
-(\Delta_i^n x_{\ell})^2) \biggr|
\\
&=
\sqrt{n} \max_{1 \le k \le n}
\biggl| \sum_{i=1}^k (\Delta_i^n \widehat x_{\ell} -\Delta_i^n x_{\ell}) 
(\Delta_i \widehat x_{\ell} -\Delta_i^n x_{\ell} 
+2 \Delta_i^n x_{\ell}) \biggr|
\\
&\lesssim
\sqrt{n} \sum_{i=1}^n 
(\Delta_i^n \widehat x_{\ell} -\Delta_i^n x_{\ell})^2
+\sqrt{n} \max_{1 \le k \le n}
\biggl| \sum_{i=1}^k (\Delta_i^n \widehat x_{\ell} -\Delta_i^n x_{\ell}) 
\Delta_i^n x_{\ell} \biggr|.
\end{align*}
Let $h_l(y;\kappa) = e_l(y;\kappa) \exp(\kappa y)$ and $h_l(y) = h_l(y;\kappa^*)$. 
Since we have
\begin{align*}
\Delta_i^n \widehat x_{\ell}
&= \sum_{j \in \mathbb M_d} 
\int_{D_j} \Delta_i^n X(y_{j})h_{\ell}(y;\widehat \kappa) \dd y,
\\
\Delta_i^n x_{\ell}
&= \sum_{j \in \mathbb M_d} \int_{D_j} \Delta_i^n X(y) 
h_{\ell}(y) \dd y
\end{align*}
and 
\begin{align*}
\Delta_i^n \widehat x_{\ell} -\Delta_i^n x_{\ell}
&= \sum_{j \in \mathbb M_d} \Delta_i^n X(y_{j}) 
\int_{D_j} (h_{\ell}(y;\widehat \kappa)-h_{\ell}(y;\kappa^*)) \dd y
\\
&\quad+ \sum_{j \in \mathbb M_d} 
\int_{D_j} (\Delta_i^n X(y_{j}) -\Delta_i^n X(y)) 
h_{\ell}(y) \dd y
\\
&=: S_{1,i} +S_{2,i},
\end{align*}
we estimate 
\begin{align*}
\sqrt{n} \sum_{i=1}^n 
(\Delta_i^n \widehat x_{\ell} -\Delta_i^n x_{\ell})^2
\lesssim 
\sqrt{n}  \sum_{i=1}^n (S_{1,i}^2 +S_{2,i}^2).
\end{align*}
We also see from the Schwarz inequality that
\begin{align*}
&\sqrt{n} \max_{1 \le k \le n} \biggl|
\sum_{i=1}^k (\Delta_i^n \widehat x_{\ell} 
-\Delta_i^n x_{\ell}) \Delta_i^n x_{\ell}
\biggr|
\\
&\le
\sqrt{n} \max_{1 \le k \le n} 
\biggl| \sum_{i=1}^k S_{1,i} \Delta_i^n x_{\ell} \biggr|
+ \sqrt{n} \max_{1 \le k \le n} 
\biggl| \sum_{i=1}^k S_{2,i} \Delta_i^n x_{\ell} \biggr|
\\ 
&\le \sqrt{n \sum_{i=1}^n S_{1,i}^2 \sum_{i=1}^n (\Delta_i^n x_{\ell})^2}
+ \sqrt{n} \max_{1 \le k \le n} 
\biggl| \sum_{i=1}^k S_{2,i} \Delta_i^n x_{\ell} \biggr|.
\end{align*}
Note that 
$\displaystyle \sum_{i=1}^n (\Delta_i^n x_{\ell})^2= \Op(1)$.
Therefore, it suffices to show
\begin{equation}\label{prop2-pf-1}
n \sum_{i=1}^n S_{1,i}^2 \pto 0,
\end{equation}
\begin{equation}\label{prop2-pf-2}
\sqrt{n} \sum_{i=1}^n S_{2,i}^2 \pto 0,
\end{equation}
\begin{equation}\label{prop2-pf-3}
\sqrt{n} \sum_{i=1}^n 
\bigl| \EE[S_{2,i}\Delta_i^n x_{\ell} | \GG_{i-1} ] \bigr| \pto 0,
\end{equation}
\begin{equation}\label{prop2-pf-4}
n \sum_{i=1}^n \EE[(S_{2,i} \Delta_i^n x_{\ell})^2 | \GG_{i-1} ] \pto 0
\end{equation}
in order to obtain the desired result by Lemma \ref{lem3} below.

\textbf{Step 1:}
We show \eqref{prop2-pf-1}. 
By the Taylor expansion, we have
\begin{align*}
S_{1,i}^2 
&\le \sum_{j \in \mathbb M_d} (\Delta_i^n X(y_{j}))^2 
\int_{D_j} 
(h_{\ell}(y;\widehat \kappa)-h_{\ell}(y;\kappa^*))^2 \dd y
\\
&\le \sum_{j \in \mathbb M_d} (\Delta_i^n X(y_{j}))^2 
\int_{D_j} 
\biggl| 
\int_0^1 \pd_{\kappa} h_{\ell}(y;\kappa^* +u(\widehat \kappa -\kappa^*)) \dd u
\biggr|^2 |\widehat \kappa -\kappa^*|^2 \dd y.
\end{align*}
For some $\epsilon_1 > 0$, we obtain
\begin{equation*}
\sqrt{n} \sum_{i=1}^n S_{1,i}^2 \le K_{n} |R(\widehat \kappa -\kappa^*)|^2,
\quad
K_n \lesssim \frac{\sqrt{n}}{R^2} 
\sum_{i=1}^n \sum_{j \in \mathbb M_d} (\Delta_i^n X(y_{j}))^2
\int_{D_j} \dd y
\end{equation*}
on $\Omega_1 = \{ |\widehat \kappa -\kappa^*| < \epsilon_1 \}$.
Therefore, we have by Lemmas \ref{lem5} and \ref{lem6} below, 
\begin{align*}
\EE[(\Delta_i^n X(y))^2]
&= \sum_{l_1,l_2 \in \mathbb N^d} 
\EE[ \Delta_i^n x_{l_1} \Delta_i^n x_{l_2} ] e_{l_1}(y) e_{l_2}(y)
\\
&\le \sum_{l \in \mathbb N^d} \EE[ (\Delta_i^n x_l)^2 ] e_l(y)^2
+ \biggl( \sum_{l \in \mathbb N^d} \EE[ \Delta_i^n x_l ] e_l(y) \biggr)^2
\\
&\lesssim 
\sum_{l \in \mathbb N^d} \frac{1 -\ee^{-\lambda_l \Delta_n}}
{\lambda_l^{1+\alpha}} e_l(y)^2
\\
&= \OO( \Delta_n^{(1+\alpha -d/2) \tand 1})
\end{align*}
uniformly in $y \in \overline D$, and see from [A2] and [A3] that 
for any $\epsilon_2 > 0$,
\begin{align*}
\PP(|K_n| > \epsilon_2) 
&\le \PP( \{ |K_n| > \epsilon_2 \} \cap \Omega_1) + \PP(\Omega_1^\mathsf{c})
\\
&\lesssim 
\frac{\sqrt{n}}{\epsilon_2 R^2} 
\sum_{i=1}^n \sum_{j \in \mathbb M_d} \EE \bigl[(\Delta_i^n X(y_{j}))^2 \bigr]
\int_{D_j} \dd y
+ \PP(\Omega_1^\mathsf{c})
\\
&=
\OO \biggl( \frac{n^{3/2} \Delta_n^{(1+\alpha -d/2) \tand 1}}{\epsilon_2 R^2}\biggr)
+\oo(1)
\\
&= \oo(1),
\end{align*}
which together with [A2] yields \eqref{prop2-pf-1}.

\textbf{Step 2:}
We show \eqref{prop2-pf-2}.
By Lemmas \ref{lem5} and \ref{lem6} below, we find that 
for $\alpha \in [0,\infty) \cap (d/2 -1, \infty)$,
\begin{align*}
\EE[(\Delta_i^n X(y) -\Delta_i^n X(z))^2] 
&= \sum_{l_1,l_2 \in \mathbb N^d} 
\EE[ \Delta_i^n x_{l_1} \Delta_i^n x_{l_2} ] 
(e_{l_1}(y) -e_{l_1}(z))(e_{l_2}(y) -e_{l_2}(z))
\\
&\lesssim 
\sum_{l \in \mathbb N^d} \frac{1 -\ee^{-\lambda_l \Delta_n}}{\lambda_l^{1+\alpha}}
(e_l(y) -e_l(z))^2
\\
&= 
\OO \biggl(
\frac{1}{M_{(1)}^{(2(1+\alpha)-d)\widetilde\land 2}}
\biggr) 
\end{align*}
uniformly in $y = (y^{(1)},\ldots,y^{(d)})$, $z = (z^{(1)},\ldots,z^{(d)})$ 
with $|y^{(k)} -z^{(k)}| \le 1/M_k$, $k \in \{ 1, \ldots, d \}$.
Hence, we obtain by [A3], 
\begin{align*}
\sqrt{n} \sum_{i=1}^n S_{2,i}^2
&\lesssim 
\sqrt{n} \sum_{i=1}^n \sum_{j \in \mathbb M_d} \int_{D_j} 
(\Delta_i^n X(y_{j}) -\Delta_i^n X(y))^2
\dd y
\\
&=\Op \biggl(
\frac{n^{3/2}}{M_{(1)}^{(2(1+\alpha)-d)\widetilde\land 2}}
\biggr) 
\\
&= \op(1).
\end{align*}

\textbf{Step 3:}
We show \eqref{prop2-pf-3}. Since
\begin{align*}
S_{2,i} \Delta_i^n x_{\ell}
&= \sum_{j \in \mathbb M_d} 
\int_{D_j} (\Delta_i^n X(y_{j}) -\Delta_i^n X(y)) h_{\ell}(y) \dd y 
\Delta_i^n x_{\ell}
\\
&= \sum_{j \in \mathbb M_d} 
\int_{D_j} \sum_{l \in \mathbb N^d}
(\Delta_i^n x_{\ell})(\Delta_i^n x_l)
(e_l(y_{j}) -e_l(y)) h_{\ell}(y) \dd y,
\end{align*}
$\EE[ \Delta_i^n x_l | \GG_{i-1} ] = \eta_{i,l}$ and 
$\EE[ (\Delta_i^n x_{l_1})(\Delta_i^n x_{l_2}) | \GG_{i-1} ]
= \EE[\zeta_{i,l_1}^2] \ind_{\{ l_1 = l_2 \}} +\eta_{i,l_1} \eta_{i,l_2}$,
it holds that
\begin{align*}
&\bigl| \EE[S_{2,i} \Delta_i^n x_{\ell} | \GG_{i-1} ] \bigr|
\\
&= \Biggl| \sum_{j \in \mathbb M_d} 
\int_{D_j} \sum_{l \in \mathbb N^d}
\EE[(\Delta_i^n x_{\ell})(\Delta_i^n x_l) | \GG_{i-1} ]
(e_l(y_{j}) -e_l(y)) h_{\ell}(y) \dd y 
\Biggr|
\\
&=
\Biggl|
\sum_{j \in \mathbb M_d} 
\int_{D_j} 
\biggl(
\EE[\zeta_{i,{\ell}}^2] 
(e_{\ell}(y_{j}) -e_{\ell}(y)) e_{\ell}(y)
+\eta_{i,\ell} \sum_{l \in \mathbb N^d} 
\eta_{i,l}(e_l(y_{j}) -e_l(y)) h_{\ell}(y)
\biggr) \dd y 
\Biggr|
\\
&\lesssim 
\Delta_n \sum_{j \in \mathbb M_d} \int_{D_j} 
|e_{\ell}(y_{j}) -e_{\ell}(y)| \dd y 
\\
&\quad+
\Biggl| 
\sum_{j \in \mathbb M_d} \int_{D_j} \eta_{i,\ell}
\sum_{l \in \mathbb N^d} \eta_{i,l} (e_l(y_{j}) -e_l(y)) h_{\ell}(y) \dd y 
\Biggr|
\\
&=: \OO \biggl( \frac{\Delta_n}{M_{(1)}} \biggr) +Q_i.
\end{align*}
It follows from [A3] that
\begin{equation*}
\sqrt{n} \sum_{i=1}^n \frac{\Delta_n}{M_{(1)}} = \frac{\sqrt{n}}{M_{(1)}} \to 0.
\end{equation*}
Since it also follows from the Schwarz inequality that
\begin{equation*}
\EE \Biggl[ \biggl( \sqrt{n} \sum_{i=1}^n Q_i \biggr)^2 \Biggr]
\le n^2 \sum_{i=1}^n \EE[Q_i^2],
\end{equation*}
it suffices to show
\begin{equation*}
n^2 \sum_{i=1}^n \EE[Q_i^2] \to 0.
\end{equation*}
Since it holds that
\begin{align*}
\EE[Q_i^2] &=
\EE \Biggl[ \biggl( 
\sum_{j \in \mathbb M_d} \int_{D_j} \eta_{i,\ell}
\sum_{l \in \mathbb N^d} \eta_{i,l} (e_l(y_{j}) -e_l(y)) h_{\ell}(y) \dd y 
\biggr)^2 \Biggr] 
\\
&\le
\EE \Biggl[ 
\sum_{j \in \mathbb M_d} \int_{D_j} \eta_{i,\ell}^2
\biggl(
\sum_{l \in \mathbb N^d} \eta_{i,l}
(e_l(y_{j}) -e_l(y)) h_{\ell}(y) \biggr)^2 \dd y \Biggr] 
\\
&=
\sum_{j \in \mathbb M_d} \int_{D_j}
\sum_{l_1, l_2 \in \mathbb N^d}
\EE [ \eta_{i,\ell}^2 \eta_{i,l_1} \eta_{i,l_2} ]
\\
&\qquad \times 
(e_{l_1}(y_{j}) -e_{l_1}(y)) 
(e_{l_2}(y_{j}) -e_{l_2}(y)) h_{\ell}(y)^2 \dd y
\end{align*}
and
\begin{equation*}
\EE [ \eta_{i,\ell}^2 \eta_{i,l_1} \eta_{i,l_2} ] = 
\begin{cases}
\EE[\eta_{i,\ell}^4], & l_1 = l_2 = \ell,
\\
\EE[\eta_{i,\ell}^2] \EE[\eta_{i,l_1}^2], & l_1 = l_2 \neq \ell,
\\
\EE[\eta_{i,\ell}^3] \EE[\eta_{i,l_2}], & l_1 = \ell \neq l_2,
\\
\EE[\eta_{i,\ell}^3] \EE[\eta_{i,l_1}], & l_2 = \ell \neq l_1,
\\
\EE[\eta_{i,\ell}^2] \EE[\eta_{i,l_1}] \EE[\eta_{i,l_2}], & \text{otherwise},
\end{cases}
\end{equation*}
we obtain by Lemmas \ref{lem4} and \ref{lem6} below, and [A3],
\begin{align*}
\EE[Q_i^2] &\le
\sum_{j \in \mathbb M_d} \int_{D_j}
\sum_{l_1, l_2 \in \mathbb N^d}
\EE [ \eta_{i,\ell}^2 \eta_{i,l_1} \eta_{i,l_2} ]
\\
&\qquad \times 
(e_{l_1}(y_{j}) -e_{l_1}(y))
(e_{l_2}(y_{j}) -e_{l_2}(y)) h_{\ell}(y)^2 \dd y
\\
&\lesssim
\EE [ \eta_{i,\ell}^4 ] 
\sum_{j \in \mathbb M_d} \int_{D_j} (e_{\ell}(y_{j}) -e_{\ell}(y))^2 h_{\ell}(y)^2 \dd y
\\
&\qquad+ 
\bigl| \EE [ \eta_{i,\ell}^3] \bigr|
\sum_{j \in \mathbb M_d} \int_{D_j}
\bigl| (e_{\ell}(y_{j}) -e_{\ell}(y)) h_{\ell}(y) \bigr|
\\
& \qquad \qquad \times
\biggl| \sum_{l \in \mathbb N^d}
\EE[ \eta_{i,l} ] (e_l(y_{j}) -e_l(y)) h_{\ell}(y) 
\biggr| \dd y
\\
&\qquad+
\EE [ \eta_{i,\ell}^2] 
\sum_{j \in \mathbb M_d} \int_{D_j} 
\sum_{l \in \mathbb N^d} \EE[\eta_{i,l}^2 ] 
(e_l(y_{j}) -e_l(y))^2 h_{\ell}(y)^2 \dd y
\\
&\qquad+
\EE [ \eta_{i,\ell}^2] 
\sum_{j \in \mathbb M_d} \int_{D_j}
\biggl( \sum_{l \in \mathbb N^d} \EE[\eta_{i,l} ] 
(e_l(y_{j}) -e_l(y)) h_{\ell}(y)
\biggr)^2
\dd y
\\
&\lesssim
\Delta_n^4 \sum_{j \in \mathbb M_d} \int_{D_j} 
(e_{\ell}(y_{j}) -e_{\ell}(y))^2 \dd y
\\
&\qquad+ \Delta_n^3 \sum_{j \in \mathbb M_d} \int_{D_j}
|e_{\ell}(y_{j}) -e_{\ell}(y)|
\biggl( \sum_{l \in \mathbb N^d} 
\frac{(e_l(y_{j}) -e_l(y))^2}{\lambda_l^{1+\alpha}} \biggr)^{1/2} \dd y
\\
&\qquad + \Delta_n^2 \sum_{j \in \mathbb M_d} \int_{D_j}
\sum_{l \in \mathbb N^d}
\frac{(e_l(y_{j}) -e_l(y))^2}{\lambda_l^{1+\alpha}} \dd y
\\
&= \OO \biggl(\frac{\Delta_n^4}{M_{(1)}^2} \biggr)
+\OO \biggl(
\frac{\Delta_n^3}{M_{(1)}^{1+(1+\alpha-d/2)\widetilde\land 1}}
\biggr) 
+\OO \biggl(
\frac{\Delta_n^2}{M_{(1)}^{(2(1+\alpha)-d)\widetilde\land 2}}
\biggr) 
\\
&= \oo(\Delta_n^3)
\end{align*}
uniformly in $i$.

\textbf{Step 4:}
We show \eqref{prop2-pf-4}. We have
\begin{align*}
\EE[(S_{2,i} \Delta_i^n x_{\ell})^2] 
&\le 
\sum_{j \in \mathbb M_d} 
\int_{D_j} 
\EE \Biggl[
\biggl(
\sum_{l \in \mathbb N^d}
(\Delta_i^n x_{\ell})(\Delta_i^n x_l)
(e_l(y_{j}) -e_l(y)) h_{\ell}(y) 
\biggr)^2
\Biggr] \dd y
\\
&= 
\sum_{j \in \mathbb M_d} \int_{D_j}
\sum_{l_1, l_2 \in \mathbb N^d}
\EE [ (\Delta_i^n x_{\ell})^2 \Delta_i^n x_{l_1} \Delta_i^n x_{l_2} ]
\\
&\qquad \times 
(e_{l_1}(y_{j}) -e_{l_1}(y)) 
(e_{l_2}(y_{j}) -e_{l_2}(y)) h_{\ell}(y)^2 \dd y
\end{align*}
and
\begin{equation*}
\EE [ (\Delta_i^n x_{\ell})^2 \Delta_i^n x_{l_1} \Delta_i^n x_{l_2} ] = 
\begin{cases}
\EE[(\Delta_i^n x_{\ell})^4], & l_1 = l_2 = \ell,
\\
\EE[(\Delta_i^n x_{\ell})^2] \EE[(\Delta_i^n x_{l_1})^2], 
& l_1 = l_2 \neq \ell,
\\
\EE[(\Delta_i^n x_{\ell})^3] \EE[\Delta_i^n x_{l_2}], 
& l_1 = \ell \neq l_2,
\\
\EE[(\Delta_i^n x_{\ell})^3] \EE[\Delta_i^n x_{l_1}], 
& l_2 = \ell \neq l_1,
\\
\EE[(\Delta_i^n x_{\ell})^2] \EE[\Delta_i^n x_{l_1}] \EE[\Delta_i^n x_{l_2}], 
& \text{otherwise}.
\end{cases}
\end{equation*}
By Lammas \ref{lem5} and \ref{lem6} below, and [A3], we obtain
\begin{align*}
\EE[(S_{2,i} \Delta_i^n x_{\ell})^2] 
&\le 
\sum_{j \in \mathbb M_d} \int_{D_j}
\sum_{l_1, l_2 \in \mathbb N^d}
\EE [ (\Delta_i^n x_{\ell})^2 \Delta_i^n x_{l_1} \Delta_i^n x_{l_2} ]
\\
&\qquad \times 
(e_{l_1}(y_{j}) -e_{l_1}(y)) 
(e_{l_2}(y_{j}) -e_{l_2}(y)) e_{\ell}(y)^2 \dd y
\\
&\lesssim
\EE [ (\Delta_i^n x_{\ell})^4 ]
\sum_{j \in \mathbb M_d} \int_{D_j} 
(e_{\ell}(y_{j}) -e_{\ell}(y))^2 h_{\ell}(y)^2 \dd y
\\
&\qquad+
\bigl| \EE [ (\Delta_i^n x_{\ell})^3] \bigr|
\sum_{j \in \mathbb M_d} \int_{D_j}
\bigl| (e_{\ell}(y_{j}) -e_{\ell}(y)) h_{\ell}(y) \bigr|
\\
& \qquad \qquad \times
\biggl| \sum_{l \in \mathbb N^d}
\EE[\Delta_i^n x_{\ell}] (e_l(y_{j}) -e_l(y)) 
h_{\ell}(y) \biggr| \dd y
\\
&\qquad+
\EE [ (\Delta_i^n x_{\ell})^2] 
\sum_{j \in \mathbb M_d} \int_{D_j}
\sum_{l \in \mathbb N^d}
\EE[(\Delta_i^n x_l)^2] (e_l(y_{j}) -e_l(y))^2 h_{\ell}(y)^2 \dd y
\\
&\qquad+
\EE [ (\Delta_i^n x_{\ell})^2] 
\sum_{j \in \mathbb M_d} \int_{D_j}
\biggl( \sum_{l \in \mathbb N^d}
\EE[\Delta_i^n x_l]  
(e_l(y_{j}) -e_l(y)) h_{\ell}(y)
\biggr)^2
\dd y
\\
&\lesssim
\Delta_n^2 \sum_{j \in \mathbb M_d} \int_{D_j} (e_{\ell}(y_{j}) -e_{\ell}(y))^2 \dd y
\\
&\qquad+ \Delta_n^2 \sum_{j \in \mathbb M_d} \int_{D_j}
|e_{\ell}(y_{j}) -e_{\ell}(y)|
\biggl( \sum_{l \in \mathbb N^d} 
\frac{(e_l(y_{j}) -e_l(y))^2}{\lambda_l^{1+\alpha}} \biggr)^{1/2} \dd y
\\
&\qquad + \Delta_n \sum_{j \in \mathbb M_d} \int_{D_j}
\sum_{l \in \mathbb N^d}
\frac{(e_l(y_{j}) -e_l(y))^2}{\lambda_l^{1+\alpha}} \dd y
\\
&= \OO \biggl(\frac{\Delta_n^2}{M_{(1)}^2} \biggr)
+\OO \biggl(
\frac{\Delta_n^2}{M_{(1)}^{1+((2(1+\alpha)-d) \tand 2)/2}}
\biggr) 
+\OO \biggl(
\frac{\Delta_n}{M_{(1)}^{(2(1+\alpha)-d) \tand 2}}
\biggr) 
\\
&= \oo(\Delta_n^2)
\end{align*}
uniformly in $i$, which yields \eqref{prop2-pf-4}. 
\end{proof}
We find from the above proof that 
if the conditions [A1]--[A3] hold true, then
it also follows that $U_n -U_n^* = \op(1)$ under $H_1$.
From the proof of Proposition \ref{prop2}, we obtain the following result.
\begin{prop}\label{prop0}
It follows that under [A1]--[A3], 
\begin{equation*}
\sqrt{n} \biggl( \sum_{i=1}^n (\Delta_i^n \widehat x_{\ell})^2
-\sum_{i=1}^n (\Delta_i^n x_{\ell})^2 \biggr)
= \op(1).
\end{equation*}
\end{prop}
We find that [A3] is a balance condition of $n$, $N$ and $M$ for asymptotic normality 
of the estimators of the coefficient parameters 
based on the approximate coordinate processes suggested by 
\cite{Kaino_Uchida2021a} or \cite{TKU2023b}.

\subsection{Proof of Theorem \ref{th2}}

For any $\epsilon >0$, we have
\begin{equation*}
\PP(T_n \ge \epsilon) = \PP(\widehat \beta_{\ell}^{-2} U_n \ge \epsilon)
\ge \PP(\widehat \beta_{\ell}^{-2} U_n^* \ge 2 \epsilon)
- \PP(\widehat \beta_{\ell}^{-2} |U_n -U_n^*| \ge \epsilon).
\end{equation*}
Since 
$\widehat \beta_{\ell}^2 \pto (\beta_{\ell}^\dag)^2$ under $H_1$ 
and $U_n -U_n^* = \op(1)$, 
we obtain
\begin{align*}
\lim_{n \to \infty} 
\PP(\widehat \beta_{\ell}^{-2} |U_n -U_n^*| \ge \epsilon) &= 0,
\\
\varliminf_{n \to \infty} \PP ( (\beta_{\ell}^\dag)^{-2} U_n^* \ge 4 \epsilon )
&=
\varliminf_{n \to \infty}
\PP \biggl( \{ (\beta_{\ell}^\dag)^{-2} U_n^* \ge 4 \epsilon \} 
\cap 
\biggl\{ \widehat \beta_{\ell}^{-2} \ge 
\frac{(\beta_{\ell}^\dag)^{-2}}{2} 
\biggr\} \biggr)
\\
&\le \varliminf_{n \to \infty}
\PP (\widehat \beta_{\ell}^{-2} U_n^* \ge 2 \epsilon)
\end{align*}
Hence, we show
\begin{equation}\label{prop3}
\varliminf_{n \to \infty} \PP( U_n^* \ge \epsilon ) = 1
\end{equation}
for any $\epsilon >0$ under $H_1$ in order to show Theorem \ref{th2}.

\textit{Proof of \eqref{prop3}.}
Let $\beta_{j,\ell} = \sigma_j^{\dag} \gamma_{\ell}^{-\alpha/2}$ 
for $j \in \{ 1,\ldots, r+1 \}$.
Since it holds that under $H_1$, 
\begin{equation*}
\sum_{i= \lfloor n \tau_{j-1} \rfloor +1}^{\lfloor n \tau_j \rfloor} 
(\Delta_i^n x_{\ell})^2 
\pto \beta_{j,\ell}^2 (\tau_j -\tau_{j-1}),
\quad j \in \{ 1, \ldots, r +1 \},
\end{equation*}
we have 
\begin{align}
D_{n,k}^* 
&= \sum_{i=1}^{\lfloor n \tau_k \rfloor} (\Delta_i^n x_{\ell})^2
-\frac{\lfloor n \tau_k \rfloor}{n} \sum_{i=1}^n (\Delta_i^n x_{\ell})^2
\nonumber
\\
&\pto
\sum_{j=1}^k \beta_{j,\ell}^2 (\tau_j -\tau_{j-1})
-\tau_k \sum_{j=1}^{r+1} \beta_{j,\ell}^2 (\tau_j -\tau_{j-1})
\nonumber
\\
&= (1 -\tau_k) \sum_{j=1}^k \beta_{j,\ell}^2 (\tau_j -\tau_{j-1})
-\tau_k \sum_{j = k+1}^{r+1} \beta_{j,\ell}^2 (\tau_j -\tau_{j-1})
=: c_k \in \mathbb R
\label{Dc}
\end{align}
for $k \in \{ 1, \ldots, r \}$.
It can then be expressed as follows.
\begin{equation*}
\begin{pmatrix}
c_1
\\
\vdots
\\
c_r
\end{pmatrix}
= 
\begin{pmatrix}
\tau_1 (1-\tau_1) & -\tau_1(\tau_2 -\tau_1) & \cdots 
& -\tau_1(\tau_r -\tau_{r-1}) & -\tau_1 (1-\tau_r) 
\\
\tau_1 (1-\tau_2) & (1-\tau_2)(\tau_2 -\tau_1) &  
& -\tau_2(\tau_r -\tau_{r-1}) & -\tau_2 (1-\tau_r) 
\\
\vdots & & \ddots & & \vdots 
\\
\tau_1(1-\tau_r) & (1-\tau_r)(\tau_2 -\tau_1) & \cdots 
& (1-\tau_r)(\tau_r -\tau_{r-1}) & -\tau_r (1-\tau_r) 
\end{pmatrix}
\begin{pmatrix}
\beta_{1,\ell}^2
\\
\vdots
\\
\beta_{r+1,\ell}^2
\end{pmatrix}.
\end{equation*}
By elementary row operations, we have
\begin{align*}
&
\begin{pmatrix}
\tau_1 (1-\tau_1) & -\tau_1(\tau_2 -\tau_1) & \cdots 
& -\tau_1(\tau_r -\tau_{r-1}) & -\tau_1 (1-\tau_r) 
\\
\tau_1 (1-\tau_2) & (1-\tau_2)(\tau_2 -\tau_1) &  
& -\tau_2(\tau_r -\tau_{r-1}) & -\tau_2 (1-\tau_r) 
\\
\vdots & \vdots & \ddots & & \vdots 
\\
\tau_1(1-\tau_r) & (1-\tau_r)(\tau_2 -\tau_1) & \cdots 
& (1-\tau_r)(\tau_r -\tau_{r-1}) & -\tau_r (1-\tau_r) 
\end{pmatrix}
\\
&\to
\begin{pmatrix}
1-\tau_1 & -(\tau_2 -\tau_1) & \cdots 
& -(\tau_r -\tau_{r-1}) & -(1-\tau_r) 
\\
\tau_1 (1-\tau_2) & (1-\tau_2)(\tau_2 -\tau_1) &  
& -\tau_2(\tau_r -\tau_{r-1}) & -\tau_2 (1-\tau_r) 
\\
\vdots & \vdots &  & \vdots & \vdots 
\\
\tau_1 (1-\tau_{r-1}) & (1-\tau_{r-1})(\tau_2 -\tau_1) &  
& -\tau_{r-1}(\tau_r -\tau_{r-1}) & -\tau_{r-1} (1-\tau_r) 
\\
\tau_1 & \tau_2 -\tau_1 & \cdots 
& \tau_r -\tau_{r-1} & -\tau_r 
\end{pmatrix}
\\
&\to
\begin{pmatrix}
1 & 0 & \cdots & 0 & -1 
\\
\tau_1 & \tau_2 -\tau_1 &  & 0 & -\tau_2
\\
\vdots & \vdots & \ddots & & \vdots 
\\
\tau_1 & \tau_2 -\tau_1 & \cdots 
& \tau_r -\tau_{r-1} & -\tau_r 
\end{pmatrix}
\\
&\to
\begin{pmatrix}
1 & 0 & \cdots & 0 & -1 
\\
0 & \tau_2 -\tau_1 &  & 0 & -(\tau_2 -\tau_1)
\\
\vdots & \vdots & \ddots & & \vdots 
\\
0 & \tau_2 -\tau_1 & \cdots 
& \tau_r -\tau_{r-1} & -(\tau_r -\tau_1) 
\end{pmatrix}
\\
&\to
\begin{pmatrix}
1 & 0 & 0 & \cdots & 0 & -1 
\\
0 & \tau_2 -\tau_1 & 0 & & 0 & -(\tau_2 -\tau_1)
\\
0 & 0 & \tau_3 -\tau_2 &  & 0 & -(\tau_3 -\tau_2)
\\
\vdots & \vdots & \vdots & \ddots & & \vdots 
\\
0 & 0 & \tau_3 -\tau_2 & & \tau_r -\tau_{r-1} & -(\tau_r -\tau_2) 
\end{pmatrix}
\\
&\to \cdots
\\
&\to
\begin{pmatrix}
1 & 0 & \cdots & 0 & -1 
\\
0 & \tau_2 -\tau_1 &  & 0 & -(\tau_2 -\tau_1)
\\
\vdots & & \ddots & & \vdots 
\\
0 & 0 & & \tau_r -\tau_{r-1} & -(\tau_r -\tau_{r-1}) 
\end{pmatrix}
\\
&\to
\begin{pmatrix}
I_{r} & -1_{r}
\end{pmatrix},
\end{align*}
where $I_r$ is the $r$-dimensional identity matrix. 
Since 
$(c_1, \ldots, c_r)^\TT =0$ if and only if 
$\beta_{1,\ell}^2 = \cdots = \beta_{r+1,\ell}^2$, 
there exists $q \in \{ 1, \ldots, r \}$ such that $c_q \neq 0$
under $\sigma_j^{\dag} \neq \sigma_{j+1}^{\dag}$ for $j \in \{ 1, \ldots, r \}$. 
Therefore, we obtain by \eqref{Dc}, 
\begin{align*}
\PP(U_n^* < \epsilon) 
&\le \PP \biggl( \sqrt{\frac{n}{2}} |D_{n,q}^*| < \epsilon \biggr)
\\
&\le \PP \biggl( \biggl\{ \sqrt{\frac{n}{2}} |D_{n,q}^*| < \epsilon \biggr\} \cap 
\biggl\{ |D_{n,q}^* -c_q| < \frac{|c_q|}{2} \biggr\} \biggr) 
+\PP \biggl( |D_{n,q}^* -c_q| \ge \frac{|c_q|}{2} \biggr)
\\
&\le \PP \biggl(\biggl\{ \sqrt{\frac{n}{2}} |D_{n,q}^*| < \epsilon \biggr\} \cap 
\biggl\{ \frac{|c_q|}{2} < |D_{n,q}^*| \biggr\} \biggr) 
+\PP \biggl( |D_{n,q}^* -c_q| \ge \frac{|c_q|}{2} \biggr)
\\
&\le \PP \biggl( \sqrt{\frac{n}{8}}|c_q| < \epsilon \biggr) 
+\PP \biggl( |D_{n,q}^* -c_q| \ge \frac{|c_q|}{2} \biggr)
\\
&\to 0
\end{align*}
for any $\epsilon >0$, and this concludes the proof.

\subsection{Auxiliary results}

\begin{lem}\label{lem3}
Let $\left\{ \chi_{i,n} \right\}_{i \in \mathbb{N}}$ be $\FF_i$-measurable random variables.
For a positive integer sequence $\{ m_n \}_{n \in \mathbb N}$ such that 
$m_n \to \infty$, if 
\begin{equation*}
\sum_{i=1}^{m_n} |\EE[\chi_{i,n}|\FF_{i-1}]| \pto 0, 
\quad
\sum_{i=1}^{m_n} \EE[\chi_{i,n}^2|\FF_{i-1}] \pto 0, 
\end{equation*}
then
\begin{equation*}
\max_{1 \le k \le m_n} \biggl| \sum_{i=1}^k \chi_{i,n} \biggr| \pto 0.
\end{equation*}
\end{lem}
\begin{proof}
The essence of the proof is same as Lemma 9 in \cite{Genon-Catalot_Jacod1993}.
Let 
\begin{align*}
\xi_{i,n} &= \chi_{i,n} -\EE[\chi_{i,n} | \FF_{i-1}], 
\quad
B_{k,n} = \sum_{i=1}^{k} \xi_{i,n},
\\
C_{k,n} &= \sum_{i=1}^{k} \EE[\xi_{i,n}^2 | \FF_{i-1}],
\quad
D_{k,n} = \sum_{i=1}^{k} \EE[\chi_{i,n}^2 | \FF_{i-1}].
\end{align*}
It follows from the Lenglart inequality
(see e.g., (4.15)$'$ in \cite{Karatzas_Shrve2012}) that for any $a,b > 0$, 
\begin{equation*}
\PP \biggl( \max_{1 \le k \le m_n} B_{k,n}^2 > a \biggr) 
\le \frac{b}{a} + \PP(C_{m_n,n} > b) \le \frac{b}{a} + \PP(D_{m_n,n} > b).
\end{equation*}
Since $\PP(D_{m_n,n} > b) \to 0$, we get
\begin{equation*}
\PP \biggl( \max_{1 \le k \le m_n} B_{k,n}^2 > a \biggr) \to 0,
\end{equation*}
which together with 
\begin{equation*}
\max_{1 \le k \le m_n} \biggl| \sum_{i=1}^k \chi_{i,n} \biggr|
\le \max_{1 \le k \le m_n} |B_{k,n}| + \sum_{i=1}^{m_n} |\EE[\chi_{i,n}|\FF_{i-1}]|
\end{equation*}
yields the desired result.
\end{proof}

Noting that $\eta_{i,l} = (\ee^{-\lambda_l \Delta_n} -1) x_l(t_{i-1}^n)$ and 
\begin{equation*}
x_l(t) = \ee^{-\lambda t} x_l(0) + \int_0^t \beta_l(s) \ee^{-\lambda_l(t -s)} \dd w_l(s),
\end{equation*}
we have the following result.
\begin{lem}\label{lem4}
Let $\{ g_l \}_{l \in \mathbb N^d}$ be measurable functions on $D$.
\begin{enumerate}
\item[(1)]
For a fixed $l \in \mathbb N^d$ and $k \in \{1,\ldots,4\}$, it follows that
\begin{equation*}
\EE[\eta_{i,l}^k] \lesssim \Delta_n^k
\end{equation*}
uniformly in $i$ under $\EE[x_l(0)^4] < \infty$.

\item[(2)]
Under [A1]-(i), it follows that 
\begin{equation*}
\sum_{l \in \mathbb N^d} \EE[\eta_{i,l}] g_l(y)
\lesssim 
\biggl(
\sum_{l \in \mathbb N^d} \frac{(1 -\ee^{-\lambda_l \Delta_n})^2}
{\lambda_l^{1+\alpha}} g_l(y)^2
\biggr)^{1/2}.
\end{equation*}

\item[(3)]
Under [A1]-(i), it follows that
\begin{equation*}
\sum_{l \in \mathbb N^d} \EE[\eta_{i,l}^2] g_l(y)^2
\lesssim 
\sum_{l \in \mathbb N^d} \frac{(1 -\ee^{-\lambda_l \Delta_n})^2}
{\lambda_l^{1+\alpha}} g_l(y)^2.
\end{equation*}
\end{enumerate}
\end{lem}

\begin{proof}
Let
$Z_l(t) =  \int_0^t \beta_l(s) \ee^{-\lambda_l(t -s)} \dd w_l(s)$.
One has $x_l(t) =  \ee^{-\lambda t} x_l(0) +Z_l(t).$ 
\begin{enumerate}
\item[(1)] 
Since it holds that
\begin{align*}
\EE[x_l(t)] &= \ee^{-\lambda_l t} \EE[x_l(0)],
\\
\EE[x_l(t)^2] &= \ee^{-2\lambda_l t} \EE[x_l(0)^2] + \EE[Z_l(t)^2],
\\
\EE[x_l(t)^3] &= \ee^{-3\lambda_l t} \EE[x_l(0)^3] 
+ 3 \ee^{-\lambda_l t} \EE[x_l(0)] \EE[Z_l(t)^2],
\\
\EE[x_l(t)^4] &= \ee^{-4\lambda_l t} \EE[x_l(0)^4] 
+ 6 \ee^{-2\lambda_l t} \EE[x_l(0)^2] \EE[Z_l(t)^2] + \EE[Z_l(t)^4]
\end{align*}
and 
\begin{equation*}
Z_l(t) \sim 
\mathrm{N} \biggl(0, \int_0^t \beta_l(s)^2 \ee^{-2\lambda_l (t-s)} \dd s \biggr),
\quad
\int_0^t \beta_l(s)^2 \ee^{-2\lambda_l (t-s)} \dd s 
\lesssim \frac{1-\ee^{-2\lambda_l t}}{\lambda_l^{1+\alpha}},
\end{equation*}
we have $\sup_{t \in [0,1]} \EE[x_l(t)^k] \lesssim 1$ under $\EE[x_l(0)^4] < \infty$,
which together with $1 -\ee^{-\lambda_l \Delta_n} \lesssim \lambda_l \Delta_n$
yields the desired result.

\item[(2)] 
It is obvious that the inequality is valid 
since $\EE[ \eta_{i,l} ] = 0$ under [A1]-(i)-a). 
It follows from the Schwartz inequality that under [A1]-(i)-b),
\begin{align*}
\biggl( \sum_{l \in \mathbb N^d} \EE[ \eta_{i,l} ] g_l(y) \biggr)^2
&= \EE \biggl[ \sum_{l \in \mathbb N^d} 
\lambda_l^{(1+\alpha)/2} x_l(0) \times  
\frac{\ee^{-\lambda_l \Delta_n} -1}{\lambda_l^{(1+\alpha)/2}} 
\ee^{-\lambda_l (i-1) \Delta_n} g_l(y)
\biggr]^2
\\
&\le \EE \biggl[ \sum_{l \in \mathbb N^d} \lambda_l^{1+\alpha} x_l(0)^2 \biggr]
\sum_{l \in \mathbb N^d} 
\frac{(1 -\ee^{-\lambda_l \Delta_n})^2}{\lambda_l^{1+\alpha}} g_l(y)^2 
\\
&= \EE \bigl[ \| A_\theta^{(1+\alpha)/2} X_0 \|^2 \bigr]
\sum_{l \in \mathbb N^d} 
\frac{(1 -\ee^{-\lambda_l \Delta_n})^2}{\lambda_l^{1+\alpha}} g_l(y)^2 
\\
&\lesssim 
\sum_{l \in \mathbb N^d} 
\frac{(1 -\ee^{-\lambda_l \Delta_n})^2}{\lambda_l^{1+\alpha}} g_l(y)^2.
\end{align*}

\item[(3)] 
Since $\EE[x_l(0)^2] \lesssim \lambda_l^{-(1+\alpha)}$ under [A1]-(i), 
it holds that
\begin{align*}
\sum_{l \in \mathbb N^d} \EE[ \eta_{i,l}^2 ] g_l(y)^2
&\le 
\sum_{l \in \mathbb N^d} 
(1 -\ee^{-\lambda_l \Delta_n})^2 
\bigl( \EE[x_l(0)^2] + \EE[Z_l(t_{i-1}^n)^2] \bigr) g_l(y)^2
\\
&\lesssim 
\sum_{l \in \mathbb N^d} 
\frac{(1 -\ee^{-\lambda_l \Delta_n})^2}{\lambda_l^{1+\alpha}} g_l(y)^2.
\end{align*}
\end{enumerate}
\end{proof}

\begin{lem}\label{lem5}
Let $\{ g_l \}_{l \in \mathbb N^d}$ be measurable functions on $D$.
\begin{enumerate}
\item[(1)]
For a fixed $l \in \mathbb N^d$, it follows that
\begin{equation*}
\EE[(\Delta_i^n x_l)^k] \lesssim 
\begin{cases}
\Delta_n, & k \in \{ 1, 2 \}, 
\\
\Delta_n^2, & k \in \{ 3, 4 \}  
\end{cases}
\end{equation*}
uniformly in $i$ under $\EE[x_l(0)^4] < \infty$.

\item[(2)]
Under [A1]-(i), it follows that 
\begin{equation*}
\sum_{l \in \mathbb N^d} \EE[\Delta_i^n x_l] g_l(y)
\lesssim 
\biggl(
\sum_{l \in \mathbb N^d} \frac{(1 -\ee^{-\lambda_l \Delta_n})^2}
{\lambda_l^{1+\alpha}} g_l(y)^2
\biggr)^{1/2}.
\end{equation*}

\item[(3)]
Under [A1]-(i), it follows that 
\begin{equation*}
\sum_{l \in \mathbb N^d} \EE[(\Delta_i^n x_l)^2] g_l(y)^2
\lesssim 
\sum_{l \in \mathbb N^d} \frac{1 -\ee^{-\lambda_l \Delta_n}}
{\lambda_l^{1+\alpha}} g_l(y)^2.
\end{equation*}
\end{enumerate}
\end{lem}
\begin{proof}
Note that $\Delta_i^n x_l = \zeta_{i,l} +\eta_{i,l}$
and $\zeta_{i,l} = \int_{t_{i-1}^n}^{t_i^n} 
\beta_l(s) \ee^{-\lambda_l(t_i^n-s)} \dd w_l(s)$.
\begin{enumerate}
\item[(1)] 
Since it holds that
\begin{align*}
\EE[\Delta_i^n x_l] &= \EE[\eta_{i,l}],
\\
\EE[(\Delta_i^n x_l)^2] &= \EE[\eta_{i,l}^2] +\EE[\zeta_{i,l}^2],
\\
\EE[(\Delta_i^n x_l)^3] &= \EE[\eta_{i,l}^3] +3 \EE[\eta_{i,l}] \EE[\zeta_{i,l}^2],
\\
\EE[(\Delta_i^n x_l)^4] &= \EE[\eta_{i,l}^4] 
+6 \EE[\eta_{i,l}^2] \EE[\zeta_{i,l}^2] +\EE[\zeta_{i,l}^4]
\end{align*}
and
\begin{equation*}
\zeta_{i,l} \sim 
\mathrm{N} \biggl(0, \int_{t_{i-1}^n}^{t_i^n} \beta_l(s)^2 
\ee^{-2\lambda_l (t_i^n -s)} \dd s \biggr),
\quad
\int_{t_{i-1}^n}^{t_i^n} \beta_l(s)^2 \ee^{-2\lambda_l (t_i^n -s)} \dd s
\lesssim \frac{1-\ee^{-2\lambda_l \Delta_n}}{\lambda_l^{1+\alpha}},
\end{equation*}
we obtain the desire results under $\EE[x_l(0)^4] < \infty$ 
from Lemma \ref{lem4} and 
$\frac{1 -\ee^{-\lambda_l \Delta_n}}{\lambda_l} \lesssim \Delta_n$.

\item[(2)]
It can be seen from $\EE[\Delta_i^n x_l] = \EE[\eta_{i,l}]$ and Lemma \ref{lem4}.

\item[(3)]
By $\EE[\zeta_{i,l}^2] \lesssim \frac{1-\ee^{-\lambda_l \Delta_n}}{\lambda_l^{1+\alpha}}$
and Lemma \ref{lem4},
we have the result.

\end{enumerate}
\end{proof}

For $\alpha \in [0, \infty) \cap (d/2-1, \infty)$ and $q \in [0,\infty)$, we define
\begin{align*}
\mathcal S_{1,n}(\alpha, q) =
\sum_{l \in \mathbb N^d} \frac{(1-\ee^{-\lambda_l \Delta_n})^q}{\lambda_l^{1+\alpha}},
\quad
\mathcal S_{2,M}(\alpha) 
= \max_{j \in \mathbb M_d} \sup_{y,z \in D_j}
\sum_{l \in \mathbb N^d} \frac{(e_l(y) -e_l(z))^2}{\lambda_l^{1+\alpha}}.
\end{align*}
\begin{lem}\label{lem6}
For $\alpha \in [0, \infty) \cap (d/2-1, \infty)$ and $q \in [0,\infty)$, we have
\begin{equation*}
\mathcal S_{1,n}(\alpha, q) 
= \OO (\Delta_n^{(d/2 -(1+\alpha)) \tand q} ),
\quad
\mathcal S_{2,M}(\alpha) 
= \OO \biggl( \frac{1}{M_{(1)}^{(2(1+\alpha) -d) \tand 2}} \biggr).
\end{equation*}
\end{lem}
\begin{proof}
Note that for $L \in (1,\infty)$, 
\begin{equation*}
\sum_{l \in \mathbb N^d:|l|_{2} < L} \frac{1}{|l|_{2}^\beta} = 
\begin{cases}
\OO(L^{d-\beta}), & \beta < d,
\\
\OO(\log(L)), & \beta = d,
\\
\OO(1), & \beta > d,
\end{cases}
\quad
\sum_{l \in \mathbb N^d:|l|_{2} \ge L} \frac{1}{|l|_{2}^\beta} = 
\OO(L^{d -\beta}), \quad \beta > d.
\end{equation*}
In particular, we have
\begin{align*}
\frac{1}{L^{2q}} \sum_{l \in \mathbb N^d:|l|_{2} < L} 
\frac{1}{|l|_{2}^{2(1+\alpha-q)}} 
&= L^{-2q} \times
\begin{cases}
\OO(L^{d -2(1 +\alpha -q)}), & \alpha < d/2 -1 +q,
\\
\OO(\log (L)), & \alpha = d/2 -1 +q,
\\
\OO(1), & \alpha > d/2 -1 +q
\end{cases}
\\
&= \OO \biggl(\frac{1}{L^{(d -2(1+\alpha)) \tand 2q}} \biggr).
\end{align*}
Using the relation $1-\ee^{-\lambda_l \Delta_n} \sim 1 \land |l|_{2}^2 \Delta_n$,
we have 
\begin{align*}
\mathcal S_{1,n}(\alpha, q) & \sim 
\sum_{l \in \mathbb N^d: |l|_{2} \ge \Delta_n^{-1/2}} 
\frac{1}{|l|_{2}^{2(1+\alpha)}} + \Delta_n^q 
\sum_{l \in \mathbb N^d:|l|_{2} < \Delta_n^{-1/2}} 
\frac{1}{|l|_{2}^{2(1 +\alpha -q)}}
\\
&= \OO ( \Delta_n^{(d-2(1+\alpha))/2} ) 
+ \OO ( \Delta_n^{((d -2(1+\alpha))/2) \tand q} )
\\
&= \OO (\Delta_n^{(d/2 -(1+\alpha)) \tand q} ).
\end{align*}
Since it follows that
\begin{equation*}
(e_l(y)-e_l(z))^2
\lesssim 1 \land \frac{|l|_{2}^2}{M_{(1)}^2}
\end{equation*}
uniformly in $y, z \in D_j$, $j \in \mathbb M_d$, we have
\begin{align*}
\mathcal S_{2,M}(\alpha)
&\sim \frac{1}{M_{(1)}^2}
\sum_{l \in \mathbb N^d: |l|_{2} < M_{(1)}} \frac{1}{|l|_{2}^{2\alpha}}
+\sum_{l \in \mathbb N^d: |l|_{2} \ge M_{(1)}} \frac{1}{|l|_{2}^{2(1+\alpha)}}
\\
&= \OO \biggl( \frac{1}{M_{(1)}^{(2(1+\alpha) -d) \tand 2}} \biggr)
+\OO \biggl( \frac{1}{M_{(1)}^{2(1+\alpha) -d}} \biggr)
\\
&= \OO \biggl( \frac{1}{M_{(1)}^{(2(1+\alpha) -d) \tand 2}} \biggr).
\end{align*}
\end{proof}


\newpage

\appendix

\section{Appendix: Estimation for linear parabolic SPDEs in one space dimension 
with volatility changes}\label{secA}
We consider parametric estimation for the linear parabolic SPDE 
\begin{equation*}
\dd X_t(y) = 
\biggl( \theta_2 \frac{\pd^2}{\pd y^2} +\theta_1 \frac{\pd}{\pd y} 
+ \theta_0 \biggr) X_t(y) \dd t 
+ \sigma(t) \dd B_t(y), 
\quad(t, y) \in [0,1] \times (0,1)
\end{equation*}
with an initial value $X_0$ and the Dirichlet boundary condition 
$X_t(0) = X_t(1) = 0$, $t \in [0,1]$, where 
$\theta_0, \theta_1 \in \mathbb R$ and $\theta_2 \in (0,\infty)$ are unknown parameters,
the volatility function $\sigma(t)$ is characterized by
\begin{equation}\label{vol}
\sigma(t) = \sum_{p = 1}^{r+1} \sigma_p \ind_{[\tau_{p-1}, \tau_{p})}(t)
\end{equation}
with $r \in \mathbb N$, 
$0 = \tau_0 < \tau_1 < \tau_2 < \cdots < \tau_r < \tau_{r+1} = 1$,
and $\sigma_p \in (0, \infty)$ for $p \in \{ 1,\ldots, r+1 \}$.
For convenience, we write $[\tau_r, \tau_{r+1}) = [\tau_r,1]$.
$\{ B_t \}_{t \ge 0}$ is the cylindrical Brownian motion 
in a Sobolev space on $(0,1)$ and is given by
\begin{equation*}
B_t = \sum_{l \in \mathbb N} w_l(t) e_l
\end{equation*}
with $e_l(y) = \sqrt{2} \exp(-\kappa y/2)\sin(\pi l y)$, 
$\kappa = \theta_1/\theta_2$ and 
independent real valued standard Brownian motions $\{ w_l \}_{l \in \mathbb N}$.

Suppose that we have discrete observations 
$\mathbf X_{M,N} = \{ X_{t_i^N}(y_j) \}_{0 \le i \le N, 0 \le j \le M}$ with
$t_i^N = i \Delta = i/N$ and $y_j = j/M$.
For $b \in (0,1/2)$, $m \in \{ 1, \ldots, M \}$ and 
$n \in \{1, \ldots, N \}$, we will write the thinned data 
obtained from $\mathbf X_{M,N}$ as 
$\mathbf X_{m,n}^{(b)} = \{ X_{t_i^n}(\widetilde y_j) \}_{0 \le i \le n, 0 \le j \le m}$
with 
\begin{equation*}
t_i^n = i \cdot \frac{1}{N} \biggl\lfloor \frac{N}{n} \biggr\rfloor,
\quad
\widetilde y_j = b + j \cdot \frac{1-2b}{m},
\end{equation*}
where $\widetilde y_j \in \{ y_0, \ldots, y_M \}$.

Let $-A_\theta = \theta_2 \frac{\pd^2}{\pd y^2} +\theta_1 \frac{\pd}{\pd y} + \theta_0$. 
While we can consider more general initial conditions such as [A1], 
for simplicity, we make the following condition.
\begin{description}
\item[{[D1]}]
The initial value $X_0 \in L^2((0,1))$ is deterministic and 
$\| A_\theta^{1/2} X_0 \| < \infty$.
\end{description}

\subsection{Estimators proposed by Bibinger and Trabs \cite{Bibinger_Trabs2020}}
Suppose that we have thinned data $\mathbf X_{m,N}^{(b)}$ 
with $b \in (0,1/2)$ and $m = \OO(N^\rho)$ for some $\rho \in (0,1/2)$.
We define
\begin{equation*}
Z_{j,N} = \frac{1}{N\sqrt{\Delta}}\sum_{i=1}^N (\Delta_i^N X(\widetilde y_j))^2
\end{equation*}
and
\begin{equation*}
V_0 = \frac{1}{\sqrt{\theta_2}} \int_0^1 \sigma^2(t) \dd t
= \frac{1}{\sqrt{\theta_2}} \sum_{p=1}^{r+1} \sigma_p^2 (\tau_{p} -\tau_{p-1}).
\end{equation*}

Bibinger and Trabs \cite{Bibinger_Trabs2020} considered linear parabolic SPDEs 
with a H\"{o}lder continuous volatility function. 
For linear parabolic SPDEs with a volatility function expressed by the step function
\eqref{vol}, we obtain the following result analogous to Proposition 6.8 
in \cite{Bibinger_Trabs2020}.
\begin{prop}\label{app_prop1}
Assume that [D1] holds. Then, it holds that for $y \in [b,1-b]$, 
\begin{equation*}
\frac{1}{N \sqrt{\Delta}} \sum_{i=1}^N \EE \bigl[ (\Delta_i^N X(y))^2 \bigr]
= \frac{V_0 \cdot \ee^{-\kappa y}}{\sqrt{\pi}} +\OO(\Delta).
\end{equation*}
\end{prop}

We define the contrast function
\begin{equation*}
U_{m,N}(\kappa, V_0) = \frac{1}{m} \sum_{j=1}^m 
\biggl(
\frac{1}{N \sqrt{\Delta}} \sum_{i=1}^N (\Delta_i^N X(\widetilde y_j))^2
-\frac{V_0 \cdot \ee^{-\kappa \widetilde y_j}}{\sqrt{\pi}}
\biggr)^2.
\end{equation*}
Let $\Xi$ be a compact convex subset of $\mathbb R \times (0,\infty)$ and
we assume that
the true value $(\kappa^*, V_0^*)$ of $(\kappa, V_0)$ belongs to the interior of $\Xi$.
We define the minimum contrast estimator of $(\kappa, V_0)$ as 
\begin{equation*}
(\widehat \kappa, \widehat V_0) = 
\underset{(\kappa, V_0) \in \Xi}{\mathrm{argmin}}\, U_{m,N}(\kappa, V_0).
\end{equation*}
We then obtain the following result.
\begin{thm}\label{app_th1}
Assume that [D1] holds. Then, it holds that 
under $m = \OO(N^\rho)$ for some $\rho \in (0, 1/2)$, 
\begin{equation*}
\sqrt{m N} 
\begin{pmatrix}
\widehat \kappa -\kappa^*
\\
\widehat V_0 -V_0^* 
\end{pmatrix}
= \Op(1).
\end{equation*}
\end{thm}
This result holds under both $H_0$ and $H_1$ of \eqref{HTP}.

\subsection{Estimators proposed by Hildebrandt and Trabs \cite{Hildebrandt_Trabs2021}}
For a sequence $\{a_n\}$, 
we write $a_n \equiv a$ if $a_n = a$ for some $a \in \mathbb R$ and all $n$.

Suppose that we have thinned data $\mathbf X_{m,N}^{(b)}$ 
with $b \in (0,1/2)$, $m = \OO(\sqrt{N})$ and $N = \OO(m^2)$.
We define
\begin{equation*}
D_{i,j} X = \Delta_i^N X(\widetilde y_j) -\Delta_i^N X(\widetilde y_{j-1}),
\end{equation*}
$\overline y_j = (\widetilde y_{j-1} +\widetilde y_j)/2$ and
\begin{equation*}
V = \int_0^1 \sigma^2(t) \dd t
= \sum_{p=1}^{r+1} \sigma_p^2 (\tau_{p} -\tau_{p-1}).
\end{equation*}
Let $\delta = (1-2b)/m$. 
For $\delta/\sqrt{\Delta} \equiv r \in (0,\infty)$, we set
\begin{equation*}
\psi_r(\theta_2) = 
\frac{2}{\sqrt{\pi \theta_2}} 
\biggl(1-\exp \Bigl(-\frac{r^2}{4 \theta_2} \Bigr) +\frac{r}{\sqrt{\theta_2}}
\int_{r/\sqrt{4\theta_2}}^\infty \ee^{-x^2} \dd x
\biggr).
\end{equation*}
We obtain the same result as Proposition 3.5 in \cite{Hildebrandt_Trabs2021} 
for linear parabolic SPDEs with a volatility function expressed by the step function
\eqref{vol}.
\begin{prop}\label{app_prop2}
Assume that [D1] holds. 
Then, it holds that for $\delta/\sqrt{\Delta} \equiv r \in (0,\infty)$, 
\begin{equation*}
\frac{1}{N \sqrt{\Delta}} \sum_{i=1}^N \EE \bigl[ (D_{i,j} X)^2 \bigr]
= V \exp(-\kappa \overline y_j) \psi_{r}(\theta_2) +\OO(\Delta).
\end{equation*}
\end{prop}

Let $\widetilde D_{i,j} X = D_{i,j} X +D_{i+1,j} X$ and $\nu = (\kappa, \theta_2, V)$.
For $\delta/\sqrt{\Delta} \equiv r \in (0,\infty)$, we define the contrast function
\begin{align*}
K_{m,N}(\nu) &= 
\frac{1}{m} \sum_{j=1}^m 
\biggl(
\frac{1}{N \sqrt{\Delta}}
\sum_{i=1}^N (D_{i,j} X)^2 -f_{r}(\overline y_j;\nu)
\biggr)^2
\\
&\qquad+\frac{1}{m} \sum_{j=1}^m 
\biggl(
\frac{1}{N \sqrt{2\Delta}}
\sum_{i=1}^{N-1} (\widetilde D_{i,j} X)^2 -f_{r/\sqrt{2}}(\overline y_j;\nu)
\biggr)^2,
\end{align*}
where $f_{r}(y;\nu) = V \exp(-\kappa y) \psi_{r}(\theta_2)$. 
Let $\Xi$ be a compact convex subset of $\mathbb R \times (0,\infty)^2$ and
we assume that
the true value $\nu^*$ of $\nu$ belongs to the interior of $\Xi$.
We define the minimum contrast estimator of $\nu$ as
\begin{equation*}
\widehat \nu = \underset{\nu \in \Xi}{\mathrm{argmin}}\,K_{m,N}(\nu).
\end{equation*}
We then get the following theorem.
\begin{thm}\label{app_th2}
Assume that [D1] holds. Then, it holds that
\begin{equation*}
\sqrt{m N} (\widehat \nu -\nu^*)= \Op(1).
\end{equation*}
\end{thm}
This result holds under both $H_0$ and $H_1$ of \eqref{HTP}.

\subsection{Proofs}
For simplicity, we assume that for any $j \in \{ 1, \ldots, r \}$, 
there exists $N_j \in \{ 1, \ldots, N \}$ such that $\tau_j = N_j \Delta$.

Since it follows that 
\begin{equation*}
x_l(t) = \ee^{-\lambda_l t} x_l(0)
+\int_0^t \sigma(s) \ee^{-\lambda_l(t-s)} \dd w_l(s),
\end{equation*}
we have
\begin{align*}
\Delta_i^N x_l 
&= -\ee^{-\lambda_l (i-1)\Delta} (1-\ee^{-\lambda_l \Delta})x_l(0)
\\
&\quad-
(1-\ee^{-\lambda_l \Delta}) 
\int_0^{(i-1)\Delta} \sigma(s) \ee^{-\lambda_l((i-1)\Delta -s)} \dd w_l(s)
\\
&\quad+\int_{(i-1)\Delta}^{i \Delta} \sigma(s)
\ee^{-\lambda_l(i \Delta-s)} \dd w_l(s)
\\
&=: A_{i,l} +B_{i,l} +C_{i,l}.
\end{align*}
Let $p \in \{1,\ldots, r+1 \}$ and $v_i = (\sigma(i \Delta))^2$. 
For $N_{p-1} +1 \le i \le N_p$, we have
\begin{align*}
\EE[B_{i,l}^2] 
&= (1-\ee^{-\lambda_l \Delta})^2 
\int_0^{(i-1)\Delta} \sigma(s)^2 \ee^{-2\lambda_l((i-1) \Delta -s)} \dd s
\\
&= (1-\ee^{-\lambda_l \Delta})^2
\int_0^{N_{p-1}\Delta } \sigma(s)^2 \ee^{-2\lambda_l((i-1) \Delta -s)} \dd s
\\
&\quad 
+(1-\ee^{-\lambda_l \Delta})^2 
\int_{N_{p-1}\Delta}^{(i-1)\Delta} \sigma(s)^2 \ee^{-2\lambda_l((i-1) \Delta -s)} \dd s
\\
&= 
(1-\ee^{-\lambda_l \Delta})^2
\sum_{k=1}^{N_{p-1}} 
\int_{(k-1)\Delta}^{k \Delta} v_{k-1} \ee^{-2\lambda_l((i-1) \Delta -s)} \dd s
\\
&\quad 
+\frac{v_{i-1}(1-\ee^{-\lambda_l \Delta})^2}{2\lambda_l}
(1-\ee^{-2\lambda_l (i-1 -N_{p-1}) \Delta})
\\
&= 
\frac{(1-\ee^{-\lambda_l \Delta})^2}{2\lambda_l} 
(1-\ee^{-2\lambda_l \Delta}) \ee^{-2\lambda_l(i-1) \Delta} 
\sum_{k=1}^{N_{p-1}} v_{k-1} \ee^{2\lambda_l k \Delta}
\\
&\quad 
+\frac{v_{i-1}(1-\ee^{-\lambda_l \Delta})^2}{2\lambda_l}
(1-\ee^{-2\lambda_l (i-1 -N_{p-1}) \Delta}),
\\
\EE[C_{i,l}^2] 
&= \frac{v_{i-1}(1-\ee^{-2\lambda_l \Delta})}{2\lambda_l}.
\end{align*}
Therefore, we obtain
\begin{align}
&\EE[B_{i,l}^2] + \EE[C_{i,l}^2] 
\nonumber
\\
&=
\frac{v_{i-1}(1-\ee^{-\lambda_l \Delta})^2}{2\lambda_l}
+\frac{v_{i-1}(1-\ee^{-2\lambda_l \Delta})}{2\lambda_l}
\nonumber
\\
&\quad 
+\frac{\ee^{-2 \lambda_l (i-1)\Delta}(1-\ee^{-\lambda_l \Delta})^2}{2\lambda_l}
\biggl(-v_{i-1} \ee^{2\lambda_l N_{p-1} \Delta} 
+(1-\ee^{-2\lambda_l \Delta})
\sum_{k=1}^{N_{p-1}} v_{k-1} \ee^{2\lambda_l k \Delta}
\biggr)
\nonumber
\\
&= \frac{v_{i-1} (1-\ee^{-\lambda_l \Delta})}{\lambda_l}
\nonumber
\\
&\quad
+\frac{\ee^{-2 \lambda_l (i-1)\Delta}(1-\ee^{-\lambda_l \Delta})^2}{2\lambda_l}
\biggl(-v_{i-1} \ee^{2\lambda_l N_{p-1} \Delta} 
+(1-\ee^{-2\lambda_l \Delta})
\sum_{k=1}^{N_{p-1}} v_{k-1} \ee^{2\lambda_l k \Delta}
\biggr)
\nonumber
\\
&=:
\frac{v_{i-1}(1-\ee^{-\lambda_l \Delta})}{\lambda_l}
+s_{i,l}
\label{eq-app-1}
\end{align}
and
\begin{align}
\sum_{i=N_{p-1}+1}^{N_p} \sum_{l \in \mathbb N} |s_{i,l}| 
&\lesssim
\sum_{i=N_{p-1}+1}^{N_p} \sum_{l \in \mathbb N}
\frac{\ee^{-2 \lambda_l (i-1)\Delta}(1-\ee^{-\lambda_l \Delta})^2}{\lambda_l}
\nonumber
\\
&\qquad \times
\biggl(
\ee^{2 \lambda_l N_{p-1} \Delta} +(1-\ee^{-2 \lambda_l \Delta}) 
\frac{\ee^{2 \lambda_l \Delta}}{\ee^{2 \lambda_l \Delta}-1}
\biggr)
\nonumber
\\
&=
\sum_{i=N_{p-1}+1}^{N_p} \sum_{l \in \mathbb N}
\frac{\ee^{-2 \lambda_l (i-1)\Delta}(1-\ee^{-\lambda_l \Delta})^2}{\lambda_l}
(1+\ee^{2 \lambda_l N_{p-1} \Delta })
\nonumber
\\
&=
\sum_{l \in \mathbb N}
\frac{(1-\ee^{-\lambda_l \Delta})^2}{\lambda_l}
(1+\ee^{2 \lambda_l N_{p-1} \Delta })
\sum_{i=N_{p-1}+1}^{N_p} \ee^{-2 \lambda_l (i-1)\Delta}
\nonumber
\\
&\le
\sum_{l \in \mathbb N}
\frac{(1-\ee^{-\lambda_l \Delta})^2}{\lambda_l}
(1+\ee^{2 \lambda_l N_{p-1} \Delta })
\times 
\frac{\ee^{-2 \lambda_l N_{p-1} \Delta}}{1-\ee^{-2 \lambda_l \Delta}}
\nonumber
\\
&\lesssim
\sum_{l \in \mathbb N}
\frac{1-\ee^{-\lambda_l \Delta}}{\lambda_l}
= \OO(\sqrt{\Delta}).
\label{eq-app-2}
\end{align}

\subsubsection{Proofs of Proposition \ref{app_prop1} and Theorem \ref{app_th1}}
\begin{proof}[\bf{Proof of Proposition \ref{app_prop1}}]
Let $p \in \{ 1, \ldots, r+1 \}$. 
For $N_{p-1} +1 \le i \le N_p$, we have
\begin{align*}
\EE \bigl[ (\Delta_i^N X(y))^2 \bigr]
&= \sum_{l_1,l_2 \in \mathbb N} 
\EE[(\Delta_i^N x_{l_1})(\Delta_i^N x_{l_2})] e_{l_1}(y) e_{l_2}(y)
\\
&= \biggl( \sum_{l \in \mathbb N} A_{i,l} e_l(y) \biggr)^2
+\sum_{l \in \mathbb N} \bigl(\EE[B_{i,l}^2] +\EE[C_{i,l}^2] \bigr) e_l(y)^2.
\end{align*}
In the same way as \cite{Bibinger_Trabs2020}, we obtain
\begin{equation*}
\sum_{i=1}^N \biggl( \sum_{l \in \mathbb N} A_{i,l} e_l(y) \biggr)^2
= \OO(\sqrt{\Delta}),
\end{equation*}
\begin{align*}
\sum_{l \in \mathbb N} \bigl(\EE[B_{i,l}^2] +\EE[C_{i,l}^2] \bigr) e_l(y)^2
&= \sum_{l \in \mathbb N}
\frac{v_{i-1}(1-\ee^{-\lambda_l \Delta})}{\lambda_l} e_l(y)^2
+\sum_{l \in \mathbb N} s_{i,l} e_l(y)^2
\\
&= 
\sqrt{\Delta} \frac{v_{i-1}}{\sqrt{\pi \theta_2}} \ee^{-\kappa y} 
+\OO(\Delta^{3/2})
+\sum_{l \in \mathbb N} s_{i,l} e_l(y)^2,
\end{align*}
\begin{equation*}
\sum_{i=N_{p-1}+1}^{N_p} \sum_{l \in \mathbb N} |s_{i,l}| e_l(y)^2 
\lesssim
\sum_{i=N_{p-1}+1}^{N_p} \sum_{l \in \mathbb N} |s_{i,l}|
=\OO(\sqrt{\Delta})
\end{equation*}
and
\begin{align*}
\EE \bigl[ (\Delta_i^N X(y))^2 \bigr]
&= 
\biggl( \sum_{l \in \mathbb N} A_{i,l} e_l(y) \biggr)^2
+\sum_{l \in \mathbb N} \bigl(\EE[B_{i,l}^2] +\EE[C_{i,l}^2] \bigr) e_l(y)^2
\\
&= \sqrt{\Delta} \frac{v_{i-1}}{\sqrt{\pi \theta_2}} 
\ee^{-\kappa y} +r_{i} +\OO(\Delta^{3/2}),
\end{align*}
where $\sum_{i=N_{p-1}+1}^{N_p} r_{i} = \OO(\sqrt{\Delta})$
for $y \in [b,1-b]$. Therefore, we obtain 
\begin{align*}
\frac{1}{N \sqrt{\Delta}} \sum_{i=1}^N \EE \bigl[ (\Delta_i^N X(y))^2 \bigr]
&= \frac{1}{N \sqrt{\Delta}} 
\sum_{p=1}^{r+1} \sum_{i=N_{p-1}+1}^{N_p} 
\EE \bigl[ (\Delta_i^N X(y))^2 \bigr]
\\
&= 
\frac{1}{N} 
\sum_{p=1}^{r+1} \sum_{i=N_{p-1}+1}^{N_p}
\frac{v_{i-1}}{\sqrt{\pi \theta_2}} 
\ee^{-\kappa y}
+\frac{1}{N \sqrt{\Delta}} \sum_{i=1}^{N} r_i +\OO(\Delta)
\\
&= \frac{\ee^{-\kappa y}}{\sqrt{\pi \theta_2}}
\sum_{p=1}^{r+1} \frac{N_p -N_{p-1}}{N} \sigma_p^2 
+\OO(\Delta)
\\
&= \frac{\ee^{-\kappa y}}{\sqrt{\pi \theta_2}}
\sum_{p=1}^{r+1} (\tau_{p} -\tau_{p-1}) \sigma_p^2 
+\OO(\Delta)
\\
&= \frac{V_0 \cdot \ee^{-\kappa y}}{\sqrt{\pi}} +\OO(\Delta)
\end{align*}
for $y \in [b, 1-b]$.
\end{proof}

We prepare the following lemma in order to show Theorem \ref{app_th1}.
\begin{lem}\label{app_lem1}
Under [D1], it follows that 
\begin{align*}
&\cov[\Delta_i^N X(\widetilde y_j), \Delta_{i'}^NX(\widetilde y_{j'})]
\\
&= 
\OO \biggl( \frac{\Delta}{|i-i'|+1} \biggr) + 
\OO \biggl( \frac{\Delta^{1/2}}{|i-i'|+1} \bigg(
\ind_{\{ j= j'\}} 
+\ind_{\{ j \neq j'\}} \frac{m^2 \Delta}{|j-j'|^2+1} \biggr) \biggr),
\end{align*}
\begin{equation*}
\sum_{i,i'=1}^N \cov \bigl[ (\Delta_i^N X(\widetilde y_j))^2, 
(\Delta_{i'}^N X(\widetilde y_j))^2 \bigr]
= \OO(1)
\quad \text{uniformly in } j, 
\end{equation*}
\begin{equation*}
\sum_{j,j'=1}^{m_1} \sum_{i,i'=1}^N
\cov \bigl[ (\Delta_i^N X(\widetilde y_j))^2, 
(\Delta_{i'}^N X(\widetilde y_{j'}))^2 \bigr]
= \OO(m).
\end{equation*}
\end{lem}
\begin{proof}
The essence of the proof is same as Lemma 4.11 in \cite{TKU2025a}.
It holds that 
\begin{align*}
\cov[\Delta_i^N X(\widetilde y_j), \Delta_{i'}^NX(\widetilde y_{j'})]
&= \sum_{l \in \mathbb N} \cov[\Delta_i^N x_l, \Delta_{i'}^Nx_l] 
e_l(\widetilde y_j) e_l(\widetilde y_{j'})
\\
&= \sum_{l \in \mathbb N} \cov[(B_{i,l}+C_{i,l})(B_{i',l}+C_{i',l})] 
e_l(\widetilde y_j) e_l(\widetilde y_{j'}).
\end{align*}
For positive integers $J$ and $K$, we define
\begin{align*}
F_{j'}^{j} &=
\sum_{l \in \mathbb N} 
\frac{1-\ee^{-\lambda_l \Delta}}{\lambda_l} e_l(\widetilde y_j) e_l(\widetilde y_{j'}),
\\
G_{J,j'}^{j} &=
\sum_{l \in \mathbb N} 
\frac{(1-\ee^{-\lambda_l \Delta})^2}{\lambda_l} 
\ee^{-\lambda_l J \Delta} e_l(\widetilde y_j) e_l(\widetilde y_{j'}),
\\
H_{J,K,j'}^{j} &=
\sum_{l \in \mathbb N} 
\biggl(
\frac{(1-\ee^{-\lambda_l \Delta})^2}{\lambda_l}
\ee^{-\lambda_l J \Delta}
\sum_{p = 1}^{K} v_{p-1} \ee^{2\lambda_l p \Delta}
\biggr) e_l(\widetilde y_j) e_l(\widetilde y_{j'}),
\quad J > 2K.
\end{align*}
Since it follows that 
\begin{equation*}
\cov[(B_{i,l}+C_{i,l})(B_{i',l}+C_{i',l})]
= 
\begin{cases}
\EE[B_{i,l}^2] +\EE[C_{i,l}^2], & i = i',
\\
\EE[B_{i,l}B_{i',l}] +\EE[B_{i,l}C_{i',l}]
+\EE[B_{i',l} C_{i,l}], & i \neq i',
\end{cases}
\end{equation*}
\begin{align*}
\EE[B_{i,l}B_{i',l}]
&= (1-\ee^{-\lambda_l \Delta})^2
\int_0^{(i \land i' -1)\Delta}
\sigma(s)^2 \ee^{-\lambda_l((i+i'-2)\Delta -2s)} \dd s
\\
&= (1-\ee^{-\lambda_l \Delta})^2 \ee^{-\lambda_l(i+i'-2)\Delta}
\int_0^{(i \land i' -1)\Delta}
\sigma(s)^2 \ee^{2\lambda_l s} \dd s
\\
&= (1-\ee^{-\lambda_l \Delta})^2
\ee^{-\lambda_l(i+i'-2)\Delta}
\sum_{k = 1}^{i \land i'-1} v_{k-1}
\int_{(k-1) \Delta}^{k\Delta} \ee^{2\lambda_l s} \dd s
\\
&= \frac{(1-\ee^{-\lambda_l \Delta})^2}{2\lambda_l}
(1 -\ee^{- 2\lambda_l \Delta}) \ee^{-\lambda_l(i+i'-2)\Delta}
\sum_{k = 1}^{i \land i'-1} v_{k-1} \ee^{2\lambda_l k \Delta},
\end{align*}
\begin{align*}
\EE[B_{i,l}C_{i',l}]
&= -\ind_{\{i > i'\}} (1-\ee^{-\lambda_l \Delta})
\int_{(i'-1)\Delta}^{i'\Delta} \sigma(s)^2
\ee^{-\lambda_l((i+i'-1)\Delta -2s)} \dd s
\\
&= -\ind_{\{i > i'\}} 
v_{i'-1} (1-\ee^{-\lambda_l \Delta})
\ee^{-\lambda_l(i-i'-1)\Delta}
\frac{1-\ee^{-2 \lambda_l \Delta}}{2 \lambda_l}
\\
&= -\ind_{\{i > i'\}} 
\frac{v_{i'\land i -1} (1-\ee^{-\lambda_l \Delta})^2}{2 \lambda_l}
\ee^{-\lambda_l(|i-i'|-1)\Delta}(1+\ee^{- \lambda_l \Delta}),
\end{align*}
\begin{align*}
&\EE[B_{i,l}^2] +\EE[C_{i,l}^2] 
\\
&= 
\frac{v_{i-1} (1-\ee^{-\lambda_l \Delta})}{\lambda_l}
\\
&\quad+\frac{\ee^{-2 \lambda_l (i-1)\Delta}(1-\ee^{-\lambda_l \Delta})^2}{2\lambda_l}
\biggl(-v_{i-1} \ee^{2\lambda_l N_{p-1} \Delta} 
+(1-\ee^{-2\lambda_l \Delta})
\sum_{k=1}^{N_{p-1}} v_{k-1} \ee^{2\lambda_l k \Delta}
\biggr),
\\
&\EE[B_{i,l}B_{i',l}] +\EE[B_{i,l}C_{i',l}] +\EE[B_{i',l} C_{i,l}] 
\\
&=
\frac{(1-\ee^{-\lambda_l \Delta})^2}{2\lambda_l}
(1 -\ee^{- 2\lambda_l \Delta}) \ee^{-\lambda_l(i+i'-2)\Delta}
\sum_{k = 1}^{i \land i'-1} v_{k-1} \ee^{2\lambda_l k \Delta} 
\\
&\quad
-\ind_{\{i \neq i'\}} 
\frac{v_{i'\land i -1} (1-\ee^{-\lambda_l \Delta})^2}{2\lambda_l}
\ee^{-\lambda_l(|i-i'|-1)\Delta}(1+\ee^{- \lambda_l \Delta}),
\end{align*}
we obtain
\begin{align*}
&\cov[\Delta_i^N X(\widetilde y_j), \Delta_{i'}^NX(\widetilde y_{j'})]
\\
&=
\begin{cases}
v_{i-1} F_{j'}^{j} +\frac{1}{2}(-v_{i-1} G_{2(i-1-N_{p-1}),j'}^{j}
+H_{2(i-1),N_{p-1}, j'}^{j}), 
& N_{p-1} +1 \le i = i' \le N_p,
\\
\frac{1}{2}(
H_{i+i'-2,i \land i'-1, j'}^{j}
-v_{i \land i' -1} (G_{|i-i'|-1,j'}^{j} +G_{|i-i'|,j'}^{j})), 
& i \neq i'.
\end{cases}
\end{align*}
Let $f(x) = \frac{1-\ee^{-x}}{x}$ and $g(x) = f(\theta_2 \pi^2 x^2)$.
It then follows from the equality $-2 \sin(a)\sin(b) = \cos(a+b) -\cos(a-b)$ and
Lemmas A.8 and A.9 in \cite{Hildebrandt_Trabs2021} that 
\begin{align*}
F_{j'}^j &= \Delta \sum_{l \in \mathbb N} f(\lambda_l \Delta) 
e_l(\widetilde y_j) e_l(\widetilde y_{j'})
\\
&= \Delta \sum_{l \in \mathbb N} f(\theta_2 \pi^2 l^2 \Delta) 
e_l(\widetilde y_j) e_l(\widetilde y_{j'})
+\OO(\Delta)
\\
&= \Delta \sum_{l \in \mathbb N} g(l \sqrt{\Delta}) 
e_l(\widetilde y_j) e_l(\widetilde y_{j'})
+\OO(\Delta)
\\
&= -\Delta \ee^{-\kappa (y_j+y_{j'})/2}
\sum_{l \in \mathbb N} g(l \sqrt{\Delta}) 
\bigl( \cos(\pi l (\widetilde y_j+ \widetilde y_{j'})) 
-\cos(\pi l (\widetilde y_j -\widetilde y_{j'})) \bigr)
+\OO(\Delta)
\\
&= 
\begin{cases}
\OO(\sqrt{\Delta}), & j=j',
\\
\OO \bigl( \frac{\Delta^{3/2}}{|j-j'|^2 \delta^2} \bigr) +\OO(\Delta), & j \neq j'.
\end{cases}
\end{align*}
Since it follows that $G_{J,j'}^j = \OO(F_{j'}^j/J)$ and 
$H_{J,K,j'}^{j} = \OO(G_{J-2K,j'}^{j})$, 
it holds that 
\begin{align*}
F_{j'}^{j} &= 
\OO(\Delta) + 
\OO \biggl( \Delta^{1/2} \bigg(
\ind_{\{ j= j'\}} 
+\ind_{\{ j \neq j'\}} \frac{m^2 \Delta}{|j-j'|^2+1} \biggr) \biggr),
\\
G_{J,j'}^{j} &= 
\OO \biggl( \frac{\Delta}{J+1} \biggr) + 
\OO \biggl( \frac{\Delta^{1/2}}{J+1} \bigg(
\ind_{\{ j= j'\}} 
+\ind_{\{ j \neq j'\}} \frac{m^2 \Delta}{|j-j'|^2+1} \biggr) \biggr),
\\
H_{J,K,j'}^{j} &= 
\OO \biggl( \frac{\Delta}{J-2K+1} \biggr) + 
\OO \biggl( \frac{\Delta^{1/2}}{J-2K+1} \bigg(
\ind_{\{ j= j'\}} 
+\ind_{\{ j \neq j'\}} \frac{m^2 \Delta}{|j-j'|^2+1} \biggr) \biggr).
\end{align*}
Therefore, we obtain
\begin{align*}
&v_{i-1} F_{j'}^{j} +\frac{1}{2}(-v_{i-1} G_{2(i-1-N_{p-1}),j'}^{j}
+H_{2(i-1),N_{p-1}, j'}^{j})
\\
&= \OO(\Delta) + 
\OO \biggl( \Delta^{1/2} \bigg(
\ind_{\{ j= j'\}} 
+\ind_{\{ j \neq j'\}} \frac{m^2 \Delta}{|j-j'|^2+1} \biggr) \biggr),
\end{align*}
\begin{align*}
&H_{i+i'-2,i \land i'-1, j'}^{j}
-v_{i \land i' -1} (G_{|i-i'|-1,j'}^{j} +G_{|i-i'|,j'}^{j})
\\
&= 
\OO \biggl( \frac{\Delta}{|i-i'|+1} \biggr) + 
\OO \biggl( \frac{\Delta^{1/2}}{|i-i'|+1} \bigg(
\ind_{\{ j= j'\}} 
+\ind_{\{ j \neq j'\}} \frac{m^2 \Delta}{|j-j'|^2+1} \biggr) \biggr)
\end{align*}
and
\begin{align*}
&\cov[\Delta_i^N X(\widetilde y_j), \Delta_{i'}^NX(\widetilde y_{j'})]
\\
&= 
\OO \biggl( \frac{\Delta}{|i-i'|+1} \biggr) + 
\OO \biggl( \frac{\Delta^{1/2}}{|i-i'|+1} \bigg(
\ind_{\{ j= j'\}} 
+\ind_{\{ j \neq j'\}} \frac{m^2 \Delta}{|j-j'|^2+1} \biggr) \biggr).
\end{align*}
In particular, it follows from $m^2 \Delta \to 0$ that
\begin{align*}
\sum_{i,i'=1}^N \cov[\Delta_i^N X(\widetilde y_j), \Delta_{i'}^NX(\widetilde y_j)]^2
= \OO \biggl( \sum_{i,i'=1}^N \frac{\Delta}{|i-i'|^2+1} \biggr)
= \OO(1)
\end{align*}
and 
\begin{align*}
&\sum_{j,j'=1}^m \sum_{i,i'=1}^N 
\cov[\Delta_i^N X(\widetilde y_j), \Delta_{i'}^NX(\widetilde y_{j'})]^2
\\
&= 
\OO \biggl( \sum_{j,j'=1}^m \sum_{i,i'=1}^N \frac{\Delta^2}{|i-i'|^2+1} \biggr) + 
\OO \Biggl( \Delta\sum_{i,i'=1}^N \frac{1}{|i-i'|^2+1} 
\bigg(
m  + \sum_{j,j'=1}^m \frac{m^4 \Delta^2}{|j-j'|^4+1} \biggr)
\Biggr)
\\
&= 
\OO ( m^2 \Delta^2 N ) + 
\OO \bigl( \Delta N (m  + m^5 \Delta^2) \bigr)
\\
&= \OO (m).
\end{align*}
Since $\mathcal A_{j, N} = \sum_{i=1}^N (\sum_{l \in \mathbb N} A_{i,l} e_l(y_j))^2 
= \OO(\sqrt{\Delta})$, we have
\begin{align*}
&\sum_{i,i'=1}^N 
\cov \bigl[ (\Delta_i^N X(\widetilde y_j))^2, (\Delta_{i'}^NX(\widetilde y_j))^2 \bigr]
\\
&\lesssim
\biggl( \sum_{i,i'=1}^N 
\cov[\Delta_i^N X(\widetilde y_j), \Delta_{i'}^NX(\widetilde y_j)]^2 \biggr)^{1/2}
\mathcal A_{j, N}
\\
&\quad
+\sum_{i,i'=1}^N \cov[\Delta_i^N X(\widetilde y_j), \Delta_{i'}^NX(\widetilde y_j)]^2
\\
&= \OO(\sqrt{\Delta}) +\OO(1)
= \OO(1),
\end{align*}
\begin{align*}
&\sum_{j,j'=1}^m \sum_{i,i'=1}^N 
\cov \bigl[ (\Delta_i^N X(\widetilde y_j))^2, 
(\Delta_{i'}^NX(\widetilde y_{j'}))^2 \bigr]
\\
&\lesssim
\biggl( 
\sum_{j,j'=1}^m \sum_{i,i'=1}^N 
\cov[\Delta_i^N X(\widetilde y_j), \Delta_{i'}^NX(\widetilde y_{j'})]^2
\biggr)^{1/2}
\sum_{j=1}^m \mathcal A_{j, N}
\\
&\quad+
\sum_{j,j'=1}^m \sum_{i,i'=1}^N 
\cov[\Delta_i^N X(\widetilde y_j), \Delta_{i'}^NX(\widetilde y_{j'})]^2
\\
&= \OO(\sqrt{m} \cdot m \sqrt{\Delta}) +\OO(m)
= \OO(m).
\end{align*}
\end{proof}

\begin{proof}[\bf{Proof of Theorem \ref{app_th1}}]
The essence of the proof is same as Theorem 2.2 in \cite{TKU2025a}.
Let $\nu = (\kappa,V_0)$ and $f(y;\nu) = V_0 \ee^{-\kappa y}/\sqrt{\pi}$.
Using the mean value theorem, we have
\begin{equation*}
-\sqrt{m N} \pd_{\nu} U_{m,N} (\nu^*)^\TT
=\int_0^1 \pd_{\nu}^2 
U_{m,N} (\nu^* +u(\widehat \nu -\nu^*)) \dd u 
\sqrt{m N} (\widehat \nu -\nu^*).
\end{equation*}
Let $h(y;\nu) = f(y;\nu) -f(y;\nu^*)$, 
$U(\nu, \nu^*) = \frac{1}{1-2b} \int_{b}^{1-b} h(y;\nu)^2 \dd y$ and
\begin{equation*}
Z_j = \frac{1}{N} \sum_{i=1}^N 
\biggl( \frac{(\Delta_i^N X(\widetilde y_j))^2}{\sqrt{\Delta}} 
-f(\widetilde y_j;\nu^*) \biggr).
\end{equation*}
Note that the function $\Xi \ni \nu \mapsto U(\nu,\nu^*)$
takes its unique minimum in $\nu = \nu^*$.
Since we find from Proposition \ref{app_prop1} that
\begin{equation*}
Z_j = \frac{1}{N \sqrt{\Delta}} \sum_{i=1}^N 
\bigl( (\Delta_i^N X(\widetilde y_j))^2 
-\EE \bigl[ (\Delta_i^N X(\widetilde y_j))^2 \bigr] \bigr)
+ \OO(\Delta),
\end{equation*}
it follows from Lemma \ref{app_lem1} and the Schwarz inequality that
\begin{align*}
\EE[Z_j^2] = \frac{1}{N} \sum_{i,i'=1}^N 
\cov \bigl[ (\Delta_i^N X(\widetilde y_j))^2, (\Delta_{i'}^NX(\widetilde y_j))^2 \bigr]
+\OO(\Delta^2)
= \OO(\Delta),
\end{align*}
\begin{align*}
\EE \Biggl[ \sup_{\nu \in \Xi} \biggl| 
\frac{1}{m} \sum_{j=1}^m Z_j h(y;\nu) \biggr|^2 \Biggr]
\lesssim
\biggl( \frac{1}{m} \sum_{j=1}^m \EE[Z_j^2]^{1/2} \biggr)^2
= \OO(\Delta).
\end{align*}
Therefore, we obtain
\begin{align*}
&\sup_{\nu \in \Xi} |U_{m,N}(\nu) -U(\nu, \nu^*)|
\\
&\le
\frac{1}{m} \sum_{j=1}^m Z_j^2 
+2 \sup_{\nu \in \Xi} \biggl| \frac{1}{m} \sum_{j=1}^m Z_j h(\widetilde y_j;\nu) \biggr|
+\sup_{\nu \in \Xi} \biggl| \frac{1}{m} 
\sum_{j=1}^m h(\widetilde y_j;\nu)^2 -U(\nu,\nu^*) \biggr|
\\
&\pto 0.
\end{align*}
Since it follows from Lemma \ref{app_lem1} that
\begin{align*}
&\VV \biggl[
\frac{1}{\sqrt{m N \Delta}} \sum_{j=1}^m \sum_{i=1}^N
(\Delta_i^N X(\widetilde y_j))^2 \pd_\nu f(\widetilde y_j;\nu^*)^\TT
\biggr]
\\
&= \frac{1}{m N \Delta} 
\sum_{j,j'=1}^m \sum_{i,i'=1}^N 
\cov \bigl[(\Delta_i^N X(\widetilde y_j))^2, (\Delta_{i'}^NX(\widetilde y_{j'}))^2 \bigr] 
\pd_\nu f(\widetilde y_j;\nu^*)^\TT \pd_\nu f(\widetilde y_j;\nu^*)
\\
&=\OO(1)
\end{align*}
and from Proposition \ref{app_prop1} that for $y \in [b, 1-b]$,
\begin{equation*}
\sum_{i=1}^N
\Bigl( \EE \bigl[ (\Delta_i^N X(y))^2 \bigr] -\sqrt{\Delta} f(y;\nu^*) \Bigr)
= \OO(\sqrt{\Delta}),
\end{equation*}
we obtain by $m = \OO(N^\rho)$ for some $\rho \in (0,1/2)$,
\begin{align*}
\sqrt{m N} \pd_{\nu} U_{m,N} (\nu^*)^\TT
&= \frac{-2}{\sqrt{m N \Delta}} \sum_{j=1}^m \sum_{i=1}^N
\Bigl( (\Delta_i^N X(\widetilde y_j))^2 
-\EE \bigl[ (\Delta_i^N X(\widetilde y_j))^2 \bigr] \Bigr)
\pd_\nu f(\widetilde y_j;\nu^*)^\TT
\\
&\quad 
-\frac{2}{\sqrt{m N \Delta}} \sum_{j=1}^m \sum_{i=1}^N
\Bigl( \EE \bigl[ (\Delta_i^N X(\widetilde y_j))^2 \bigr] 
-\sqrt{\Delta} f(\widetilde y_j;\nu^*) \Bigr)
\pd_\nu f(\widetilde y_j;\nu^*)^\TT
\\
&=\Op(1) +\OO \biggl( \sqrt{\frac{m}{N}} \biggr)
\\
&=\Op(1).
\end{align*}

For $b \in (0,1/2)$, we define
\begin{equation*}
L(\nu^*) = 2 \int_{b}^{1-b} \pd_\nu f(y;\nu^*)^\TT \pd_\nu f(y;\nu^*) \dd y.
\end{equation*}
Note that $L(\nu^*)$ is positive definite (see \cite{Bibinger_Trabs2020}). 
Since there exist continuous functions 
$[b,1-b] \times \Xi \ni (y,\nu) \mapsto g_j(y;\nu)$ ($j = 1, 2$)
such that
\begin{equation*}
\pd_\nu^2 U_{m,N}(\nu)
= \frac{1}{m} \sum_{j=1}^m \biggl\{ g_1(\widetilde y_j;\nu) 
+\biggl( 
\frac{1}{N \sqrt{\Delta}} \sum_{i=1}^N (\Delta_i^N X(\widetilde y_j))^2
\biggr) g_2(\widetilde y_j;\nu)
\biggr \}
\end{equation*}
and it holds from Proposition \ref{app_prop1} that
\begin{equation*}
\sup_{m, N} 
\EE \biggl[ \frac{1}{m N \sqrt{\Delta}} \sum_{j=1}^m \sum_{i=1}^N 
(\Delta_i^N X(\widetilde y_j))^2 \biggr]
< \infty,
\end{equation*}
it follows that for $\epsilon_{m,N} \downarrow 0$,
\begin{align*}
&\sup_{|\nu -\nu^*| \le \epsilon_{m,N}}
\bigl| \pd_\nu^2 U_{m,N}(\nu) -L(\nu^*) \bigr|
\\
&\le 
\sup_{|\nu -\nu^*| \le \epsilon_{m,N}}
\bigl| \pd_\nu^2 U_{m,N}(\nu) -\pd_\nu^2 U_{m,N}(\nu^*) \bigr|
+ \bigl| \pd_\nu^2 U_{m,N}(\nu^*) -L(\nu^*) \bigr|
\\
&\le 
\sup_{\substack{y \in [b,1-b], \\ |\nu -\nu^*| \le \epsilon_{m,N}}} 
|g_1(y;\nu) -g_1(y;\nu^*)|
\\
&\quad+
\biggl( 
\frac{1}{m N \sqrt{\Delta}} \sum_{j=1}^m \sum_{i=1}^N (\Delta_i^N X(\widetilde y_j))^2
\biggr)
\sup_{\substack{y \in [b,1-b], \\ |\nu -\nu^*| \le \epsilon_{m,N}}} 
|g_2(y;\nu) -g_2(y;\nu^*)|
\\
&\quad+ \biggl| \frac{2}{m} \sum_{j=1}^m \pd_\nu f(y_j;\nu^*)^\TT \pd_\nu f(y_j;\nu^*)
-L(\nu^*) \biggr|
\\
&\pto 0
\end{align*}
as $m, N \to \infty$.

Therefore, we get the desire result.
\end{proof}

\subsubsection{Proofs of Proposition \ref{app_prop2} and Theorem \ref{app_th2}}

\begin{proof}[\bf{Proof of Proposition \ref{app_prop2}}]
For $j \in \{ 1, \ldots, m \}$ and $l \in \mathbb N$, we set
$\delta_j e_l = e_l(\widetilde y_j) -e_l(\widetilde y_{j-1})$.
Since it follows that
\begin{equation*}
D_{i,j} X = \sum_{l \in \mathbb N} \Delta_i^N x_l \delta_j e_l,
\end{equation*}
we have
\begin{align*}
\EE \bigl[ (D_{i,j} X)^2 \bigr]
&= 
\biggl(
\sum_{l \in \mathbb N} A_{i,l} \delta_j e_l
\biggr)^2
+\sum_{l \in \mathbb N} 
\bigl(\EE[B_{i,l}^2] +\EE[C_{i,l}^2] \bigr) (\delta_j e_l)^2
\\
&=: S_1 +S_2.
\end{align*}
It follows from the Schwarz inequality that 
\begin{equation*}
S_1 \le \| A_{\theta}^{1/2} X_0 \|^2 \sum_{l \in \mathbb N} 
\frac{(1-\ee^{-\lambda_l \Delta})^2}{\lambda_l}
\ee^{-2(i-1)\lambda_l \Delta} (\delta_j e_l)^2
=: \| A_{\theta}^{1/2} X_0 \|^2 G_{i,j},
\end{equation*}
\begin{equation*}
\sum_{i=1}^{N} G_{i,j} \le F_j
:= \sum_{l \in \mathbb N} 
\frac{1-\ee^{-\lambda_l \Delta}}{\lambda_l} (\delta_j e_l)^2
\end{equation*}
and from \eqref{eq-app-1} and \eqref{eq-app-2} that for $N_{p-1} +1 \le i \le N_p$, 
\begin{equation*}
S_2 = v_{i-1} F_j + \sum_{l \in \mathbb N} s_{i,l} (\delta_j e_l)^2
=: v_{i-1} F_j +H_{i,j},
\quad
\sum_{i=N_{p-1}+1}^{N_p} H_{i,j} \lesssim F_j.
\end{equation*}
Since it holds from the proof of Proposition 3.5 in \cite{Hildebrandt_Trabs2021} that
\begin{equation*}
F_j = \sqrt{\Delta} \ee^{-\kappa \overline y_j} \psi_{r}(\theta_2)
+ \OO(\Delta^{3/2}),
\end{equation*}
we obtain
\begin{equation*}
\EE \bigl[ (D_{i,j} X)^2 \bigr] 
= \sqrt{\Delta} v_{i-1} 
\ee^{-\kappa \overline y_j} \psi_{r}(\theta_2)
+ R_{i,j} + \OO(\Delta^{3/2}),
\end{equation*}
where
\begin{equation*}
\sum_{i= N_{p-1}+1}^{N_p} R_{i,j} = \OO(\sqrt{\Delta})
\end{equation*}
uniformly in $j$. Hence, we obtain
\begin{align*}
\frac{1}{N \sqrt{\Delta}} \sum_{i=1}^N \EE \bigl[ (D_{i,j} X)^2 \bigr]
&= \frac{1}{N \sqrt{\Delta}} 
\sum_{p=1}^{r+1} 
\sum_{i=N_{p-1}+1}^{N_p} 
\EE \bigl[ (D_{i,j} X)^2 \bigr]
\\
&= 
\frac{1}{N} 
\sum_{p=1}^{r+1} \sum_{i=N_{p-1}+1}^{N_p} 
v_{i-1} \ee^{-\kappa \overline y_j} \psi_{r}(\theta_2)
\\
&\qquad+
\frac{1}{N \sqrt{\Delta}} 
\sum_{p=1}^{r+1} \sum_{i=N_{p-1}+1}^{N_p} R_{i,j} +\OO(\Delta)
\\
&= 
\ee^{-\kappa \overline y_j} \psi_{r}(\theta_2)
\sum_{p=1}^{r+1} 
\frac{N_p-N_{p-1}}{N} \sigma_p^2 
+\OO(\Delta)
\\
&= 
\ee^{-\kappa \overline y_j} \psi_{r}(\theta_2)
\sum_{p=1}^{r+1} (\tau_p -\tau_{p-1}) \sigma_p^2 
+\OO(\Delta)
\\
&= V \ee^{-\kappa \overline y_j} \psi_{r}(\theta_2) +\OO(\Delta).
\end{align*}
\end{proof}

We give the following lemma in order to show Theorem \ref{app_th2}.
\begin{lem}\label{app_lem2}
Under [D1], it follows that 
\begin{equation*}
\cov[D_{i,j}X, D_{i',j'}X]
= \OO \biggl( \frac{\Delta^{1/2}}{|i-i'|+1} 
\Bigl(\Delta +\frac{1}{|j-j'|^2+1} \Bigr) \biggr),
\end{equation*}
\begin{equation*}
\sum_{i,i'=1}^N \cov \bigl[ (D_{i,j} X)^2, (D_{i',j} X)^2 \bigr]
= \OO(1)
\quad \text{uniformly in } j,
\end{equation*}
\begin{equation*}
\sum_{j,j'=1}^m \sum_{i,i'=1}^N
\cov \bigl[ (D_{i,j} X)^2, (D_{i',j'} X)^2 \bigr]
= \OO(m).
\end{equation*}
\end{lem}

\begin{proof}
For positive integers $J$ and $K$, we set
\begin{align*}
F_{j'}^{j} &=
\sum_{l \in \mathbb N} 
\frac{1-\ee^{-\lambda_l \Delta}}{\lambda_l}
\delta_j e_l \delta_{j'} e_l,
\\
G_{J,j'}^{j} &=
\sum_{l \in \mathbb N} 
\frac{(1-\ee^{-\lambda_l \Delta})^2}{\lambda_l}
\ee^{-\lambda_l J \Delta}
\delta_j e_l \delta_{j'} e_l,
\\
H_{J,K,j'}^{j} &=
\sum_{l \in \mathbb N} 
\biggl(
\frac{(1-\ee^{-\lambda_l \Delta})^2}{\lambda_l}
\ee^{-\lambda_l J \Delta}
\sum_{p = 1}^{K} v_{p-1} \ee^{2\lambda_l p \Delta}
\biggr)
\delta_j e_l \delta_{j'} e_l,
\quad J > 2K.
\end{align*}
By a simple calculation similar to that in the proof of Lemma 4.10 in \cite{TKU2025a},
we have
\begin{align*}
F_{j'}^{j} &= 
\OO \biggl( \Delta^{1/2} \Bigl(\Delta +\frac{1}{|j-j'|^2+1} \Bigr) \biggr),
\\
G_{J,j'}^{j} &= 
\OO \biggl( \frac{\Delta^{1/2}}{J+1} \Bigl(\Delta +\frac{1}{|j-j'|^2+1} \Bigr) \biggr)
\end{align*}
and 
\begin{equation*}
H_{J,K,j'}^{j} = \OO \biggl( \frac{\Delta^{1/2}}{J-2K+1} 
\Bigl(\Delta +\frac{1}{|j-j'|^2+1} \Bigr) \biggr).
\end{equation*}
In the same way as the proof of Lemma \ref{app_lem1}, we have
\begin{align*}
\cov[D_{i,j}X, D_{i',j'}X]
&= 
\OO \biggl( \frac{\Delta^{1/2}}{|i-i'|+1} 
\Bigl(\Delta +\frac{1}{|j-j'|^2+1} \Bigr) \biggr).
\end{align*}
Since it holds from $m^2 \Delta = \OO(1)$ that
\begin{align*}
\sum_{i,i'=1}^N \cov[D_{i,j} X, D_{i',j} X]^2
= \OO \biggl( \sum_{i,i'=1}^N \frac{\Delta}{|i-i'|^2+1} \biggr)
= \OO(1)
\end{align*}
and
\begin{align*}
&\sum_{j,j'=1}^m \sum_{i,i'=1}^N \cov[D_{i,j} X, D_{i',j'} X]^2
\\
&= 
\OO \Biggl( \Delta \sum_{i,i'=1}^N \frac{1}{|i-i'|^2+1} 
\bigg(
m^2 \Delta^2  + \sum_{j,j'=1}^m \frac{1}{|j-j'|^4+1} \biggr)
\Biggr)
\\
&= 
\OO \bigl( \Delta N (m^2 \Delta^2  + m) \bigr)
\\
&= \OO (m),
\end{align*}
we obtain the desired results in the same manner as in the proof of Lemma \ref{app_lem1}.
\end{proof}

\begin{proof}[\bf{Proof of Theorem \ref{app_th2}}]
In the same manner as the proof of Theorem \ref{app_th1}, 
it can be shown by Proposition \ref{app_prop2} and Lemma \ref{app_lem2}.
\end{proof}

\section{Appendix: Estimation for linear parabolic SPDEs in two space dimensions 
with volatility changes}\label{secB}
We consider parametric estimation for the linear parabolic SPDE 
\begin{equation*}
\dd X_t(y,z)
= \biggl\{ 
\theta_2 
\biggl( \frac{\pd^2}{\pd y^2} +\frac{\pd^2}{\pd z^2}\biggr)
+\theta_{1,1} \frac{\pd}{\pd y} +\theta_{1,2} \frac{\pd}{\pd z}
+ \theta_0 
\biggr\} X_t(y,z) \dd t
+\sigma(t) \dd W_t^{Q}(y,z)
\end{equation*}
for $(t,y,z) \in [0,1] \times (0,1)^2$
with an initial value $X_0$ and the Dirichlet boundary condition 
$X_t(y,z) = 0$, $(t,y,z) \in [0,1] \times \pd (0,1)^2$, where 
$\theta_0 \in \mathbb R$, 
$\theta_1 = (\theta_{1,1}, \theta_{1,2})^\TT \in \mathbb R^2$
and $\theta_2 \in (0, \infty)$ are unknown parameters,
the volatility function $\sigma(t)$ is characterized by
\begin{equation}\label{vol2}
\sigma(t) = \sum_{p = 1}^{r+1} \sigma_p \ind_{[\tau_{p-1}, \tau_{p})}(t)
\end{equation}
with $r \in \mathbb N$, 
$0 = \tau_0 < \tau_1 < \tau_2 < \cdots < \tau_r < \tau_{r+1} = 1$,
and $\sigma_p \in (0, \infty)$ for $p \in \{ 1,\ldots, r+1 \}$.
For convenience, we write $[\tau_r, \tau_{r+1}) = [\tau_r,1]$.
$\{ W_t^Q \}_{t \ge 0}$ is the $Q$-Wiener process 
in a Sobolev space on $(0,1)^2$ and is given by
\begin{equation*}
W_t^Q = \sum_{l \in \mathbb N^2} \gamma_l^{-\alpha/2} w_l(t) e_l
\end{equation*}
with an unknown damping parameter $\alpha \in (0,2)$,
independent real valued standard Brownian motions $\{ w_l \}_{l \in \mathbb N^2}$,
and $e_l(y,z) = e_{l_1}^{(1)}(y) e_{l_2}^{(2)}(z)$ for $l = (l_1,l_2) \in \mathbb N^2$,
where $\kappa = (\kappa_1, \kappa_2)^\TT$, $\kappa_j = \theta_{1,j}/\theta_2$
and $e_{l}^{(j)}(x) = \sqrt{2} \exp(-\kappa_j x/2) \sin(\pi l x)$.
We here set 
$\gamma_l = \lambda_l = \theta_2 \pi^2 |l|_2^2 
+\frac{|\theta_1|_2^2}{4 \theta_2} -\theta_0$ or $\gamma_l = \pi^2 |l|_2^2 +\mu_0$ 
with an unknown parameter $\mu_0 \in (-2\pi^2,\infty)$.
Let $[\underline{\alpha}, \overline{\alpha}] \subset (0,2)$ be 
the parameter space of $\alpha$
and we assume that the true value $\alpha^*$ belongs to $(\underline{\alpha}, \overline{\alpha})$.

Suppose that we have discrete observations 
$\mathbf X_{M_1,M_2,N} = \{ X_{t_i^N}(y_{j_1}^{(1)},y_{j_2}^{(2)}) \}_{
0 \le i \le N, 0 \le j_1 \le M_1, 0 \le j_2 \le M_2}$ with
\begin{equation*}
t_i^N = i \Delta = \frac{i}{N},
\quad
y_j^{(k)} = \frac{j}{M_k}, 
\quad k \in \{ 1,2 \}.
\end{equation*}
For $c \in (0,1/2)$, $m_k \in \{1, \ldots, M_k \}$, $k \in \{1,2\}$ 
and $n \in \{1,\ldots, N \}$, we will write the thinned data 
obtained from $\mathbf X_{M_1,M_2,N}$ as 
$\mathbf X_{m_1,m_2,n}^{(c)} = 
\{ X_{t_i^n}(\widetilde y_{j_1}^{(1)}, \widetilde y_{j_2}^{(2)}) \}_{
0 \le i \le n, 0 \le j_1 \le m_1, 0 \le j_2 \le m_2}$
with $m_1 = m_2$, $\delta = (1-2b)/m_1$
\begin{equation*}
t_i^n = i \cdot \frac{1}{N} \biggl\lfloor \frac{N}{n} \biggr\rfloor,
\quad
\widetilde y_j^{(k)} = c + j \cdot \delta, 
\end{equation*}
where $\widetilde y_j^{(k)} \in \{ y_0^{(k)}, \ldots, y_{M_k}^{(k)} \}$.

Let $ -A_\theta = \theta_2 
( \frac{\pd^2}{\pd y^2} +\frac{\pd^2}{\pd z^2})
+\theta_{1,1} \frac{\pd}{\pd y} +\theta_{1,2} \frac{\pd}{\pd z}
+ \theta_0$. 
While we can consider more general initial conditions such as [A1], 
for simplicity, we make the following condition.
\begin{description}
\item[{[D2]}]
The initial value $X_0 \in L^2((0,1)^2)$ is deterministic and 
$\| A_\theta^{(1+\overline \alpha)/2} X_0 \| < \infty$.
\end{description}

Let $b \in (0, 1/2)$. 
Suppose that we have two thinned data $\mathbf X_{m_1,m_2,N}^{(b)}$
and $\mathbf X_{m_1',m_2',N'}^{(b)}$ obtained from $\mathbf X_{M_1,M_2,N}$,
where $m_k' = m_k/2$, $m' = m_1' m_2'$, $N' = N/4$, 
$m := m_1 m_2 = \OO(N)$ and $N = \OO(m)$.
For the thinned data $\mathbf X_{m_1,m_2,N}^{(b)}$, we set the triple increments
\begin{align*}
T_{i,j,k} X 
&= \Delta_i^N X(\widetilde y_{j}^{(1)},\widetilde y_{k}^{(2)})
-\Delta_i^N X(\widetilde y_{j-1}^{(1)},\widetilde y_{k}^{(2)})
-\Delta_i^N X(\widetilde y_{j}^{(1)},\widetilde y_{k-1}^{(2)})
+\Delta_i^N X(\widetilde y_{j-1}^{(1)},\widetilde y_{k-1}^{(2)})
\\
&=
\sum_{l_1,l_2 \in \mathbb N} (x_{l_1,l_2}(t_i) -x_{l_1,l_2}(t_{i-1}))
(e_{l_1}^{(1)}(\widetilde y_{j}^{(1)}) -e_{l_1}^{(1)}(\widetilde y_{j-1}^{(1)}))
(e_{l_2}^{(2)}(\widetilde y_{k}^{(2)}) -e_{l_2}^{(2)}(\widetilde y_{k-1}^{(2)})).
\end{align*}
We define 
$\overline y_{j}^{(p)} = (\widetilde y_{j-1}^{(p)} +\widetilde y_j^{(p)})/2$,
$\overline y_{j,k} = (\overline y_{j}^{(1)}, \overline y_{k}^{(2)})^\TT$ and
\begin{equation*}
V = \int_0^1 \sigma^2(t) \dd t
= \sum_{p=1}^{r+1} \sigma_p^2 (\tau_{p} -\tau_{p-1}).
\end{equation*}
Let $J_0$ be the Bessel function of the first kind of order $0$: 
\begin{equation*}
J_0(x) = 1+\sum_{k=1}^\infty \frac{(-1)^k}{(k!)^2} \Bigl(\frac{x}{2} \Bigr)^{2k}.
\end{equation*}
For $r, \alpha > 0$, we set
\begin{equation*}
\psi_{r,\alpha}(\theta_2)
=\frac{2}{\theta_2 \pi}
\int_0^\infty 
\frac{1-\ee^{-x^2}}{x^{1+2\alpha}}
\biggl(
J_0\Bigl(\frac{\sqrt{2}r x}{\sqrt{\theta_2}}\Bigr)
-2J_0\Bigl(\frac{r x}{\sqrt{\theta_2}}\Bigr)+1
\biggr) \dd x.
\end{equation*}
For the SPDE with the time-dependent volatility function \eqref{vol2}, 
we then get the following proposition analogous to \cite{TKU2025a}, 
who considered the SPDE with a constant volatility.
\begin{prop}\label{app_prop3}
Let $\alpha \in (0,2)$ and $\delta/\sqrt{\Delta} \equiv r \in (0,\infty)$.
Assume that [D2] holds. Then, it holds that
\begin{equation*}
\frac{1}{N \Delta^\alpha} \sum_{i=1}^N \EE \bigl[ (T_{i,j,k} X)^2 \bigr]
= V \ee^{-\kappa^\TT \overline y_{j,k}} c_\gamma^\alpha \psi_{r,\alpha}(\theta_2)
+\OO(\Delta),
\end{equation*}
where $c_\gamma = \lim_{|l|_2 \to \infty} \frac{\gamma_l}{\pi^2 |l|_2^2}$.
\end{prop}

Let $T_{i,j,k}' X$ be the triple increments 
obtained from $\mathbb X_{m_1', m_2', N'}^{(b)}$.
In the same manner as \cite{TKU2025arXiv1},
we define the estimator of the damping parameter $\alpha$ as follows.
\begin{equation*}
\widehat \alpha = 
\log \left(
\frac{
\displaystyle \frac{1}{m' N'} 
\sum_{k=1}^{m_2'} \sum_{j=1}^{m_1'} \sum_{i=1}^{N'} (T_{i,j,k}' X)^2
}{\displaystyle \frac{1}{m N} 
\sum_{k=1}^{m_2} \sum_{j=1}^{m_1} \sum_{i=1}^N (T_{i,j,k} X)^2}
\right)/\log(4).
\end{equation*}

For the SPDE with the volatility function \eqref{vol2},
we  obtain the following result 
analogous to \cite{TKU2025arXiv1}, who provided an estimator 
for the damping parameter $\alpha$ for the SPDE with a constant volatility.
\begin{thm}\label{app_th3}
Assume that [D2] holds. Then, it holds that
\begin{equation*}
\sqrt{m N} (\widehat \alpha -\alpha^*)= \Op(1).
\end{equation*}
\end{thm}

Let $\widetilde T_{i,j,k} X = T_{i,j,k} X +T_{i+1,j,k} X$, 
$\nu = (\kappa, \theta_2, V)$ and 
$f_{r,\alpha}(y;\nu) = V \ee^{-\kappa^\TT y} c_\gamma^\alpha \psi_{r,\alpha}(\theta_2)$.
For $\alpha \in (0,2)$ and $\delta/\sqrt{\Delta} \equiv r \in (0,\infty)$, we define
\begin{align*}
K_{m,N}(\nu;\alpha) &= 
\frac{1}{m} \sum_{k=1}^{m_2} \sum_{j=1}^{m_1} 
\biggl(
\frac{1}{N \Delta^\alpha}
\sum_{i=1}^N (T_{i,j,k} X)^2 -f_{r,\alpha}(\overline y_{j,k};\nu)
\biggr)^2
\\
&\qquad+\frac{1}{m} \sum_{k=1}^{m_2} \sum_{j=1}^{m_1} 
\biggl(
\frac{1}{N (2\Delta)^\alpha}
\sum_{i=1}^{N-1} 
(\widetilde T_{i,j,k} X)^2 -f_{r/\sqrt{2},\alpha}(\overline y_{j,k};\nu)
\biggr)^2.
\end{align*}
Note that the target parameters are $(\kappa, \theta_2, \int_0^1 \sigma(t)^2 \dd t)$,
rather than $(\kappa, \theta_2, \sigma^2)$ as in \cite{TKU2025arXiv1},
which considered SPDEs with the constant volatility $\sigma$.
Let $\Xi$ be the parameter space of $\nu$ and a compact convex subset of 
$\mathbb R^2 \times (-\frac{r^2}{8 \log(\sqrt{2}-1)},\infty) \times(0,\infty)$.
We assume that the true value 
$\nu^* = (\kappa^*,\theta_2^*,V^*)$ belongs to $\mathrm{Int}(\Xi)$.
For the estimator $\widehat \alpha$, we define the estimator of $\nu$ by
\begin{equation*}
\widehat \nu = \underset{\nu \in \Xi}{\mathrm{argmin}}\,K_{m,N}(\nu;\widehat \alpha).
\end{equation*}
We then get the following result.
\begin{thm}\label{app_th4}
Assume that [D2] holds. Then, it holds that
\begin{equation*}
\frac{\sqrt{m N}}{\log(N)} (\widehat \nu -\nu^*)= \Op(1).
\end{equation*}
\end{thm}
This result holds under both $H_0$ and $H_1$ of \eqref{HTP}.

\subsection{Proofs}
For simplicity, we assume that for any $j \in \{ 1, \ldots, r \}$, 
there exists $N_j \in \{ 1, \ldots, N \}$ such that $\tau_j = N_j \Delta$.

Since it holds that 
\begin{equation*}
x_l(t) = \ee^{-\lambda_l t} x_l(0)
+\frac{1}{\gamma_l^{\alpha/2}} 
\int_0^t \sigma(s) \ee^{-\lambda_l(t-s)} \dd w_l(s),
\end{equation*}
we have
\begin{align*}
\Delta_i^N x_l 
&= -\ee^{-\lambda_l (i-1)\Delta} (1-\ee^{-\lambda_l \Delta})x_l(0)
\\
&\quad-
\frac{1-\ee^{-\lambda_l \Delta}}{\gamma_l^{\alpha/2}} 
\int_0^{(i-1)\Delta} \sigma(s) \ee^{-\lambda_l((i-1)\Delta -s)} \dd w_l(s)
\\
&\quad+\frac{1}{\gamma_l^{\alpha/2}}
\int_{(i-1)\Delta}^{i \Delta} \sigma(s)
\ee^{-\lambda_l(i \Delta-s)} \dd w_l(s)
\\
&=: A_{i,l} +B_{i,l} +C_{i,l}.
\end{align*}
Let $p \in \{1,\ldots, r+1 \}$ and $v_i = (\sigma(i \Delta))^2$. 
For $N_{p-1} +1 \le i \le N_p$, we have
\begin{align*}
\EE[B_{i,l}^2] 
&= 
\frac{(1-\ee^{-\lambda_l \Delta})^2}{2\lambda_l \gamma_l^{\alpha}} 
(1-\ee^{-2\lambda_l \Delta}) \ee^{-2\lambda_l(i-1) \Delta} 
\sum_{k=1}^{N_{p-1}} v_{k-1} \ee^{2\lambda_l k \Delta}
\\
&\quad 
+\frac{v_{i-1}(1-\ee^{-\lambda_l \Delta})^2}{2\lambda_l \gamma_l^{\alpha}}
(1-\ee^{-2\lambda_l (i-1 -N_{p-1}) \Delta}),
\\
\EE[C_{i,l}^2] 
&= \frac{v_{i-1}(1-\ee^{-2\lambda_l \Delta})}{2\lambda_l\gamma_l^{\alpha}}.
\end{align*}
We thus obtain
\begin{align*}
\EE[B_{i,l}^2] + \EE[C_{i,l}^2] =
\frac{v_{i-1}(1-\ee^{-\lambda_l \Delta})}{\lambda_l \gamma_l^{\alpha}}
+s_{i,l},
\end{align*}
where
\begin{equation*}
s_{i,l} = 
\frac{\ee^{-2 \lambda_l (i-1)\Delta}(1-\ee^{-\lambda_l \Delta})^2}
{2\lambda_l\gamma_l^{\alpha}}
\biggl(-v_{i-1} \ee^{2\lambda_l N_{p-1} \Delta} 
+(1-\ee^{-2\lambda_l \Delta})
\sum_{k=1}^{N_{p-1}} v_{k-1} \ee^{2\lambda_l k \Delta}
\biggr)
\end{equation*}
and 
\begin{align*}
\sum_{i=N_{p-1}+1}^{N_p} \sum_{l \in \mathbb N^2} |s_{i,l}| 
&\lesssim
\sum_{l \in \mathbb N^2}
\frac{1-\ee^{-\lambda_l \Delta}}{\lambda_l \gamma_l^\alpha}.
\end{align*}

\begin{proof}[\bf{Proof of Proposition \ref{app_prop3}}]
For $(j,k) \in \{1,\ldots,m_1\} \times \{1,\ldots,m_2\}$
and $l = (l_1,l_2) \in \mathbb N^2$, we set
\begin{equation*}
\delta_{j,k} e_l = 
(e_{l_1}^{(1)}(\widetilde y_j^{(1)}) -e_{l_1}^{(1)}(\widetilde y_{j-1}^{(1)}))
(e_{l_2}^{(2)}(\widetilde y_k^{(2)}) -e_{l_2}^{(2)}(\widetilde y_{k-1}^{(2)})).
\end{equation*}
Since we have
\begin{equation*}
T_{i,j,k} X = \sum_{l \in \mathbb N^2} \Delta_i^N x_l \delta_{j,k} e_l,
\end{equation*}
we see from the proof of (4.3) in \cite{TKU2025a}
that for $N_{p-1} +1 \le i \le N_p$, 
\begin{align*}
\EE \bigl[ (T_{i,j,k} X)^2 \bigr]
&= 
\biggl(
\sum_{l \in \mathbb N^2} A_{i,l} \delta_{j,k} e_l
\biggr)^2
+\sum_{l \in \mathbb N^2} 
\bigl(\EE[B_{i,l}^2] +\EE[C_{i,l}^2] \bigr) (\delta_{j,k} e_l)^2
\\
&=: S_1 +S_2,
\end{align*}
where
\begin{equation*}
S_1 \le \| A_{\theta}^{(1+\alpha)/2} X_0 \|^2 G_{i,j,k},
\quad
\sum_{i=1}^{N} G_{i,j,k} \le F_{j,k}
:= \sum_{l \in \mathbb N^2} 
\frac{1-\ee^{-\lambda_l \Delta}}{\lambda_l \gamma_l^\alpha} (\delta_{j,k} e_l)^2
\end{equation*}
and 
\begin{equation*}
S_2 = v_{i-1} F_{j,k} +H_{i,j,k},
\quad
\sum_{i=N_{p-1}+1}^{N_p} H_{i,j,k} \lesssim F_{j,k}.
\end{equation*}
Since it follows from (4.5) and Lemma 4.2 in \cite{TKU2025a} that
\begin{equation*}
F_{j,k} = \Delta^\alpha 
\ee^{-\kappa^\TT \overline y_{j,k}} c_\gamma^\alpha \psi_{r,\alpha}(\theta_2)
+ \OO(\Delta^{1+\alpha}),
\end{equation*}
we have
\begin{equation*}
\EE \bigl[ (T_{i,j,k} X)^2 \bigr] 
= \Delta^\alpha v_{i-1} 
\ee^{-\kappa^\TT \overline y_{j,k}} c_\gamma^\alpha \psi_{r,\alpha}(\theta_2)
+ R_{i,j,k} 
+ \OO(\Delta^{1+\alpha}),
\end{equation*}
where
\begin{equation*}
\sum_{i= N_{p-1}+1}^{N_p} R_{i,j,k} 
= \OO(\Delta^\alpha)
\end{equation*}
uniformly in $j,k$. Therefore, we obtain
\begin{align*}
\frac{1}{N \Delta^\alpha} \sum_{i=1}^N \EE \bigl[ (T_{i,j,k} X)^2 \bigr]
&= \frac{1}{N \Delta^\alpha} 
\sum_{p=1}^{r+1} 
\sum_{i=N_{p-1}+1}^{N_p} 
\EE \bigl[ (T_{i,j,k} X)^2 \bigr]
\\
&= 
\frac{1}{N} 
\sum_{p=1}^{r+1} \sum_{i=N_{p-1}+1}^{N_p} 
v_{i-1} \ee^{-\kappa^\TT \overline y_{j,k}} c_\gamma^\alpha \psi_{r,\alpha}(\theta_2)
\\
&\qquad+
\frac{1}{N \Delta^\alpha} 
\sum_{p=1}^{r+1} \sum_{i=N_{p-1}+1}^{N_p} R_{i,j,k} +\OO(\Delta)
\\
&= 
\ee^{-\kappa^\TT \overline y_{j,k}} c_\gamma^\alpha \psi_{r,\alpha}(\theta_2)
\sum_{p=1}^{r+1} 
\frac{N_p-N_{p-1}}{N} \sigma_p^2 
+\OO(\Delta)
\\
&= 
\ee^{-\kappa^\TT \overline y_{j,k}} c_\gamma^\alpha \psi_{r,\alpha}(\theta_2)
\sum_{p=1}^{r+1} (\tau_p -\tau_{p-1}) \sigma_p^2 
+\OO(\Delta)
\\
&= 
V \ee^{-\kappa^\TT \overline y_{j,k}} c_\gamma^\alpha \psi_{r,\alpha}(\theta_2)
+\OO(\Delta).
\end{align*}
\end{proof}

For linear parabolic SPDEs with the volatility function \eqref{vol2}, 
we obtain the same result as Lemma 4.11 in \cite{TKU2025a}.
\begin{lem}\label{app_lem3}
Let $\alpha \in (0,2)$ and $\delta = r \sqrt{\Delta}$. Under [D2], it holds that
\begin{align*}
\cov[T_{i,j,k}X, T_{i',j,k}X]
&= 
\OO \biggl( \frac{\Delta^\alpha}{|i-i'|+1} 
\Bigl(\Delta +\frac{1}{(|j-j'|+1)(|k-k'|+1)} \Bigr) \biggr)
\\
&\quad + \OO \biggl( \frac{\Delta^{1/2+\alpha}}{|i-i'|+1} 
\Bigl( \ind_{\{ j \neq j'\}} \frac{1}{|j-j'|+1} 
+\ind_{\{ k \neq k'\}} \frac{1}{|k-k'|+1} \Bigr) \biggr),
\end{align*}
\begin{equation*}
\sum_{i,i'=1}^N \cov[(T_{i,j,k} X)^2, (T_{i',j,k} X)^2]
= \OO(N \Delta^{2\alpha})
\quad \text{uniformly in } j,k,
\end{equation*}
\begin{equation*}
\sum_{k,k'=1}^{m_2} \sum_{j,j'=1}^{m_1} \sum_{i,i'=1}^N
\cov[(T_{i,j,k} X)^2, (T_{i',j',k'} X)^2]
= \OO(m N \Delta^{2\alpha}).
\end{equation*}
\end{lem}
\begin{proof}
For positive integers $J$ and $K$, we define
\begin{align*}
F_{j',k'}^{j,k} &=
\sum_{l \in \mathbb N^2} 
\frac{1-\ee^{-\lambda_l \Delta}}{\lambda \gamma_l^\alpha}
\delta_{j,k} e_l \delta_{j',k'} e_l,
\\
G_{J,j',k'}^{j,k} &=
\sum_{l \in \mathbb N^2} 
\frac{(1-\ee^{-\lambda_l \Delta})^2}{\lambda_l \gamma_l^{\alpha}}
\ee^{-\lambda_l J \Delta}
\delta_{j,k} e_l \delta_{j',k'} e_l,
\\
H_{J,K,j',k'}^{j,k} &=
\sum_{l \in \mathbb N^2} 
\biggl(
\frac{(1-\ee^{-\lambda_l \Delta})^2}{\lambda_l \gamma_l^{\alpha}}
(1 -\ee^{- 2\lambda_l \Delta}) \ee^{-\lambda_l J \Delta}
\sum_{p = 1}^{K} v_{p-1} \ee^{2\lambda_l p \Delta}
\biggr)
\delta_{j,k} e_l \delta_{j',k'} e_l,
\quad J > 2K.
\end{align*}
Since
\begin{equation*}
\cov[(B_{i,l}+C_{i,l})(B_{i',l}+C_{i',l})]
= 
\begin{cases}
\EE[B_{i,l}^2] +\EE[C_{i,l}^2], & i = i',
\\
\EE[B_{i,l}B_{i',l}] +\EE[B_{i,l}C_{i',l}]
+\EE[B_{i',l} C_{i,l}], & i \neq i',
\end{cases}
\end{equation*}
\begin{align*}
&\EE[B_{i,l}^2] +\EE[C_{i,l}^2] 
\\
&= 
\frac{v_{i-1} (1-\ee^{-\lambda_l \Delta})}{\lambda_l\gamma_l^{\alpha}}
\\
&\quad+\frac{\ee^{-2 \lambda_l (i-1)\Delta}(1-\ee^{-\lambda_l \Delta})^2}
{2\lambda_l\gamma_l^{\alpha}}
\biggl(-v_{i-1} \ee^{2\lambda_l N_{p-1} \Delta} 
+(1-\ee^{-2\lambda_l \Delta})
\sum_{k=1}^{N_{p-1}} v_{k-1} \ee^{2\lambda_l k \Delta}
\biggr),
\\
&\EE[B_{i,l}B_{i',l}] +\EE[B_{i,l}C_{i',l}] +\EE[B_{i',l} C_{i,l}] 
\\
&=
\frac{(1-\ee^{-\lambda_l \Delta})^2}{2\lambda_l \gamma_l^{\alpha}}
(1 -\ee^{- 2\lambda_l \Delta}) \ee^{-\lambda_l(i+i'-2)\Delta}
\sum_{k = 1}^{i \land i'-1} v_{k-1} \ee^{2\lambda_l k \Delta} 
\\
&\quad
-\ind_{\{i \neq i'\}} 
\frac{v_{i'\land i -1} (1-\ee^{-\lambda_l \Delta})^2}{2\lambda_l\gamma_l^\alpha}
\ee^{-\lambda_l(|i-i'|-1)\Delta}(1+\ee^{- \lambda_l \Delta}),
\end{align*}
we obtain
\begin{align*}
&\cov[T_{i,j,k} X, T_{i',j',k'} X] 
\\
&= \sum_{l \in \mathbb N^2} \cov[(B_{i,l}+C_{i,l})(B_{i',l}+C_{i',l})] 
\delta_{j,k} e_l \delta_{j',k'} e_l
\\
&=
\begin{cases}
v_{i-1} F_{j',k'}^{j,k} +\frac{1}{2}(-v_{i-1} G_{2(i-1-N_{p-1}),j',k'}^{j,k}
+H_{2(i-1),N_{p-1}, j',k'}^{j,k}), 
& N_{p-1} +1 \le i = i' \le N_p,
\\
\frac{1}{2}(
H_{i+i'-2,i \land i'-1, j',k'}^{j,k}
-v_{i \land i' -1} (G_{|i-i'|-1,j',k'}^{j,k} +G_{|i-i'|,j',k'}^{j,k})), 
& i \neq i'.
\end{cases}
\end{align*}
Since we find from Lemma 4.10 in \cite{TKU2025a} and 
$H_{J,K,j',k'}^{j,k} = \OO(G_{J-2K,j',k'}^{j,k})$ that
\begin{align*}
F_{j',k'}^{j,k} &= 
\OO \biggl( \Delta^\alpha \Bigl(\Delta +\frac{1}{(|j-j'|+1)(|k-k'|+1)} \Bigr) \biggr)
\\
&\quad + \OO \biggl( \Delta^{1/2+\alpha} 
\Bigl( \ind_{\{ j \neq j'\}} \frac{1}{|j-j'|+1} 
+\ind_{\{ k \neq k'\}} \frac{1}{|k-k'|+1} \Bigr) \biggr),
\\
G_{J,j',k'}^{j,k} &= 
\OO \biggl( \frac{\Delta^\alpha}{J+1} 
\Bigl(\Delta +\frac{1}{(|j-j'|+1)(|k-k'|+1)} \Bigr) \biggr)
\\
&\quad + \OO \biggl( \frac{\Delta^{1/2+\alpha}}{J+1} 
\Bigl( \ind_{\{ j \neq j'\}} \frac{1}{|j-j'|+1} 
+\ind_{\{ k \neq k'\}} \frac{1}{|k-k'|+1} \Bigr) \biggr),
\\
H_{J,K,j',k'}^{j,k} &= 
\OO \biggl( \frac{\Delta^\alpha}{J-2K+1} 
\Bigl(\Delta +\frac{1}{(|j-j'|+1)(|k-k'|+1)} \Bigr) \biggr)
\\
&\quad + \OO \biggl( \frac{\Delta^{1/2+\alpha}}{J-2K+1} 
\Bigl( \ind_{\{ j \neq j'\}} \frac{1}{|j-j'|+1} 
+\ind_{\{ k \neq k'\}} \frac{1}{|k-k'|+1} \Bigr) \biggr),
\end{align*}
we obtain
\begin{align*}
&v_{i-1} F_{j',k'}^{j,k} +\frac{1}{2}(-v_{i-1} G_{2(i-1-N_{p-1}),j',k'}^{j,k}
+H_{2(i-1),N_{p-1}, j',k'}^{j,k})
\\
&= 
\OO \biggl( \Delta^\alpha 
\Bigl(\Delta +\frac{1}{(|j-j'|+1)(|k-k'|+1)} \Bigr) \biggr)
\\
&\quad + \OO \biggl( \Delta^{1/2+\alpha} 
\Bigl( \ind_{\{ j \neq j'\}} \frac{1}{|j-j'|+1} 
+\ind_{\{ k \neq k'\}} \frac{1}{|k-k'|+1} \Bigr) \biggr),
\end{align*}
\begin{align*}
&H_{i+i'-2,i \land i'-1, j',k'}^{j,k}
-v_{i \land i' -1} (G_{|i-i'|-1,j',k'}^{j,k} +G_{|i-i'|,j',k'}^{j,k})
\\
&= 
\OO \biggl( \frac{\Delta^\alpha}{|i-i'|+1} 
\Bigl(\Delta +\frac{1}{(|j-j'|+1)(|k-k'|+1)} \Bigr) \biggr)
\\
&\quad + \OO \biggl( \frac{\Delta^{1/2+\alpha}}{|i-i'|+1} 
\Bigl( \ind_{\{ j \neq j'\}} \frac{1}{|j-j'|+1} 
+\ind_{\{ k \neq k'\}} \frac{1}{|k-k'|+1} \Bigr) \biggr)
\end{align*}
and
\begin{align*}
\cov[T_{i,j,k}X, T_{i',j,k}X]
&= 
\OO \biggl( \frac{\Delta^\alpha}{|i-i'|+1} 
\Bigl(\Delta +\frac{1}{(|j-j'|+1)(|k-k'|+1)} \Bigr) \biggr)
\\
&\quad + \OO \biggl( \frac{\Delta^{1/2+\alpha}}{|i-i'|+1} 
\Bigl( \ind_{\{ j \neq j'\}} \frac{1}{|j-j'|+1} 
+\ind_{\{ k \neq k'\}} \frac{1}{|k-k'|+1} \Bigr) \biggr).
\end{align*}
In the same way as the proof of Lemma 4.11 in \cite{TKU2025a}, 
this result yields the rest of two results.
\end{proof}

\begin{proof}[\bf{Proof of Theorem \ref{app_th3}}]
We define
\begin{equation*}
\mathcal Z = 
\frac{1}{m N \Delta^{\alpha^*}} \sum_{k=1}^{m_2} \sum_{j=1}^{m_1} \sum_{i=1}^N 
(T_{i,j,k}X)^2,
\quad
\mathcal Z' = 
\frac{1}{m' N' (\Delta')^{\alpha^*}} \sum_{k=1}^{m_2'} \sum_{j=1}^{m_1'} \sum_{i=1}^N 
(T_{i,j,k}'X)^2,
\end{equation*}
and
\begin{equation*}
g_{r,\alpha}(\nu) = \frac{V c_\gamma^\alpha \psi_{r,\alpha}(\theta_2)}{(1-2b)^2}
\int_{[b,1-b]^2} \ee^{-\kappa^\TT y} \dd y.
\end{equation*}
By Proposition \ref{app_prop3} and 
a simple calculation similar to that in the proof of (6.2) in \cite{TKU2025arXiv1},
we have
\begin{align*}
\frac{1}{mN \Delta^\alpha} 
\sum_{k=1}^{m_2} \sum_{j=1}^{m_1} \sum_{i=1}^N \EE \bigl[ (T_{i,j,k} X)^2 \bigr]
&= 
V c_\gamma^\alpha \psi_{r,\alpha}(\theta_2)
\times \frac{1}{m} \sum_{k=1}^{m_2} \sum_{j=1}^{m_1} 
\ee^{-\kappa^\TT \overline y_{j,k}} 
+\OO(\Delta)
\\
&= g_{r,\alpha}(\nu) +\OO(\Delta).
\end{align*}
Since Lemma \ref{app_lem3} yields
\begin{align*}
&\EE \Biggl[ \biggl( \sum_{k=1}^{m_2} \sum_{j=1}^{m_1} \sum_{i=1}^N 
\bigl( (T_{i,j,k} X)^2 -\EE[(T_{i,j,k} X)^2] \bigr) \biggr)^2 \Biggr]
\\
&=
\sum_{k,k'=1}^{m_2} \sum_{j,j'=1}^{m_1} \sum_{i,i'=1}^N
\cov[(T_{i,j,k} X)^2, (T_{i',j',k'} X)^2]
\\
&= \OO(m N \Delta^{2\alpha}),
\end{align*}
we obtain
\begin{align*}
&\sqrt{m N} (\mathcal Z -g_{r,\alpha^*}(\nu))
\\
&=
\frac{1}{\sqrt{m N} \Delta^{\alpha^*}} \sum_{k=1}^{m_2} \sum_{j=1}^{m_1} \sum_{i=1}^N 
\bigl( (T_{i,j,k}X)^2 -\EE[(T_{i,j,k}X)^2] \bigr)
\\
&\qquad +
\sqrt{m N} \Biggl(
\frac{1}{m N \Delta^{\alpha^*}} \sum_{k=1}^{m_2} \sum_{j=1}^{m_1} \sum_{i=1}^N 
\EE[(T_{i,j,k}X)^2] -g_{r,\alpha^*}(\nu)
\Biggr)
\\
&= \Op(1) + \OO(\sqrt{m N} \Delta)
\\
&= \Op(1).
\end{align*}
In the same way as the proof of Theorem 3.1 in \cite{TKU2025arXiv1}, we have
\begin{equation*}
\sqrt{m N} ( \mathcal Z' -g_{r,\alpha^*}(\nu) ) = \Op(1)
\end{equation*}
and $\sqrt{m N}(\mathcal Z'/\mathcal Z -1) = \Op(1)$, which yields 
\begin{align*}
\sqrt{m N} (\widehat \alpha -\alpha^*)
= \frac{\sqrt{m N}}{\log(4)} \log \biggl( \frac{\mathcal Z'}{\mathcal Z} \biggr)
= \frac{\sqrt{m N}}{\log(4)} \Biggl(
\frac{\mathcal Z'}{\mathcal Z} -1 +\Op \biggl( \frac{1}{m N} \biggr)
\Biggr)
= \Op(1).
\end{align*}
This concludes the proof.
\end{proof}

\begin{proof}[\bf{Proof of Theorem \ref{app_th4}}]
By the mean value theorem, we have
\begin{equation*}
-\frac{\sqrt{m N}}{\log(N)} 
\pd_{\nu} K_{m,N} (\nu^*;\widehat \alpha)^\TT
=\int_0^1 \pd_{\nu}^2 
K_{m,N} (\nu^* +u(\widehat \nu -\nu^*);\widehat \alpha) \dd u 
\frac{\sqrt{m N}}{\log(N)} (\widehat \nu -\nu^*).
\end{equation*}
It follows that
\begin{align*}
\sup_{\nu \in \Xi} \bigl| K_{m,N}(\nu;\widehat \alpha) - K(\nu,\nu^*) \bigr|
&\le \sup_{\nu \in \Xi} 
\bigl| K_{m,N}(\nu;\widehat \alpha)- K_{m,N}(\nu;\alpha^*) \bigr|
\\
&\quad+ \sup_{\nu \in \Xi} 
\bigl| K_{m,N}(\nu;\alpha^*) -K(\nu,\nu^*) \bigr|
\\
&=: G_1 +H_1,
\\
\frac{\sqrt{m N}}{\log(N)} \pd_{\nu} K_{m,N}(\nu^*;\widehat \alpha)
&= \frac{\sqrt{m N}}{\log(N)} \bigl(
\pd_{\nu} K_{m,N}(\nu^*;\widehat \alpha)
- \pd_{\nu} K_{m,N}(\nu^*;\alpha^*) \bigr)
\\
&\quad+
\frac{1}{\log(N)} \times \sqrt{m N} \pd_{\nu} K_{m,N}(\nu^*;\alpha^*)
\\
&=: G_2 + H_2,
\end{align*}
and that for $\epsilon_{m,N} \downarrow 0$, 
\begin{align*}
\sup_{|\nu-\nu^*| \le \epsilon_{m,N}}
\bigl| \pd_{\nu}^2 K_{m,N}(\nu;\widehat \alpha)
- L(\nu^*;\alpha^*) \bigr| 
&\le 
\sup_{\nu \in \Xi}
\bigl| \pd_{\nu}^2 K_{m,N}(\nu;\widehat \alpha)
- \pd_{\nu}^2 K_{m,N}(\nu;\alpha^*) \bigr|
\\
&\quad+
\sup_{|\nu-\nu^*| \le \epsilon_{m,N}}
\bigl| \pd_{\nu}^2 K_{m,N}(\nu;\alpha^*)
- L(\nu^*;\alpha^*) \bigr| 
\\
&=: G_3 +H_3,
\end{align*}
where the function $\Xi \ni \nu \mapsto K(\nu,\nu^*)$ takes its unique minimum 
in $\nu = \nu^*$ and
$L(\nu^*;\alpha^*)$ is a positive definite $4$-dimensional matrix 
(see the proof of Theorem 2.2 in \cite{TKU2025a}).

We find from Lemma \ref{app_lem3} and the proof of Theorem 2.2 in \cite{TKU2025a} 
that $H_j = \op(1)$ as $m,N \to \infty$ for $j=1,2,3$.
We also find from Theorem \ref{app_th3} and
the proof of Theorem 4.1 in \cite{TKU2025arXiv1} that
\begin{equation*}
G_1 = \op(1),
\quad
G_2 = \Op(1),
\quad
G_3 = \op(1).
\end{equation*}
Hence, we obtain
\begin{equation*}
\widehat \nu \pto \nu^*,
\end{equation*}
\begin{equation*}
\frac{\sqrt{m N}}{\log(N)} \pd_{\nu} K_{m,N}(\nu^*;\widehat \alpha)^\TT = \Op(1),
\end{equation*}
\begin{equation*}
\int_0^1 \pd_{\nu}^2 
K_{m,N} (\nu^* +u(\widehat \nu -\nu^*);\widehat \alpha) \dd u
\pto L(\nu^*;\alpha^*)
\end{equation*}
and this concludes the proof.
\end{proof}

\end{document}